\documentclass[a4paper]{amsart}

\usepackage{amssymb} 
\usepackage{amsmath}
\usepackage{amsthm}
\usepackage{dsfont, enumerate}
\usepackage{color}

\usepackage{verbatim}
\usepackage{url}

\newtheorem{thm}{Theorem}[section]

\newtheorem{lem}[thm]{Lemma}
\newtheorem{cor}[thm]{Corollary}

\usepackage{caption}
\usepackage{array} 
\newcolumntype{M}[1]{>{\centering\arraybackslash}m{#1}} 

\newenvironment{proofbold}[1][\proofname]{\noindent{\textbf{Proof of #1.\ }}}{\hfill$\blacksquare$}
\usepackage{sagetex}
\newenvironment{mysage}{\sagesilent}{\endsagesilent}

\def\1{\mathds{1}}
\let\ve=\varepsilon

\def\mod{\mathbin{\,\textrm{mod}\,}}

\title{On a M\"obius double sum%
}
\author{Olivier Ramar\'e, Sebastian Zuniga-Alterman}

\newcommand{\Addresses}{{
  {\footnotesize
\ \\
  O.~RAMAR\'E - \textsc{CNRS / Institut de Math\'ematiques de Marseille, Aix Marseille Universit\'e,
U.M.R. 7373, Campus de Luminy, Case 907, 13288 Marseille Cedex 9, FRANCE.}\\
  \texttt{olivier.ramare@univ-amu.fr}

\ \\
  S.~ZUNIGA-ALTERMAN - \textsc{School of Science, The University of New South Wales Canberra, Northcott Drive, Campbell ACT 2612, AUSTRALIA.}\\
  \texttt{s.zuniga$\_$alterman@unsw.edu.au}

}}}

\thanks{The authors thank A. de Camargo and A. Simonič for independently verifying several computations appearing in this work and for helpful comments.}
\thanks{The first author has been partly supported by the joint
  FWF-ANR project Arithrand: FWF: I 4945-N and ANR-20-CE91-0006ANR. The second author has been supported by ARC Grant.}

\begin{document}

\begin{mysage}

######################################################
####### MANAGING NUMBER REPRESENTATION ###########
######################################################

def RoundUp(x,d):
    return float(ceil(x*10^d)/10^d)
    
def RoundDown(x,d):
    return float(floor(x*10^d)/10^d)

def Upper(x,d): 
    return RoundUp(RIF(x).upper(),d)
    
def Lower(x,d):
    return RoundDown(RIF(x).lower(),d) 
    
def Trunc(x,d):
     if x>=0:
          return float(floor(x*10^d)/10^d)
     return float(ceil(x*10^d)/10^d)

######################################################
############### GENERAL CONSTANTS ##################
######################################################

digits = int(4)

E1 = 1044/10^3 
E2 = 232/10^3

######################################################
#################### LEMMA 2.5 ########################
############################################# {m4} ####

H1 = 10032/10^6
H2 = 296/10^4

trop = H2/H1

######################################################
#################### LEMMA 4.3 ########################
###################################### {harmonicmu2} ###

Harm1_a = 578/10^3
Harm2_a = 1166/10^3

Harm1 = 1040/10^3
Harm2 = 1048/10^3

######################################################
#################### LEMMA 4.4 ######################## 
########################################### {Dusart1} ###
     
def B(k,M):
    C=5761455
    if k>1 and M>10^8-1:
        return RIF(k/(k-1)/M^(k-1)/log(M)*(1+1/log(M)+253816/10^5/log(M)^2)-C/M^k)
    return False
    
def C(k,M):
    D=99987730
    if k>1 and M>10^8-1:
        return RIF(k/(k-1)/M^(k-1)*(1+2/10/log(M)^2)-D/M^k)
    return False

######################################################
#################### LEMMA 4.5 ########################
############################################# {Aux1} ###

delt1 = 1/3
delt2 = 5/12
delt3 = 4/9

xi = RIF(1-1/12/log(10))

def A(d):
    return RIF(max(euler_gamma,1/d/exp(euler_gamma*d+1)))

g0_2 = RIF(sqrt(3)*(sqrt(2)-1)/2)
g2_2 = RIF(H2/H1 * (1-1/2^xi))

fx1_2 = RIF((2-1)/(sqrt(2)-1)^2)
Fx1_2 = RIF(1-(1-fx1_2)/2-fx1_2/2^2)
Fx1_2_delta2 = RIF(1+abs(1-fx1_2)/2^(1-delt2)+abs(fx1_2)/2^(2-2*delt2))

fx2_2 = RIF(g0_2*(2-1)/(sqrt(2)-1)^2)
Fx2_2 = RIF(1-(1-fx2_2)/2-fx2_2/2^2)
Fx2_2_delta2 = RIF(1+abs(1-fx2_2)/2^(1-delt2)+abs(fx2_2)/2^(2-2*delt2))

f1_2 = RIF(g0_2^2*(2-1)/(sqrt(2)-1)^2)
F1_2 = RIF(1-(1-f1_2)/2-f1_2/2^2)
F1_2_delta1 = RIF(1+abs(1-f1_2)/2^(1-delt1)+abs(f1_2)/2^(2-2*delt1))
F1_2_delta2 = RIF(1+abs(1-f1_2)/2^(1-delt2)+abs(f1_2)/2^(2-2*delt2))
F1_2_delta3 = RIF(1+abs(1-f1_2)/2^(1-delt3)+abs(f1_2)/2^(2-2*delt3))

M = 10^8  

P1_M =  290950563/10^8 #PRECISION 10^8
P1_delta1_M = 62357582/10^6
P1_delta2_M = 3061036494/10^7 #PRECISION 10^8
P1_delta3_M = 70785717/10^5

def Cm(d):
    return RIF((2+M^(d-1/2)+M^(d-1))/(1-M^(-1/2)))

T = 4*10^9

vx1 = Fx1_2*P1_M*exp(2/(1-M^(-1/2))*B(3/2,M))+(2-delt2)/(1-delt2)*A(delt2)*Fx1_2_delta2*P1_delta2_M*exp(Cm(delt2)*B(3/2-delt2,M))*T^(-delt2)
vx2 = Fx2_2*P1_M*exp(2/(1-M^(-1/2))*B(3/2,M))+(2-delt2)/(1-delt2)*A(delt2)*Fx2_2_delta2*P1_delta2_M*exp(Cm(delt2)*B(3/2-delt2,M))*T^(-delt2)
v1 = F1_2*P1_M*exp(2/(1-M^(-1/2))*B(3/2,M))+(2-delt2)/(1-delt2)*A(delt2)*F1_2_delta2*P1_delta2_M*exp(Cm(delt2)*B(3/2-delt2,M))*T^(-delt2) #delta=5/12 seems the better choice

Fmx1 = 62359917454422/10^(13)
Fmx2 = 316843122347233/10^(14)
mx1 = 206803754617859/10^(14)

THx1 = max(vx1,Fmx1)
THx2 = max(vx2,Fmx2)
TH1 = max(v1,mx1)

######################################################
#################### LEMMA 4.6 ########################
############################################# {Aux2} ###

delt = 1/3 

f2_2 = g0_2*g2_2*(1+(2^(xi-1/2)+1)/(2^xi-1))
F2_2 = RIF(1-(1-f2_2)/2-f2_2/2^2)
F2_2_delta = RIF(1+abs(1-f2_2)/2^(1-delt)+abs(f2_2)/2^(2-2*delt))

P2 = 1638009774/10^9 #PRECISION 10^8
P2_delta = 1381614454/10^8 #PRECISION 10^8

cm0_2 = RIF((1+M^(1/2-xi))/(1-M^(-xi)))
cmd_2 = RIF((1+M^(-1/6)+M^(1/2-xi)+M^(-2/3))/(1-M^(-xi)))

T2 = 10^7

v2 = F2_2*P2*exp(cm0_2*B(3/2,M))+5/2*A(delt)*F2_2_delta*P2_delta*exp(cmd_2*B(7/6,M))*T2^-delt
mx2 = 190380793763037/10^(14)

TH2 = max(v2,mx2)

######################################################
################## COROLLARY 4.7 #####################
############################################# {Cor1} ###

L1 = 2*TH2/2/log(10)/12+TH2/log(10)/12

######################################################
#################### LEMMA 4.8 ########################
############################################# {Aux3} ###

f3_2 = RIF(g2_2^2*(2-1)*2^(2*xi-1)/(2^xi-1)^2)
F3_2 = RIF(1-(1-f3_2)/2-f3_2/2^2)
F3_2_spec = RIF(1+abs(f3_2-1)/(sqrt(2)-1))  

W = RIF((sqrt(2)-1)/(sqrt(2)-1+abs(f3_2-1))*(E1+abs(f3_2-1)*E2/(sqrt(2)-1)))

P3 = 1031648197/10^9 #PRECISION 10^8
P3_spec = 3191814808/10^9 #PRECISION 10^8

cm0_3 = RIF((2+M^(xi-2))/(1-M^(-xi))^2)
cmd_3 = RIF(2/(1-M^(-xi))^2/(1-M^(-1/2)))

T3 = 10^6

v3 = F3_2*P3+3*W*P3_spec*T3^(-1/2)
mx3 = 267710025/10^8
TH3 = max(v3,mx3)

######################################################
################## COROLLARY 4.9 #####################
############################################# {Cor2} ###

L2 = TH3*(1/12/log(10)-1/12^2/log(10)^2)+TH3/12^2/log(10)^2

######################################################
#################### LEMMA 4.10 ########################
########################################### {auxnu2} ###

fxxx1_2 = RIF((sqrt(2)+1)/2)
fx1_2 = RIF(g0_2 * (sqrt(2)+1)/2)

F1xxx_2 = RIF(1-(1-fxxx1_2*sqrt(2))/2-fxxx1_2*sqrt(2)/2^2)
F1xxx_2_delta = RIF(1+abs(1-fxxx1_2*sqrt(2))/2^(1-delt)+abs(fxxx1_2*sqrt(2))/2^(2-2*delt))

F1x_2 = RIF(1-(1-fx1_2*sqrt(2))/2-fx1_2*sqrt(2)/2^2)
F1x_2_delta = RIF(1+abs(1-fx1_2*sqrt(2))/2^(1-delt)+abs(fx1_2*sqrt(2))/2^(2-2*delt))

P1x = 12159031/10^7 #PRECISION 10^8
P1x_delta = 726497961/10^8 #PRECISION 10^8

vxxx_aux11 = 2*F1xxx_2*P1x*exp(B(3/2,M))+4*A(delt)*F1xxx_2_delta*P1x_delta*exp((1+M^(-1/6)+M^(-2/3))*B(7/6,M))/(5*10^6)^delt
mxxx_aux11 = 22526506709/10^10 #CHECKED
THxxx_aux11 = max(vxxx_aux11,mxxx_aux11)

vx_aux11 = 2*F1x_2*P1x*exp(B(3/2,M))+4*A(delt)*F1x_2_delta*P1x_delta*exp((1+M^(-1/6)+M^(-2/3))*B(7/6,M))/(5*10^6)^delt
mx_aux11 = 1587160669/10^9 #CALCULATED UP TO 5·10^6
THx_aux11 = max(vx_aux11,mx_aux11)

######################################################
#################### LEMMA 4.11 #######################
############################################# {auxf2} ###

f2_lx = g2_2 * (2-1)/2^2*2^xi/(2^xi-1)
W2_lx = (sqrt(2)-1)/(sqrt(2)-1+abs(2*f2_lx-1))*(E1+E2*abs(2*f2_lx-1)/(sqrt(2)-1))

F2_LX = 1-(1-2*f2_lx)/2-f2_lx/2
F2P1_LX = 1+abs(2*f2_lx-1)/(sqrt(2)-1)

P0_LX = 694045937/10^9 #PRECISION 10^8
P1_LX = 1079535589/10^9 #PRECISION 10^8

Sumx = euler_gamma+log(2)/(f2_lx+1)+656118309/10^9  #PRECISION 10^8

N=2*10^8

lc1 = F2_LX*P0_LX
uc1 = F2_LX*P0_LX*exp((1-M^(-xi))^(-1)*B(1+xi,M))

logmx = 129424331261228/10^(14)
logmm = 505508801388535/10^(15)

######################################################
################## COROLLARY 4.12 ####################
########################################### {auxnu1} ###

loglemma1_1 = uc1
loglemma1_2 = logmx/12/log(10)

######################################################
#################### LEMMA 4.13 #######################
############################################ {auxf22} ###

Af2_lx = g2_2^2 * (2-1)/2^2*2^(2*xi)/(2^xi-1)^2

AF2_LX = 1-(1-2*Af2_lx)/2-Af2_lx/2
AF2P1_LX = 1+abs(2*Af2_lx-1)/(sqrt(2)-1)

AP0_LX = 1031648197/10^9 #PRECISION 10^8
AP1_LX = 3191814808/10^9 #PRECISION 10^8

ASumx = log(2)/(Af2_lx+1)+3611216153/10^(10)+euler_gamma #PRECISION 10^8

lc2 = AF2_LX*AP0_LX
uc2 = AF2_LX*AP0_LX*exp((2+M^(xi-2))*(1-M^(-xi))^(-2)*B(1+xi,M))

log2mx = 129424331261228/10^(14)
log2mm = 505508801388535/10^(15)

######################################################
################## COROLLARY 4.14 ####################
############################################ {auxnu} ###

loglemma2 = uc2/12/log(10)+log2mx/(12*log(10))^2

######################################################
################# LEMMA 5.2 ###########################
############################################ {GetGq} ###

E1 = 1044/10^3
E2 = 232/10^3

f2 = 1/4
W2 = (sqrt(2)-1)/(sqrt(2)-1+abs(2*f2-1))*(E1+E2*abs(2*f2-1)/(sqrt(2)-1))

APROD = 40282372/10^8 #PRECISION 10^8

ASUM = euler_gamma+2046752376/10^9 #PRECISION 10^8

PPdelta = 494778677/10^8 #PRECISION 10^8

APROD_delta = PPdelta*exp(B(3/2,M)/(1-M^(-1/2)))

G1 = W2*APROD_delta
G2 = E2*APROD_delta
       
######################################################
################## THEOREM 1.1 #######################
############################################# {Main} ###

S11_80 = 333415398793/10^(12)  #CHECKED 
S12_80 = 4134625411/10^5 #CHECKED

Z_80 = 809999/10^4

Thresh33_1 = 2*TH1/Z_80+2*sqrt(2)*H1*L1/sqrt(Z_80)+H1^2*L2
Thresh33_2 = APROD*(S11_80+log(2))+(G2*S12_80+G1*(1+sqrt(2)))/sqrt(10^(33))

Bound0 = 445/10^3
Bound1 = 65924939690269/10^(14) #CHECKED 
Bound2 = 4987002674334/10^(13) #CHECKED 
Bound3 = 4669804238966/10^(13) #CHECKED 
Bound4 = Thresh33_1+Thresh33_2

######################################################
################# LEMMA 6.2 ###########################
############################################ {A_est} ###

ax_1 = 1/30*(6/pi^2*log(10.9/5)+Harm2-Harm1)

u1 = 2323/30030
u2= 57731/570570
ax_2 = max(u1,u2)

rsubs = RIF(Upper(ax_2,digits))
subs = 1- 1012/10^4 #Needs to be changed if digits differs from 4

######################################################
################# LEMMA 6.3 ##########################
############################################# {B_est} ###

c = 1/1000
 
value1 = 91/10

######################################################
################## THEOREM 1.2 #######################
########################################### {Mainbis} ###

def ind1(t):
    return Integer(t > Integer(10)**12)
    
def ind2(t):
    return Integer(RIF(t) >= RIF(QQ(109)/10) * Integer(10)**8)

def Grr(ee,t):
    return RIF(2*ee-ind2(t)*value1*ee^2/exp(2*ee*euler_gamma))

def mn1(ee, tt):
    ee = RIF(ee)
    tt = RIF(tt)
    return min(ee, 1/log(tt))
 
def mn2(ee, tt):
    ee = RIF(ee)
    tt = RIF(tt)
    return min(2*ee, 1/log(tt))
    
def Theta(T,T0):
    if T==422 and T0==10^8:
        return Bound0
    if T==10^8 and T0==10^(12):
        return Bound1
    if T==10^(12) and T0==10^(16):
        return Bound2
    if T==10^(16) and T0==10^(33):
        return Bound3 
    if T==10^(33) and T0==10^(50):
        return Bound4
    return False      

def Omega1(R,T,T0):
    p1=Theta(T,T0)+3/100*loglemma1_1*1/(1+12*log(10))*ind1(T0)+Grr(1/R,T)
    p2=mn2(1/R,T)*(3/100*loglemma1_1*ind1(T0)+(ax_2*THx_aux11*41/10/sqrt(109/10)+THxxx_aux11*5/2/sqrt(109/10)+subs*THx_aux11*41/10/sqrt(47)+THxxx_aux11*5/2/sqrt(47))*2^(1/R)+(ax_2*THxxx_aux11/2/sqrt(109/10)+subs*THxxx_aux11/2/sqrt(47))*2^(2/R)/R)
    p3 = mn1(1/R,T)*2*(RIF(3)/RIF(100)*RIF(loglemma1_2) + RIF(ax_1))
    p4=mn2(1/R,T)*1/2/R*((1+1/c)*(3/100)^2*loglemma2+((25/4*THx1+5*41/10*THx2+(41/10)^2*TH1)*(c+1)/109*10)*2^(2/R)+((10/4*THx1+41/10*THx2)*(c+1)/109*10)*2^(3/R)/R+(THx1*(c+1)/4/109*10)*2^(4/R)/R^2)
    return RIF(p1+p2+p3+p4)

def exx1(T,R):
    return RIF(1/exp(log(T)/R))
    
def exx2(T,R):
    return RIF(1/exp(2*log(T)/R))

def Omega2(R,S,T,T0):
    q1=exx2(T,R)*Theta(T,T0)+3/100*loglemma1_1*1/(1+12*log(10))*ind1(T0)+Grr(1/S,T)
    q2 = 1/S*exx1(T,R)*2*(RIF(3)/RIF(100)*RIF(loglemma1_2) + RIF(ax_1)) #CORRECTED
    q3=2/S*exx2(T,R)*(3/100*loglemma1_1*ind1(T0)+(ax_2*THx_aux11*41/10/sqrt(109/10)+THxxx_aux11*5/2/sqrt(109/10)+subs*THx_aux11*41/10/sqrt(47)+THxxx_aux11*5/2/sqrt(47))*2^(1/S)+(ax_2*THxxx_aux11/2/sqrt(109/10)+subs*THxxx_aux11/2/sqrt(47))*2^(2/S)/S)
    q4=1/S^2*exx2(T,R)*((1+1/c)*(3/100)^2*loglemma2+((25/4*THx1+5*41/10*THx2+(41/10)^2*TH1)*(c+1)/109*10)*2^(2/S)+((10/4*THx1+41/10*THx2)*(c+1)/109*10)*2^(3/S)/S+(THx1*(c+1)/4/109*10)*2^(4/S)/S^2) #CORRECTED
    return RIF(q1+q2+q3+q4)

R01 = 43
S01 = 25
R02 = 82
S02 = 50
R03 = 158
S03 = 100

R11 = 94
S11 = 25
R12 = 179
S12 = 50
R13 = 347
S13 = 100

R21 = 112
S21 = 25
R22 = 204
S22 = 50
R23 = 388
S23 = 100

R31 = 133
S31 = 25
R32 = 243
S32 = 50
R33 = 461
S33 = 100

R41 = 242
S41 = 25
R42 = 440
S42 = 50
R43 = 836
S43 = 100

def CHECK(N,S,T,T0):
    R=S+1
    M=100
    Rm=0
    while R<=N:
        print(R,S,Omega1(R,T,T0),Omega2(R,S,T,T0))
        if abs(Omega1(R,T,T0)-Omega2(R,S,T,T0))<M:
            M=abs(Omega1(R,T,T0)-Omega2(R,S,T,T0))
            Rm=R
        R=R+1
    return print('(S,T,T0)=',S,T,T0,'Min at',Rm)

#CHECK(300,25,422,10^8)
#CHECK(400,50,422,10^8)
#CHECK(500,100,422,10^8)

#CHECK(300,25,10^8,10^(12))
#CHECK(400,50,10^8,10^(12))
#CHECK(500,100,10^8,10^(12))

#CHECK(300,25,10^(12),10^(16))
#CHECK(400,50,10^(12),10^(16))
#CHECK(500,100,10^(12),10^(16))

#CHECK(300,25,10^(16),10^(33))
#CHECK(400,50,10^(16),10^(33))
#CHECK(500,100,10^(16),10^(33))

#CHECK(300,25,10^(33),10^(50))
#CHECK(600,50,10^(33),10^(50))
#CHECK(1000,100,10^(33),10^(50))

#CHECK(500,44,10^(12),10^(16)) #A. Simonic 
#CHECK(500,47,10^(12),10^(16)) #A. Simonic 
#CHECK(500,74,10^(12),10^(16)) #A. Simonic 
#CHECK(1000,124,10^(12),10^(16)) #A. Simonic 
#CHECK(2000,271,10^(12),10^(16)) #A. Simonic 

S90 = 44
R90 = 182
S91 = 47
R91 = 193
S92 = 74
R92 = 293
S93 = 124
R93 = 477
S94 = 271
R94 = 1017

\end{mysage}

\maketitle

\begin{abstract} We study the double sum $S_\ve(X)$$=$$\sum_{\substack{d,e\le
        X}}\frac{\mu(d)\mu(e)}{[d,e]^{1+\ve}}$, which converges even in the case $\ve=0$, where $\mu$ denotes the M\"obius function and $[d,e]$ is the least common multiple of $d$ and $e$. 
        
        Such expressions arise naturally in analytic number theory, notably as the diagonal contribution in certain squared mean values, and they play a significant role in zero-density estimates for the Riemann zeta function and related $L$-functions. We establish uniform upper bounds for $S_\ve(X)$ across various ranges of $X$, with particular emphasis on the case $\ve$ close to $0^+$.
 \end{abstract}

\section{Introduction and results}

Zero-density estimates for the Riemann zeta function and, more generally, for Dirichlet $L$-functions, are obtained by combining approximate functional equations with mean-value estimates for Dirichlet polynomials. In the classical approach, upper bounds for 
\begin{equation*}
N(\sigma,T)=|\{\rho=\beta+i\gamma:\zeta(\rho)=0, \beta>\sigma, |\gamma|\leq T\}|
\end{equation*}
are obtained from estimates involving mollified moments of $\zeta(s)$ in the half plane $\{\sigma>1/2\}$. The strength of such bounds depends on the mollifier itself and, in particular, on the evaluation of its second moment.

We study the sum
\begin{equation*}
S_\ve(X)=\sum_{\substack{d,e\le
        X}}\frac{\mu(d)\mu(e)}{[d,e]^{1+\ve}},\quad\ve\geq 0,
 \end{equation*}
 where $\mu$ is the M\"obius function and $[d,e]$ denotes the least common multiple between $d$ and $e$. This sum appears in several problems in analytic number theory and plays an important role in zero-density arguments. Indeed, it captures the diagonal contribution for some mollified second moments and thus directly participates in the efficiency of explicit density estimates.
 
We are especially interested in $S_0(X)$. 
It has been shown by F.~Dress,
G.~Tenenbaum and H.~Iwaniec in \cite{Dress-Iwaniec-Tenenbaum*83}, that the sum $S_0(X)$
converges to a constant. Later, H.~Helfgott computed precisely
its first four decimals in
\cite[Prop. 6.30]{Helfgott*30}. This implies in particular
 that this constant is smaller than $0.4408$. 
R. de la Bret\`eche, F. Dress and G. Tenenbaum dealt with a similar sum in \cite{Breteche-Dress-Tenenbaum*20}.

Our aim in this article is to provide uniform, practical explicit
upper bounds for $S_0(X)$ and its perturbation $S_\ve(X)$ when $\ve$ is close to $0^+$, namely $\ve\in[0,1/25]$. We present the following two results.


\begin{thm} 
  \label{Main}
  Let $X>0$. Then
 \begin{equation*}
    0\le S_0(X)= \sum_{\substack{d,e\le
        X}}\frac{\mu(d)\mu(e)}{[d,e]}\leq
    \begin{cases}
     \mbox{$0.445$}&\text{when $422\leq X< 10^{8}$},\\
         \mbox{$\sage{Upper(Bound1,digits)}$}&\text{when $10^8\leq X< 10^{12}$},\\
      \mbox{$\sage{Upper(Bound2,digits)}$}&\text{when $10^{12}\leq X< 10^{16}$},\\
       \mbox{$\sage{Upper(Bound3,digits)}$}&\text{when $10^{16}\leq X< 10^{33}$},\\
        \mbox{$\sage{Upper(Thresh33_1+Thresh33_2,digits)}$}&\text{when $X\geq 10^{33}$}.
    \end{cases}
  \end{equation*}
\end{thm}

\begin{thm}
  \label{Mainbis}
  When $0\leq\ve\leq 1/25$, we have
  \begin{equation*}
    0\le S_\ve(X)=\sum_{\substack{d,e\le
        X}}\frac{\mu(d)\mu(e)}{[d,e]^{1+\ve}}
    \le
    \begin{cases} \mbox{$\sage{max(Upper(Omega1(R03,422,10^8),digits),Upper(Omega2(R03,S03,422,10^8),digits))}$} &\mbox{if $422\le X<10^{8}$ and $\ve\le 1/\sage{S03}$},\\
    \mbox{$\sage{max(Upper(Omega1(R13,10^8,10^(12)),digits),Upper(Omega2(R13,S13,10^8,10^(12)),digits))}$}&\mbox{if $10^8\le X<10^{12}$ and $\ve\le 1/\sage{S13}$},\\
       \mbox{$\sage{max(Upper(Omega1(R21,10^(12),10^(16)),digits),Upper(Omega2(R21,S21,10^(12),10^(16)),digits))}$} &\mbox{if $10^{12}\le X<10^{16}$ and $\ve\le 1/\sage{S21}$},\\
      \mbox{$\sage{max(Upper(Omega1(R22,10^(12),10^(16)),digits),Upper(Omega2(R22,S22,10^(12),10^(16)),digits))}$}&\mbox{if $10^{12}\le X<10^{16}$ and $\ve\le 1/\sage{S22}$},\\
            \mbox{$\sage{max(Upper(Omega1(R32,10^(16),10^(33)),digits),Upper(Omega2(R32,S32,10^(16),10^(33)),digits))}$}&\mbox{if $10^{16}\leq X< 10^{33} $ and $\ve\le 1/\sage{S32}$},\\
      \mbox{$ \sage{max(Upper(Omega1(R33,10^(16),10^(33)),digits),Upper(Omega2(R33,S33,10^(16),10^(33)),digits))}$}&\mbox{if $10^{16}\leq X<10^{33}$ and $\ve\le 1/\sage{S33}$},\\
        \mbox{$\sage{max(Upper(Omega1(R41,10^(33),10^(50)),digits),Upper(Omega2(R41,S41,10^(33),10^(50)),digits))}$}&\mbox{if $X\ge 10^{33} $ and $\ve\le 1/\sage{S41}$},\\
      \mbox{$ \sage{max(Upper(Omega1(R42,10^(33),10^(50)),digits),Upper(Omega2(R42,S42,10^(33),10^(50)),digits))}$}&\mbox{if $X\ge 10^{33} $ and $\ve\le 1/\sage{S42}$},\\
       \mbox{$\sage{max(Upper(Omega1(R43,10^(33),10^(50)),digits),Upper(Omega2(R43,S43,10^(33),10^(50)),digits))}$}&\mbox{if $X\ge 10^{33}$ and $\ve\le 1/\sage{S43}$}.
    \end{cases}
  \end{equation*}
  In fact, we can provide uniform upper bounds for $S_\ve(X)$ with $\ve\in[0,1/10]$. 

\end{thm}

These results were long out of reach, but recent work of H. Helfgott and the authors has helped to overcome the main obstacles.

It is noteworthy that Theorem 1.2 is only slightly weaker than Theorem 1.1 for the specific ranges of $X$. It may even be the case that $S_\ve(X)\leq S_0(X)$
for all $X>0$, which would make Theorem \ref{Mainbis} an immediate consequence of Theorem \ref{Main}; however, we were unable to establish this.

\noindent\phantom{x}\textbf{Notation.}
Our notation is classical: we use
\begin{equation}
  \label{eq:4}
  m_q(X;1+\ve)=\sum_{\substack{\ell\le X\\ (\ell,q)=1}}\frac{\mu(\ell)}{\ell^{1+\ve}},\quad m_q(X)=m_q(X;1),
  \quad
  \text{and}
  \quad
  m(X)=m_1(X).
\end{equation}
Furthermore, for two functions $f,g:\mathbb{C}\to\mathbb{R}_{>0}$, $f=O^*(g)$ means that $|f|\le g$.

We consider the {\em Euler $\varphi_s$ function}: let $s$ be any complex number, we define $\varphi_s:\mathbb{Z}_{>0}\to\mathbb{C}$ as $q\mapsto q^s\prod_{p|q}(1-\frac{1}{p^s})$. We also consider the generalised sum of divisors function $\sigma_s:\mathbb{Z}_{>0}\to\mathbb{C}$ as $q\mapsto q^s\prod_{p|q}(1+\frac{1}{p^s})$

On the other hand, for $\xi=1-1/\log(10^{12})$, we consider the following two functions 
   \begin{align} \label{defg0g2}
  g_0(q)=\prod_{2|q}\frac{\sqrt{3}(\sqrt{2}-1)}{2},\quad\qquad g_2(q)=\prod_{2|q}2.9506\bigg(1-\frac{1}{2^\xi}\bigg).
  \end{align}
  Finally, $\gamma$ denotes Euler's constant $0.5772156649...$ .
  
 \noindent\phantom{x}\textbf{Organisation.} The article is structured as follows. In $\S$\ref{mq}, we review classical and recent estimates for the function $m_q(X)$. $\S$\ref{mq_ve} establishes an estimate for $m_q(X;1+\ve)$, providing the main tool needed to treat the perturbed case $\ve>0$ of $S_\ve(X)$. In $\S$\ref{meanvalues}, we develop the analytic framework required for the proofs of our main results; this includes general mean-value estimates for multiplicative functions, bounds for Euler products, and several auxiliary lemmas. The proofs of these auxiliary results are deferred to $\S$\ref{auxiliaries}.
$\S$\ref{S_Main} is devoted to the proof of Theorem \ref{Main}.
$\S$\ref{S_Mainbis} contains the proof of Theorem \ref{Mainbis}; there we work in a slightly more general setting than that stated in the theorem, allowing flexibility in the choice of $\ve$.
 

\section{On the summatory function $m_q(X;1+\ve)$ at $\ve=0$}\label{mq}

In \cite[Lemma 10.2]{Granville-Ramare*96}, the following handy estimate is proved.
\begin{lem}\label{BOUND} For any $X>0$ and $q\in\mathbb{N}$,
\begin{equation*}|m_q(X)|\le 1.
\end{equation*}
\end{lem}
By \cite[Lemma 5.10]{Helfgott*30} and
\cite[Eq. (5.79)]{Helfgott*30} and then \cite[Cor. 1.2]{Ramare-Zuniga*25-1} and \cite[Lemma 5.17]{Helfgott*30}, we may in fact derive bounds that exhibit convergence.
\begin{lem}
  \label{m1}
  We have 
\begin{align*}
\phantom{xxxxxxxxxxxx}|m(X)|&\le \sqrt{\frac{2}{X}}&&\text{ for }\quad 0<X\le 10^{14},\phantom{xxxxxxxxxxxx}\\
|m_2(X)|&\le \sqrt{\frac{3}{X}}&&\text{ for }\quad 0<X\le 10^{12}.
\end{align*}
\end{lem}

\begin{lem}
  \label{m2}
  \begin{align*}
\phantom{xxxxxxxxxxxx}|m(X)|&\le \frac{\sage{Trunc(H1,6)}}{\log X}&&\text{ for }\quad X \ge 617\,990,\phantom{xxxxxxxxxxxx}\\
|m_2(X)|&\le \frac{\sage{Trunc(H2,4)}}{\log X}&&\text{ for }\quad X\ge 5379.
\end{align*}
\end{lem}

By combining lemmas \ref{m1} and \ref{m2} and using that the function $t\geq 10^{12}\mapsto t^{1-\xi}/\log t$ is increasing, we arrive at
 
\begin{lem}
  \label{m3}
   Let $y\ge 1$ and any $t\in(0,y]$. Then
  \begin{align}\label{e1}
    |m(t)|&\le \sqrt{\frac{2}{t}}
    +\sage{Trunc(H1,6)}\cdot\frac{y^{1-\xi}}{\log y}\frac{\1_{y\ge 10^{12}}}{t^{1-\xi}},\\
    \label{e2}
    |m_2(t)|&\le \sqrt{\frac{3}{t}}
    +\sage{Trunc(H2,4)}\cdot\frac{y^{1-\xi}}{\log y}\frac{\1_{y\ge 10^{12}}}{t^{1-\xi}}
  \end{align}
\end{lem}

By using the following identity, established for example in \cite[Eq. (5.73)]{Helfgott*30}, one can study a sum with coprimality conditions from the same sum without such conditions.  

 \begin{lem}\label{coprime}
 We have the identity $ \sum_{d|q^{\infty},d|\ell}\mu\left(\frac{\ell}{d}\right)=\mu(\ell)\mathds{1}_{\{(\ell,q)=1\}}(\ell)$.
Hence, for any function $h:\mathbb{Z}_{>0}\to\mathbb{C}$, we have the formal identity
 \begin{equation*}
 \sum_{\substack{\ell\\(\ell,q)=1}}\frac{\mu(\ell)}{\ell}h(\ell)=\sum_{d|q^\infty}\frac{1}{d}\sum_{\ell}\frac{\mu(\ell)}{\ell}h(d\ell).
 \end{equation*}
 \end{lem} 
 
We can now combine estimates \eqref{e1} and \eqref{e2}, distinguish the cases $(2,q)=1$ or $2|q$ and use Lemma \ref{coprime} to obtain the following estimation subject to coprimality conditions.
 
\begin{lem}
  \label{m4}
  Let $X>0$ and $q\in\mathbb{Z}_{>0}$. We have
  \begin{equation*}
    |m_q(X)|\le \frac{g_0(q)\sqrt{q}}{\varphi_{\frac{1}{2}}(q)}\sqrt{\frac{2}{X}}
    +\sage{Trunc(H1,6)}\cdot\frac{g_2(q)q^\xi}{\varphi_{\xi}(q)}\frac{\1_{X\ge 10^{12}}}{\log X},
  \end{equation*}
  where $g_0$ and $g_2$ are defined in \eqref{defg0g2}.
\end{lem}

Note that we have diminished the threshold from $10^{14}$ to $10^{12}$ in the first estimate so that both estimates have similar shapes. Moreover, when $q=2q'$, 
\begin{equation*}
\frac{\sqrt{2}g_0(q)\sqrt{q}}{\varphi_{\frac{1}{2}}(q)}=\frac{2g_0(2)}{(\sqrt{2}-1)}\frac{\sqrt{q'}}{\varphi_{\frac{1}{2}}(q')}=\frac{\sqrt{3}\sqrt{q'}}{\varphi_{\frac{1}{2}}(q')},
\end{equation*}
and 
\begin{equation*}
\sage{Trunc(H1,6)}\cdot\frac{g_2(q)q^\xi}{\varphi_{\xi}(q)}=\sage{Trunc(H1,6)}\cdot\frac{g_2(2)2^\xi}{\varphi_{\xi}(2)}\frac{q'^\xi}{\varphi_{\xi}(q')}\geq\sage{Trunc(H2,4)}\cdot\frac{q'^\xi}{\varphi_{\xi}(q')}.
\end{equation*}
The above analysis explains the definitions of $g_0$ and $g_2$ in \eqref{defg0g2}, where the value $\sage{Upper(H2/H1,digits)}$ is an upper bound for $\sage{Trunc(H2,4)}/\sage{Trunc(H1,6)}$.


\section{On the summatory function of the M\"obius function at
  $s=1+\ve$}\label{mq_ve}

By simply replacing \cite[Lemma 5.2]{Ramare-Zuniga*23-1} by Lemma \ref{m4}, thus simply changing the function $g_1$ defined therein by the function $g_2$ defined in \eqref{defg0g2}, we obtain the following direct modification of  \cite[Theorem~1.7]{Ramare-Zuniga*23-1}.

\begin{lem}
  \label{m1e}
  Let $\sigma\in[1,2]$ and $X\geq 1$. We have the following estimation
  \begin{equation*}
    m_q(X;1+\ve)
    =\frac{m_q(X)}{X^{\ve}}
    +\frac{q^{1+\ve}}{\varphi_{1+\ve}(q)}\frac{1}{\zeta(1+\ve)}
    +\frac{\ve\Delta_q(X,\ve)}{X^{\ve}}
  \end{equation*}
  where
  \begin{equation*}
    |\Delta_q(X,\ve)|\le
    0.03\cdot g_2(q)\frac{q^\xi}{\varphi_\xi(q)}\frac{\1_{X\ge 10^{12}}}{\log X}
    +
    \biggl(
    4.1\cdot g_0(q)
    +
    \frac{5 + \ve 2^\ve}{2}\biggr)\frac{\sqrt{q}}{\varphi_{\frac12}(q)}
    \frac{2^\ve}{\sqrt{X}}
  \end{equation*}
  where the multiplicative functions $g_0$ and $g_2$ are defined
  in~\eqref{defg0g2}.
  Moreover,
  \begin{equation*} \frac{m_q(X)}{X^{\ve}}\le m_q(X;1+\ve)
  \end{equation*} and
  \begin{align*}
    \frac{\Delta_q(X,\ve)}{X^\ve}&\leq\left\{
      \begin{array}{ll}
        0 &\text{ if } X < 10.9,\ \ve \le 4/25,\\
         0 &\text{ if } X < 47,\ \bigl(q,\prod_{p\le
                                      43}p\bigr)\notin\{1,11,13, 17\},\\
       0.014 &\text{ if } X \le 47,\ \bigl(q,\prod_{p\le
                                          43}p\bigr)=1,\\
   0.00005 &\text{ if } X < 47,\ \bigl(q,\prod_{p\le
                                             43}p\bigr)\in\{11,13,
                                            17\},
                        \end{array}
                      \right.\\
    \frac{\Delta_q(X,\ve)}{X^\ve}&\geq\left\{
      \begin{array}{ll}
        -\frac{q}{\varphi(q)}, &\\
         -\frac{1}{\ve\zeta(1+\ve)}\frac{q^{1+\ve}}{\varphi_{1+\ve}(q)} .&
                        \end{array}
                      \right.
  \end{align*}
\end{lem}


\section{Mean values of some multiplicative functions}\label{meanvalues}

\noindent\phantom{x}\textbf{General statements.} The following result is a direct application of \cite[Thm. 3.3]{Alterman*22}, which is a generalisation of \cite[Lemma
  3.2]{Ramare*95}, with $f'(\ell)=f(\ell)/\ell^{1-\alpha}$ via summation by parts. Based upon a convolution identity, this version essentially extends analytically a function to the left of the real line $\{\Re(s)=0\}$ to obtain an explicit mean-value.
 \begin{thm}\label{general}    
 Let  $q\in\mathbb{Z}^{+}$ and let $X$, $\alpha$, $\beta$ be real numbers such that $X>0$, $\alpha\leq 1/2$ and $\beta>\alpha$. Consider a multiplicative function $f:\mathbb{Z}^+\to\mathbb{C}$ such that 
 $f(p)=\frac{1}{p^{\alpha}}+O(\frac{1}{p^{\beta}})$, for every sufficiently large prime number $p$ coprime to $q$. Then for any real number $\delta>0$ such that $0<\delta<\min\{\beta-\alpha,1/2\}$ we have the estimate
 \begin{equation*}
 \sum_{\substack{\ell\leq X\\ (\ell,q)=1}}\mu^2(\ell)f({\ell})=\frac{H_{f'}^q(0)\varphi(q)}{(1-\alpha)q}X^{1-\alpha}+O^*\bigg(\frac{(2-2\alpha-\delta)\Delta_{1}^{\delta}\overline{H}_{f'}^{\phantom{.}q}(-\delta)}{1-\alpha-\delta}\frac{\sigma_{1-\delta}(q)}{q^{1-\delta}}\ X^{1-\alpha-\delta}\bigg),
 \end{equation*} 
 where
 \begin{equation*}
 \Delta_{1}^{\delta}=\max\bigg\{\gamma,\frac{1}{\delta e^{\gamma\delta+1}}\bigg\}
 \end{equation*} 
 and $H_{f'}^{q}, \overline{H}_{f'}^{\phantom{.}q}:\{s\in\mathbb{C},\ \Re(s)>\max\{-(\beta-\alpha),-1/2\}\}\to\mathbb{C}$ are two analytic functions defined as  
 \begin{align*}
 H_{f'}^{q}(s)&=\prod_{p\nmid q}\bigg(1-\frac{1-f(p)p^{\alpha}}{p^{s+1}}-\frac{f(p)p^{\alpha}}{p^{2s+2}}  \bigg),\\
 \overline{H}_{f'}^{\phantom{.}q}(s)&=\prod_{p\nmid q}\bigg(1+\frac{|1-f(p)p^{\alpha}|}{p^{s+1}}+\frac{|f(p)p^{\alpha}|}{p^{2s+2}}  \bigg).
 \end{align*} 
 \end{thm}
  
  The above method produces flexibility in choosing the parameter $\delta\in(0,\min\{\beta-\alpha,1/2\})$. One may be tempted to choose $\delta$ very close to $\min\{\beta-\alpha,1/2\}^{-}$ but the evidence is that the constants accompanying $X^{1-\alpha-\delta}$ turn out to be too big to be practical. Nonetheless, in \cite[Thm. 4.6]{Alterman*22}, it is proven that in the case $\beta-\alpha>1/2$ one can go to the edge of the above convolution method, thus allowing the choice $\delta=\min\{\beta-\alpha,1/2\}=1/2$ which further produces reasonable constants. 
  
  For $v\in\{1,2\}$, define
 \begin{align*}
\mathrm{E}_\alpha^{(v)}=\max\left\{\mathrm{D}_v\left(1+\frac{|\alpha-1|}{\alpha-\frac{1}{2}}\right),\frac{\varphi_{\frac{1}{2}}(v)}{\sqrt{v}}\left|\frac{v^\alpha}{\kappa_\alpha(v)}\frac{\zeta(\alpha)}{\zeta(2\alpha)}-\frac{v}{\kappa(v)}\frac{6}{(\alpha-1)\pi^2}\right|,\phantom{xxx}\right.&\\
\left.\frac{\varphi_{\frac{1}{2}}(v)}{\sqrt{v}}\frac{|\alpha-1|}{\alpha-\frac{1}{2}}\left(\frac{3\kappa_\alpha(v)\zeta(2\alpha)}{\left(\alpha-\frac{1}{2}\right)v^{\alpha-1}\kappa(v)\pi^2|\zeta(\alpha)(\alpha-1)|}\right)^{\frac{2}{\alpha-1}}\right\},\phantom{x}\text{ if }\alpha>1/2,\ \alpha\neq 1,&\\
\mathrm{E}_1^{(1)}=1.044,\qquad\qquad\qquad \mathrm{E}_1^{(2)}=0.232,\qquad\qquad\qquad\qquad&
\end{align*}
and $\mathrm{D}_1=0.43$, $\mathrm{D}_2=0.12$.
We reproduce \cite[Thm. 4.6]{Alterman*22} below.  
  \begin{thm}\label{general++}
 Let $q\in\mathbb{Z}^+$ and $X>0$. Consider a multiplicative function $f:\mathbb{Z}^+\to\mathbb{C}$ such that for every prime number $p$ satisfying $(p,q)=1$, $f(p)=\frac{1}{p^{\alpha}}+O(\frac{1}{p^{\beta}})$, where $\alpha$, $\beta$ are real numbers satisfying $\beta>\alpha$, $\beta-\alpha>\frac{1}{2}$. We have the following

 \noindent $\mathbf{[A]}$ If $\alpha>\frac{1}{2}$ then 
 \begin{align*}
 \sum_{\substack{\ell\leq X\\ (\ell,q)=1}}\mu^2(\ell)f({\ell})=F_\alpha^{q}(X)+O^*\bigg(\mathrm{t}_\alpha(q)\cdot\frac{\mathrm{w}_\alpha^q\ \mathrm{P}_\alpha}{X^{\alpha-\frac{1}{2}}}\bigg),
 \end{align*} 
 where
  \begin{align*} 
 F_\alpha^q(X)&=\frac{H_{f}^{q}(0)\zeta(\alpha)\varphi_\alpha(q)}{q^\alpha}-\frac{H_{f}^{q}(1-\alpha)\varphi(q)}{(\alpha-1)q}\frac{1}{X^{\alpha-1}},&&\text{ if }\alpha\neq 1,\\
 F_1^q(X)&=\frac{H_{f}^{q}(0)\varphi(q)}{q}\bigg(\log\left(X\right)+T_f^q+\gamma+\sum_{p|q}\frac{\log p}{p-1}\bigg),&&\text{ if }H_{f}^{q}(0)\neq 0,\\
  F_1^q(X)&=H_{f}^{q}\phantom{.}'(0),&&\text{ if }H_{f}^{q}(0)= 0,\\
T_{f}^{q}&=\sum_{p\nmid q}\frac{(1-(p-2)f(p))\log p}{(f(p)+1)(p-1)},
\end{align*}
and
\begin{align*}
\mathrm{w}_\alpha^q&=\prod_{2|q}\mathrm{E}_\alpha^{(2)}\prod_{2\nmid q}\bigg(\frac{\sqrt{2}-1}{\sqrt{2}-1+|2^\alpha f(2)-1|}\bigg)\bigg(\mathrm{E}_\alpha^{(1)}+\frac{|2^\alpha f(2)-1|\ \mathrm{E}_\alpha^{(2)}}{\varphi_{\frac{1}{2}}(2)}\bigg),\\
\mathrm{t}_\alpha(q)&=\prod_{p|q}\bigg(1+\frac{1-|f(p)p^\alpha-1|}{\sqrt{p}-1+|f(p)p^{\alpha}-1|}\bigg),
\qquad\mathrm{P}_\alpha=\prod_{p}\bigg(1+\frac{|f(p)p^\alpha-1|}{\sqrt{p}-1}\bigg),
 \end{align*}    
 \noindent $\mathbf{[B]}$ If $\alpha<\frac{1}{2}$ then $ \sum_{\substack{\ell\leq X\\ (\ell,q)=1}}\mu^2(\ell)f({\ell})$ can be expressed as 
 \begin{equation*}
\frac{H_{f'}^q(0)\varphi(q)}{(1-\alpha)q}X^{1-\alpha}+O^*\bigg(\mathrm{t}_{\alpha}(q)\cdot\bigg(1+\frac{2-2\alpha}{1-2\alpha}\bigg)\mathrm{w'}_{\alpha}^q\mathrm{P}_{\alpha} \ X^{\frac{1}{2}-\alpha}\bigg),
 \end{equation*} 
where $\mathrm{t}_{\alpha}(q)$ and $\mathrm{P}_{\alpha}$ are as in $\mathbf{[A]}$ and for $\alpha\leq\frac{1}{2}$,
 \begin{align*}
H_{f'}^q(0)&=\prod_{p\nmid q}\bigg(1-\frac{p^{1-\alpha}-f(p)p+f(p)}{p^{2-\alpha}} \bigg),\\
\mathrm{w'}_{\alpha}^q&=\prod_{2|q}\mathrm{E}_1^{(2)}
\prod_{2\nmid q}\bigg(\frac{\sqrt{2}-1}{\sqrt{2}-1+|2^\alpha f(2)-1|}\bigg)\bigg(\mathrm{E}_1^{(1)}+\frac{|2^\alpha f(2)-1|\ \mathrm{E}_1^{(2)}}{\varphi_{\frac{1}{2}}(2)}\bigg),
\end{align*} 

 \noindent $\mathbf{[C]}$ If $\alpha=\frac{1}{2}$ then $\sum_{\substack{\ell\leq X\\ (\ell,q)=1}}\mu^2(\ell)f({\ell})$ can be written as 
 \begin{equation*}
\frac{H_{f'}^q(0)\varphi(q)}{(1-\alpha)q}X^{1-\alpha}+O^*\bigg(\mathrm{C}+\mathrm{t}_{\alpha}(q)\cdot\mathrm{w'}_{\alpha}^q\mathrm{P}_{\alpha} \ \bigg(1+\frac{1}{2}\log(X)\bigg)\bigg),
 \end{equation*} 
where $\mathrm{t}_{\alpha}(q)$ and $\mathrm{P}_{\alpha}$ are as in $\mathbf{[A]}$, $H_{f'}^q(0)$ and $\mathrm{w'}_{\alpha}^q$ are as in $\mathbf{[B]}$ and
\begin{align*}
\mathrm{C}&=\left|\frac{H_{f'}^{q}(0)\varphi(q)}{q}\bigg(\sum_{p\nmid q}\frac{\log(p)(\sqrt{p}-(p-2)f(p))}{(f(p)+\sqrt{p})(p-1)}+\gamma+\sum_{p|q}\frac{\log(p)}{p-1}-2\bigg)\right|.
\end{align*}
 \end{thm}

Let us reproduce \cite[Lemma 3.4]{Ramare*13d}, which gives the following special mean-value estimate.

\begin{lem}
  \label{harmonicmu2}
  For $X\ge1$, we have 
  \begin{equation*}
  \sage{Trunc(Harm1_a,digits)}\le \sum_{\ell\le X}\frac{\mu^2(\ell)}{\ell}-\frac{6}{\pi^2}\log X\le \sage{Trunc(Harm2_a,digits)}
  \end{equation*}
  For $X\geq 10^3$, we have
  \begin{equation*}
  \sage{Trunc(Harm1,digits)}\le \sum_{\ell\le X}\frac{\mu^2(\ell)}{\ell}-\frac{6}{\pi^2}\log X\le \sage{Trunc(Harm2,digits)}
  \end{equation*}
\end{lem}

\noindent\phantom{x}\textbf{Estimating infinite sums and products.} We want to give rigorous estimates for infinite products or sums running over the prime numbers and whose signs behave regularly (not oscillating). That problem is usually reduced to that of evaluating convergent expressions of the form 
\begin{equation*}
\prod_{p}\bigg(1+\frac{v_p}{p^{\kappa}}\bigg),\qquad\qquad\sum_p\frac{u_p}{p^{\kappa}}
\end{equation*} for $\kappa>1$ and such that, for a sufficiently large $M$, there exists constants $b_M,c_M>0$ such that for any prime $p>M$, $|v_p|\leq b_M$, $|u_p|\leq c_M$.

Recall \cite[Cor. 5.2]{Dusart*18} and \cite[Thm. 4.2]{Dusart*18} by P. Dusart. 
\begin{lem}\label{Dusart1}\itshape Let $\pi(t)$ the prime counting function and $\theta(t)=\sum_{p\leq t}\log p$ the Chebyshev function. Then 
\begin{align*}\pi(t)&\leq \frac{t}{\log t}\bigg(1+\frac{1}{\log t}+\frac{2.53816}{\log^2t}\bigg),&&\text{ for any $t>1$},\phantom{xxxxxxxx}\\
\theta(t)&< t+\frac{0.2\ t}{\log^2 t},&&\text{for any $t>3\,594\,641$}.
\end{align*}
\end{lem}
Sharper bounds could be found, for instance, in \cite[Prop. 3.1]{Johnston*22} by D.R. Johnston. They are unnecessary in our computations as we obtain sufficiently sharp constants.

Suppose now that for any $p>M$, $v_p\geq 0$. Then
\begin{equation}\label{product1}
\prod_{p\leq M}\bigg(1+\frac{v_p}{p^{\kappa}}\bigg)\leq\prod_{p}\bigg(1+\frac{v_p}{p^{\kappa}}\bigg)\leq\prod_{p\leq M}\bigg(1+\frac{v_p}{p^{\kappa}}\bigg)\prod_{p> M}\bigg(1+\frac{v_p}{p^{\kappa}}\bigg),
\end{equation}
so, if we are able to evaluate $\prod_{p\leq M}(1+v_p/p^{\kappa})$, we just then need to estimate the  part $\prod_{p>M}(1+v_p/p^{\kappa})$. Continuing from \eqref{product1}, as $\log(1+x)\leq |x|$ for all $x>-1$, we have 
\begin{align}
&\sum_{p>M}\log\bigg(1+\frac{v_p}{p^{\kappa}}\bigg)\leq\sum_{p>M}\frac{v_p}{p^{\kappa}}=b_M\kappa\int_M^\infty\frac{\pi(t)}{t^{1+\kappa}}dt-\frac{b_M\pi(M)}{M^\kappa}\nonumber\\
&\quad\leq 
b_M\bigg(\frac{1}{\log M}\bigg(1+\frac{1}{\log M}+\frac{2.53816}{\log^2M}\bigg)\kappa\int_{M}^\infty\frac{dt}{t^\kappa}-\frac{\pi(M)}
{M^\kappa}\bigg):=b_MB_\kappa(M)\label{Bk_def}.
\end{align} 
We conclude that for any $M>1$ large enough so that $v_p<b_M$ for all $p>M$,  
\begin{align}\label{multup}
\prod_{p}\bigg(1+\frac{v_p}{p^{\kappa}}\bigg)\in\prod_{p\leq M}\bigg(1+\frac{v_p}{p^{\kappa}}\bigg)\times\left[1, \exp(b_M\cdot B_\kappa(M))\right].
\end{align}
On the other hand, suppose that, for any $p>M$, $v_p< 0$ and $v_p/p^\kappa\neq-1$ for all $p$. Then we obtain from \eqref{product1}
\begin{equation*}
\prod_{p\leq M}\bigg(1+\frac{v_p}{p^{\kappa}}\bigg)^{-1}\leq\prod_{p}\bigg(1+\frac{v_p}{p^{\kappa}}\bigg)^{-1}\leq\prod_{p\leq M}\bigg(1+\frac{v_p}{p^{\kappa}}\bigg)^{-1}\prod_{p> M}\bigg(1+\frac{v_p}{p^{\kappa}}\bigg)^{-1},
\end{equation*}
Moreover, for any $p>M$,
\begin{equation*}
 \bigg(1+\frac{v_p}{p^{\kappa}}\bigg)^{-1}= 1+\frac{|v_p|}{p^{\kappa}-|v_p|}\leq 1+\frac{b_M}{1-b_M/M^\kappa}\frac{1}{p^\kappa}
\end{equation*}
and we conclude that
\begin{equation}\label{multdown}
\prod_{p}\bigg(1+\frac{v_p}{p^{\kappa}}\bigg)\in
\prod_{p\leq M}\bigg(1+\frac{v_p}{p^{\kappa}}\bigg)\times\bigg[\exp\bigg(-\frac{b_M\cdot B_\kappa(M)}{1-b_M/M^\kappa}\bigg),1\bigg].
\end{equation}
Thus, if $M>3\,594\,641$, 
\begin{align}
 \sum_{p>M}\frac{\log p}{p^\kappa}&=\kappa\int_M^\infty\frac{\theta(t)}{t^{1+\kappa}}dt-\frac{\theta(M)}{M^\kappa}<
 \kappa\bigg(1+\frac{0.2}{\log^2M}\bigg)\int_M^\infty\frac{1}{t^{\kappa}}dt-\frac{\theta(M)}{M^\kappa}\nonumber\\
 &\qquad\qquad\quad<\bigg(1+\frac{0.2}{\log^2M}\bigg)\frac{\kappa}{(\kappa-1)M^{\kappa-1}}-\frac{[\theta(M)]}{M^\kappa}:=C_\kappa(M).\label{Ck_def}
\end{align}
Hence, 
\begin{align}\label{sumup}
 \sum_p\frac{u_p\log p}{p^\kappa}&\in\sum_{p\leq M}\frac{u_p\log p}{p^\kappa}+[0,c_MC_{\kappa}(M)]&&\text{ if $u_p\geq 0$ for any $p>M$},\\
 \label{sumdown}
 \sum_p\frac{u_p\log p}{p^\kappa}&\in\sum_{p\leq M}\frac{u_p\log p}{p^\kappa}+[-c_MC_{\kappa}(M),0]&&\text{ if $u_p< 0$ for any $p>M$}.
\end{align}

In all our product or sum calculations, we will set precision $M=10^8$. Observe that $\pi(10^8)=5\,761\,455$ and $[\theta(10^8)]=99\,987\,730$.\\


\noindent\phantom{x}\textbf{Auxiliaries.} In order to prove Theorem \ref{Main} and Theorem \ref{Mainbis}, we will need the following auxiliary mean-value results. Their proofs apply either Theorem \ref{general} or Theorem \ref{general++} and will be provided in $\S$\ref{auxiliaries}.
 
\begin{lem}
  \label{Aux1}\itshape
  When $X\ge 0$, we have
  \begin{align*}
   \phantom{xxxxxxxxxxxx}\mathbf{(i)}\qquad \sigma_1(X)&=\sum_{\ell\le X}
    \frac{\mu^2(\ell)\varphi(\ell)}{\varphi_{\frac12}(\ell)^2}&&\le \sage{Upper(THx1,digits)}\cdot X,\phantom{xxxxxxxxxxxxx}\\
  \mathbf{(ii)}\qquad   \sigma_2(X)&=\sum_{\ell\le X}\frac{\mu^2(\ell)\varphi(\ell)g_0(\ell)}{\varphi_{\frac{1}{2}}(\ell)^2}&&\le \sage{Upper(THx2,digits)}\cdot X,\\
    \mathbf{(iii)}\qquad\Sigma_1(X)&=\sum_{\ell\le X}\frac{\mu^2(\ell)\varphi(\ell)g_0(\ell)^2}{\varphi_{\frac{1}{2}}(\ell)^2}     &&\le \sage{Upper(TH1,digits)}\cdot X.
    \end{align*}
\noindent
In the three cases, the maximum is reached at $X=42$.
\end{lem}


\begin{lem}
  \label{Aux2}\itshape
  When $X\ge 0$, we have
  \begin{equation*}
    \sum_{\ell\le
      X}\frac{\mu^2(\ell)\varphi(\ell)}{\ell}\bigg[g_0(\ell)\frac{\sqrt{\ell}}{\varphi_{\frac{1}{2}}(\ell)}\bigg]\bigg[g_2(\ell)\frac{\ell^\xi}{\varphi_{\xi}(\ell)}\bigg]
    \le \sage{Upper(TH2,digits)}\cdot X.
    \end{equation*}
    The maximum is reached at $X=6$.
\end{lem}
\noindent


\begin{cor}\itshape
  \label{Cor1}
  When $X\ge D\ge0$, we have
  \begin{equation*}
   \Sigma_2(X)= \sum_{\ell\le \min\{D, X/10^{12}\}}
    \frac{\mu^2(\ell)\varphi(\ell)}{\ell^{3/2}\log(X/\ell)}\bigg[\frac{g_0(\ell)\sqrt{\ell}}{\varphi_{\frac{1}{2}}(\ell)}\bigg]\bigg[\frac{g_2(\ell)\ell^\xi}{\varphi_{\xi}(\ell)}\bigg]
    \le \sage{Upper(L1,digits)}\cdot\sqrt{D}.
  \end{equation*}
\end{cor}


\begin{lem}\itshape
  \label{Aux3}
  When $X\ge 0$, we have
  \begin{equation*}
    \sum_{\ell\le X}\frac{\mu^2(\ell)\varphi(\ell)}{\ell}\bigg[g_2(\ell)\frac{\ell^\xi}{\varphi_{\xi}(\ell)}\bigg]^2
    \le \sage{Upper(TH3,digits)}\cdot X.
    \end{equation*}
    \noindent
The maximum is reached at $X=2$.
\end{lem}


\begin{cor}\itshape
  \label{Cor2}
  Let $X\ge D\ge0$. Then
  \begin{equation*}
    \Sigma_3(X)=\sum_{ \ell\le \min\{D, X/10^{12}\}}
      \frac{\mu^2(\ell)\varphi(\ell)}{\ell^2\log(X/\ell)^2}\bigg[g_2(\ell)\frac{\ell^\xi}{\varphi_{\xi}(\ell)}\bigg]^2
    \le  \sage{Upper(L2,digits)}
  \end{equation*}
\end{cor}


\begin{lem}\itshape
  \label{auxnu2}
  When $X\ge 1$, we have
   \begin{align*}
    \mathbf{(i)}\qquad \Xi_1(X)&=\sum_{\ell\le X}
    \frac{\mu^2(\ell)\varphi(\ell)}{\ell\varphi_{\frac{1}{2}}(\ell)}&&\le \sage{Upper(THxxx_aux11,digits)}\cdot\sqrt{X},\\
  \phantom{xxxxxxxx} \mathbf{(ii)}\qquad \Xi_2(X)&=\sum_{\ell \le X}
    \frac{\mu^2(\ell)\varphi(\ell)}{\ell^{3/2}}\bigg[g_0(\ell)\frac{\sqrt{\ell}}{\varphi_{\frac{1}{2}}(\ell)}\bigg]&&\le \sage{Upper(THx_aux11,digits)}\cdot\sqrt{X}.\phantom{xxxxxxxxx}
  \end{align*}
\end{lem}


\begin{lem}\itshape
  \label{auxf2}
  When $X\ge1$, we have
  \begin{equation*}
    c_1\log X + \sage{Lower(logmm,digits)}
    \le K_1(X)=\sum_{\ell\le X}
     \frac{\mu^2(\ell)\varphi(\ell)}{\ell^2}\bigg[g_2(\ell)\frac{\ell^\xi}{\varphi_{\xi}(\ell)}\bigg]
     \le c_1\log X + \sage{Upper(logmx,digits)}
   \end{equation*}
   where $c_1\in(\sage{Lower(lc1,digits)},\sage{Upper(uc1,digits)})$.
\end{lem}


 \begin{cor}\itshape
  \label{auxnu1}
  When $X\geq 10^{12}$, we have
  \begin{equation*}
     L_1(X)=\sum_{\ell\le X/10^{12}}
     \frac{\mu^2(\ell)\varphi(\ell)}{\ell^2\log
       (X/\ell)}\bigg[g_2(\ell)\frac{\ell^\xi}{\varphi_{\xi}(\ell)}\bigg]
     \le \sage{Upper(loglemma1_1,digits)}\cdot\log\bigg(\frac{\log X}{12\log 10}\bigg)+\sage{Upper(loglemma1_2,digits)}.
  \end{equation*}
\end{cor}


\begin{lem}\itshape
  \label{auxf22}
  When $X\ge1$, we have
  \begin{equation*}
   c_2\log X +\sage{Lower(log2mm,digits)}
    \le K_2(X)=\sum_{\ell\le X}
     \frac{\mu^2(\ell)\varphi(\ell)}{\ell^2}[g_2(\ell)\frac{\ell^\xi}{\varphi_{\xi}(\ell)}]^2
     \le c_2\log X +\sage{Upper(log2mx,digits)}.
   \end{equation*}
   for  $c_2\in(\sage{Lower(lc2,digits)},\sage{Upper(uc2,digits)})$.
\end{lem}


 \begin{cor}\itshape
  \label{auxnu}
  When $X\ge10^{12}$, we have
  \begin{equation*}
     L_2(X)=\sum_{\ell\le X/10^{12}}
     \frac{\mu^2(\ell)\varphi(\ell)}{\ell^2\log^2
       (X/\ell)}\bigg[g_2(\ell)\frac{\ell^\xi}{\varphi_{\xi}(\ell)}\bigg]^2
     \le \sage{Upper(loglemma2,digits)}.
  \end{equation*}
\end{cor}

\section{Proof of Theorem~\ref{Main}}\label{S_Main}

\noindent\phantom{x}\textbf{Direct computations.}  For reasonably
small values of $X$, we carry out direct computations for the
evaluation of $S_0(X)$ via the following algorithm. 

When $X$ is an integer, say $\ell$, we find that
\begin{equation*}
  S_0(\ell)-S_0(\ell-1)
  =\frac{\mu^2(\ell)}{\ell}
  +2\mu(\ell)\sum_{m < \ell}\frac{\mu(m)}{[\ell,m]}.
\end{equation*}
This yields a formula to find the maximum of $S_0(\ell)$ over some
range, but each step is costly. We instead continue with
\begin{align*}
  S_0(\ell)-S_(\ell-1)
  &=\frac{\mu^2(\ell)}{\ell}
  +2\frac{\mu(\ell)}{\ell}\sum_{m < \ell}(\ell,m)\frac{\mu(m)}{m}.
  \\&=\frac{\mu^2(\ell)}{\ell}
    +2\frac{\mu(\ell)}{\ell}\sum_{\delta |\ell}\mu(\delta)
    \sum_{\substack{m < \ell/\delta \\ (\ell,m)=1}}\frac{\mu(m)}{m}.
\end{align*}
By using Lemma \ref{coprime},
\begin{equation*}
  \sum_{\substack{m < \ell/\delta \\ (\ell,m)=1}}\frac{\mu(m)}{m}
  =\sum_{k|\ell^\infty}\frac{1}{k}\sum_{\substack{m <\ell /(\delta k)}}\frac{\mu(m)}{m},
\end{equation*}
to obtain
\begin{equation*}
  S_0(\ell)-S_0(\ell-1)
  =
  \frac{\mu^2(\ell)}{\ell}
  +2\frac{\mu(\ell)}{\ell}\sum_{\substack{\delta |\ell\\ k|\ell^\infty}}\frac{\mu(\delta)}{k}
    m\bigg(\frac{\ell-1}{\delta k}\bigg).
\end{equation*}
We merge now the variables $\delta$ and $k$ in $n=\delta k$. We
have $n|\ell^\infty$ and 
\begin{equation*}
  \sum_{\delta k=n}\frac{\mu(\delta)}{k}=
  \frac{1}{n}\prod_{p|n}(1-p)=\frac{(-1)^{\omega(n)}\varphi(n)}{n}.
\end{equation*}
Hence, we obtain the more computationally efficient identity
\begin{equation}
  \label{eq:3}
  S_0(\ell)-S_0(\ell-1)
  =
  \frac{\mu^2(\ell)}{\ell}
  +2\frac{\mu(\ell)}{\ell}\sum_{\substack{n|\ell^\infty}}\frac{(-1)^{\omega(n)}\varphi(n)}{n^2}
    m\bigg(\frac{\ell-1}{n}\bigg).
\end{equation}

\begin{proofbold}[Theorem \ref{Main} - Part 1] Identity \eqref{eq:3} entails to precompute all the values $m(t)$ for $t$ up to a threshold that produces a large array. Observe that, for any $q_0\in\mathbb{N}$,
\begin{equation}\label{eqqq}
  m(t)=\sum_{\ell\le t}\frac{\mu(\ell)}{\ell}
  =\sum_{q|q_0}\frac{\mu(q)}{q}m_{q_0}\bigg(\frac{t}{q}\bigg)
  =\sum_{q|q_0}\frac{\mu(q)}{q}
  \sum_{\substack{0\le a_0<q_0\\(a_0,q_0)=1}}m\bigg(\frac{t}{q}; a_0,q_0\bigg)
\end{equation}
 where
\begin{align*}
 m(t';a_0,q_0)=\sum_{\substack{\ell\le t'\\ \ell\equiv a_0\mod q_0}}\frac{\mu(\ell)}{\ell}.
\end{align*}
Thus, for the sake of efficiency, for a particular $q_0\in\mathbb{N}$,
we store only the values of $m(t';a_0,q_0)$  
for $0\leq a_0<q_0$ and $(a_0,q_0)=1$.
 
 We chose $q_0=6$. This reduced
sizeably the amount of storage while still accessing to the values $m(t)$ via \eqref{eqqq}.

  Significant outcomes of our calculations are the following. When $422\le X\le 10^8$, 
  \begin{equation*}
    0\le S_0(X)\le 0.445.
  \end{equation*} 
  On $[6, 10^8]$, $S_0(X)$ is bounded above by $0.528$.
  On $[2, 10^8$, we have
  $S_0(X)\leq 19/30=0.633...$.
  When $X\ge 1000$, this sum remained $\ge 0.437$. A value larger than $0.44455$ is reached around $X=1321$.

\end{proofbold}


\noindent\phantom{x}\textbf{A bound for the tail.}

\begin{lem}
  \label{Tail}
  When $X\ge D>0$, we have
  \begin{equation*}
    S_0^{**}(X,D)=\sum_{ \ell\le D}\frac{\mu^2(\ell)\varphi(\ell)}{\ell^2}
    m_\ell\bigg(\frac{X}{\ell}\bigg)^2
    \le  
    \sage{Upper(2*TH1,digits)}\cdot\frac{D}{X}
    +\sage{Upper(2*sqrt(2)*H1*L1,digits)}\cdot\sqrt{\frac{D}{X}}
    +\sage{Upper(H1^2*L2,digits)}.
  \end{equation*}
\end{lem}

\begin{proof}
  By Lemma \ref{m4} and expanding out the square, we readily find that
  \begin{align*}
    &\sum_{ \ell\le D}\frac{\mu^2(\ell)\varphi(\ell)}{\ell^2}
    \bigg(\frac{g_0(\ell)\sqrt{\ell}}{\varphi_{\frac{1}{2}}(\ell)}\sqrt{\frac{2\ell}{X}}
    +\sage{Trunc(H1,6)}\cdot\frac{g_2(\ell)\ell^\xi}{\varphi_{\xi}(\ell)}\frac{\1_{X/\ell\ge 10^{12}}}{\log(X/\ell)}\bigg)^2\\
    &\leq 
    \frac{2}{X}\Sigma_1(D)
    +\frac{2\sqrt{2}\times\sage{Trunc(H1,6)}}{\sqrt{X}}\Sigma_2(X)
    +\sage{Trunc(H1,6)}^2\cdot\Sigma_3(X),
    \end{align*}
  where $\Sigma_1(X)$, $\Sigma_2(X)$ and $\Sigma_3(X)$ are defined in Lemma \ref{Aux1} and corollaries \ref{Cor1} and~\ref{Cor2}, respectively. By recalling those results, we have
   \begin{align*}
    &\sum_{ \ell\le D}\frac{\mu^2(\ell)\varphi(\ell)}{\ell^2}
    \bigg(\frac{g_0(\ell)\sqrt{\ell}}{\varphi_{\frac{1}{2}}(\ell)}\sqrt{\frac{2\ell}{X}}
    +\sage{Trunc(H1,6)}\cdot\frac{g_2(\ell)\ell^\xi}{\varphi_{\xi}(\ell)}\frac{\1_{X/\ell\ge 10^{12}}}{\log(X/\ell)}\bigg)^2\\
    &\leq 
    \frac{2\times\sage{Upper(TH1,digits)}}{X}\ D
    +\frac{2\sqrt{2}\times\sage{Trunc(H1,6)}\times\sage{Upper(L1,digits)}}{\sqrt{X}}\sqrt{D}
    +\sage{Trunc(H1,6)}^2\times  \sage{Upper(L2,digits)},
    \end{align*}
    which, upon calculation, gives the result.
\end{proof}


\noindent\phantom{x}\textbf{A special function.}
To prove Theorem \ref{Main}, we will need to handle the following function with coprimality conditions
 \begin{align*}
     \mathrm{G}_q(X)=\sum_{\substack{\ell\le X\\
      (\ell,q)=1}}\frac{\mu^2(\ell)\varphi(\ell)}{\ell^2} .
      \end{align*}
 \begin{lem} \label{GetGq}
  For every $X>0$, 
  \begin{align*}
     \mathrm{G}_q(X)=\begin{cases}&
    \mathrm{W}\ \mathrm{r}_q\bigl(\log X+\mathrm{s}_q\bigr)
    +O^*(\sage{Upper(G1,digits)}\cdot\mathrm{t}_q/\sqrt{X}),\quad\text{ if }(2,q)=1,\\
    &
    \mathrm{W}\ \mathrm{r}_q\bigl(\log X+\mathrm{s}_q\bigr)
    +O^*(\sage{Upper(G2,digits)}\cdot \mathrm{t}_q/\sqrt{X}),\quad\text{ if }2|q,
    \end{cases}
  \end{align*}
  where
  \begin{align*}
 \mathrm{W}=\prod_{p}\bigg(1-\frac{2}{p^2}+\frac{1}{p^3}\bigg)\in(\sage{Lower(APROD,digits)},\sage{Upper(APROD,digits)}),\phantom{xxxxxxxxxxxxxxxx}&\\
  \mathrm{r}_q=\prod_{p|q}\bigg(\frac{p^2}{p^2+p-1}\bigg),
    \quad
\mathrm{t}_q=\prod_{p|q}\bigg(\frac{p^{3/2}}{p^{3/2}-p+1}\bigg),\phantom{xxxxxxxxxxxx}&\\
  \mathrm{c}_q=\sum_{p|q}\frac{(p-1)\log p}{p^2+p-1},\ \ 
    \mathrm{s}_q=\mathrm{c}_q+\gamma
    +\sum_{p}\frac{(3p-2)\log p}{(p-1)(p^2+p-1)}\in\mathrm{c}_q+(\sage{Lower(ASUM,digits)},\sage{Upper(ASUM,digits)}).&
  \end{align*}
\end{lem}

\begin{proof} Note that, for any sufficiently large prime $p$,
\begin{equation*} 
f(p)=\frac{\varphi(p)}{p^2}=\frac{1}{p}+O\bigg(\frac{1}{p^2}\bigg).
\end{equation*}

 Thus a direct application of Theorem \ref{general++}$\mathbf{[A]}$ with $(\alpha,\beta)=(1,2)$ and gives 
  \begin{equation*}
 \mathrm{G}_q(X)=\mathrm{F}_q(X)+O^*\bigg(\frac{\mathrm{w}_q\ \mathrm{P}\cdot\mathrm{t}_q}{\sqrt{X}}\bigg)
 \end{equation*}
 with
 \begin{equation*}
 F_1^q(X)=\mathrm{W}\ \mathrm{r}_q\bigg(\log X+\mathrm{c}_q+\gamma+\sum_{p}\frac{(3p-2)\log p}{(p-1)(p^2+p-1)}\bigg),
 \end{equation*}
  where $\mathrm{W}$, $\mathrm{r}_q$, $\mathrm{c}_q$ and $\mathrm{t}_q$ are defined as in the statement and 
  \begin{align*}
  \mathrm{w}_q&=\prod_{2|q}\sage{Trunc(E2,digits)}\ \prod_{2\nmid q}\sage{Upper(W2,digits)}, \\
 \mathrm{W}=\prod_{p}\bigg(1-\frac{2}{p^2}+\frac{1}{p^3}\bigg)&\in
\prod_{p\leq M}\bigg(1-\frac{2}{p^2}+\frac{1}{p^3}\bigg)\times\bigg[\exp\bigg(-\frac{2\cdot B_2(M)}{1-2/M^2}\bigg),1\bigg]\\
&\ \subset(\sage{Lower(APROD*exp(-2*B(2,M)/(1-2/M^2)),digits)},\sage{Upper(APROD,digits)}),\\
\gamma+\sum_{p}\frac{(3p-2)\log p}{(p-1)(p^2+p-1)}\\
\in\gamma+\sum_{p\leq M}&\frac{(3p-2)\log p}{(p-1)(p^2+p-1)}
+\bigg[0,\frac{3\cdot C_2(M)}{(1-M^{-1})(1-M^{-2})}\bigg]\\
&\ \subset(\sage{Lower(ASUM,digits)},\sage{Upper(ASUM+3*C(2,M)/(1-M^(-1))/(1-M^(-2)),digits)}),\\
\mathrm{P}=\prod_{p}\bigg(1+\frac{1}{p^{3/2}-p}\bigg)&\in\prod_{p\leq M}\bigg(1+\frac{1}{p^{3/2}-p}\bigg)\times\bigg[1,\exp\bigg(\frac{B_{3/2}(M)}{1-M^{-1/2}}\bigg)\bigg] \\
&\ \subset(\sage{Lower(PPdelta,digits)},\sage{Upper(APROD_delta,digits)}).
\end{align*}     
Note that
 \begin{equation*}\mathrm{w}_q\mathrm{P}\leq\prod_{2|q}\sage{Upper(G1,digits)}\ \prod_{2\nmid q}\sage{Upper(G2,digits)}
 \end{equation*}
 defining the numerical value of the error term in the statement.
 \end{proof}
 
\begin{proofbold}[Theorem \ref{Main} - Part 2]
The lower bound follows from the following interpretation of the constant. By the Parseval identity for Dirichlet series, when $\Re s=1+\ve/2>1$,
\begin{align*}
0\leq  \lim_{T\rightarrow\infty}
  \frac{1}{4i\pi T}\int_{-T}^T
  \biggl|\zeta(s)\sum_{d\le X}\frac{\mu(d)}{d^s}\biggr|^2&ds
  =\sum_{n\ge 1}\frac{1}{n^{1+\ve}}
  \biggl(\sum_{\substack{d|n\\ d\le X}}\mu(d)\biggr)^2
  \\=\ &\zeta(1+\ve)\sum_{\substack{d_1,d_2\le X}}
  \frac{\mu(d_1)\mu(d_2)}{[d_1,d_2]^{1+\ve}}=\zeta(1+\ve)S_\ve(X).
\end{align*}
Thus, since $\zeta(1+\ve)>0$ (see Lemma \ref{zeta}), we derive $S_\ve(X)\geq 0$ for all $\ve>0$ and then a continuity argument proves that $S_0(X)\geq 0$. 

Alternatively, we readily find that
  \begin{align*}
    S_0(X)=\sum_{\ell\le X}\frac{\mu^2(\ell)\varphi(\ell)}{\ell^2}
    \biggl(\sum_{\substack{m\le X/\ell\\ (\ell,m)=1}}\frac{\mu(m)}{m}\biggr)^2,
  \end{align*}
  which also proves that $S_0(X)\geq 0$. Moreover, for a parameter $D$ to be chosen later, we may write 
  \begin{equation}\label{split}
    S_0(X)=S_0^*(X,D)+S_0^{**}(X,D),
    \end{equation}
  where
   \begin{align*}
    S_0^*(X,D)
    = \sum_{D< \ell\le X}\frac{\mu^2(\ell)\varphi(\ell)}{\ell^2}m_\ell\bigg(\frac{X}{\ell}\bigg)^2,
  \end{align*}
  and $S_0^{**}(X,D)$ is defined in Lemma \ref{Tail}.
  
 Set $\Pi(j)=\prod_{p\le j}p$ and write
  \begin{align}\label{S0*}
    S_0^*(X,D)
    &=
    \sum_{j\leq \frac{X}{D}}\sum_{\max(D,\frac{X}{j+1})< \ell\le \frac{X}{j}}\frac{\mu^2(\ell)\varphi(\ell)}{\ell^2}m_\ell(j)^2\nonumber\\
    =&  \sum_{\max(D,\frac{X}{2})< \ell\le X}\frac{\mu^2(\ell)\varphi(\ell)}{\ell^2}
+    \sum_{1<j\leq \frac{X}{D}}\sum_{\max(D,\frac{X}{j+1})< \ell\le \frac{X}{j}}\frac{\mu^2(\ell)\varphi(\ell)}{\ell^2}m_\ell(j)^2
    \nonumber\\
    =&
    \sum_{\max(D,\frac{X}{2})< \ell\le X}\frac{\mu^2(\ell)\varphi(\ell)}{\ell^2}
+ \sum_{1<j\leq\frac{X}{D}}\sum_{n|\Pi(j)}
    m_n(j)^2
    \sum_{\substack{\max(D,\frac{X}{j+1})< \ell\le\frac{X}{j}\\ (\ell,\Pi(j))=n}}\frac{\mu^2(\ell)\varphi(\ell)}{\ell^2}
    \nonumber\\\le&
     \sum_{\frac{X}{2}< \ell\le X}\frac{\mu^2(\ell)\varphi(\ell)}{\ell^2}
+ S_0^{***}(X,D)
\end{align}
where
  \begin{align*}S_0^{***}(X,D)&=
    \sum_{1<j\le \frac{X}{D}}\sum_{\substack{n|\Pi(j)\\n\leq\frac{X}{j}}}
    \frac{\mu^2(n)\varphi(n)}{n^2}m_n(j)^2
    \sum_{\substack{\frac{X}{(j+1)n}< \ell\le \frac{X}{jn}
    \\ (\ell,\Pi(j))=1}}\frac{\mu^2(\ell)\varphi(\ell)}{\ell^2}.
  \end{align*}
  A direct usage of Lemma \ref{GetGq} with $q=1$ gives
    \begin{align*}
\sum_{\frac{X}{2}< \ell\le X}\frac{\mu^2(\ell)\varphi(\ell)}{\ell^2}= \mathrm{W}\cdot\log 2
    +O^*\bigg(\frac{\sage{Upper(G1,digits)}\times (1+\sqrt{2})}{\sqrt{X}}\bigg) = \mathrm{W}\cdot\log 2
    +O^*\bigg(\frac{\sage{Upper(G1*(1+sqrt(2)),digits)}}{\sqrt{X}}\bigg).
  \end{align*}
 Note that, unless $j=1$, we always have $2|\Pi(j)$. Using again Lemma \ref{GetGq} with $q=\Pi(j)$ and $2|q$, we have
  \begin{align*}S_0^{***}(X,D)=\mathrm{W}\ S_{1}^{(1)}(X,D)+O^*\bigg(\frac{\sage{Upper(G2,digits)}\cdot S_{1}^{(2)}(X,D)}{\sqrt{X}}\bigg),
  \end{align*}
  and we conclude from Lemma \ref{Tail} and equation \eqref{S0*} that
    \begin{align}\label{S0*_final}
    S_0^*(X,D)\le\mathrm{W}\ (S_{1}^{(1)}(X,D)+\log 2)+O^*\bigg(\frac{\sage{Upper(G2,digits)}\cdot S_{1}^{(2)}(X,D)+\sage{Upper(G1*(1+sqrt(2)),digits)}}{\sqrt{X}}\bigg),
\end{align}
  where
  \begin{align*}
  S_{1}^{(1)}(X,D)&=\sum_{1<j\le \frac{X}{D}}\mathrm{r}_{\Pi(j)}\log\bigg(\frac{j+1}{j}\bigg) \sum_{\substack{n|\Pi(j)\\n\leq\frac{X}{j}}}
    \frac{\mu^2(n)\varphi(n)}{n^2}m_n(j)^2\\
     S_{1}^{(2)}(X,D)&=\sum_{1<j\le \frac{X}{D}}\mathrm{t}_{\Pi(j)}(\sqrt{j+1}+\sqrt{j})\sum_{\substack{n|\Pi(j)\\ n\leq \frac{X}{j}}}
    \frac{\mu^2(n)\varphi(n)}{n^{3/2}}m_n(j)^2
  \end{align*}
 Since $\Pi(80)\times 80\in(10^{32},10^{33})$, when $X\geq 10^{33}$ and $j\leq 80$, the innermost condition $\{n|\Pi(j),n\leq X/j\}$ of the two above innermost sums becomes $\{n|\Pi(j)\}$.
Therefore, for any $X\geq 10^{33}$,

 \begin{align*}
        S_{1}^{(1)}(X,X/80.9999)&\leq\sum_{1<j\le 80}\mathrm{r}_{\Pi(j)}\log\bigg(\frac{j+1}{j}\bigg)\sum_{n|\Pi(j)}
\frac{\mu^2(n)\varphi(n)}{n^2}m_n(j)^2\\
&\qquad\in(\sage{Lower(S11_80,digits)},\sage{Upper(S11_80,digits)}),\\
        S_{1}^{(2)}(X,X/80.9999)&\leq\sum_{1<j\le 80}\mathrm{t}_{\Pi(j)}(\sqrt{j+1}+\sqrt{j})\sum_{\substack{n|\Pi(j)}}
    \frac{\mu^2(n)\varphi(n)}{n^{3/2}}m_n(j)^2
    \\&\qquad\in(\sage{Lower(S12_80,digits)},\sage{Upper(S12_80,digits)}).
  \end{align*}
  Thus, for $X\geq 10^{33}$ and $X/D=80.9999$, we conclude from \eqref{split}, Lemma \ref{Tail} and \eqref{S0*_final} that 
   \begin{align*}S_0(X)&\leq\frac{\sage{Upper(2*TH1,digits)}}{80.9999}
    +\frac{\sage{Upper(2*sqrt(2)*H1*L1,digits)}}{\sqrt{80.9999}}
    +\sage{Upper(H1^2*L2,digits)}\\&\qquad+
    \sage{Upper(APROD,digits)}\times (\sage{Upper(S11_80,digits)}+\log 2)+\frac{\sage{Upper(G2,digits)}\times \sage{Upper(S12_80,digits)}+\sage{Upper(G1*(1+sqrt(2)),digits)}}{\sqrt{10^{33}}}\\
    &\leq \sage{Upper(Thresh33_1+Thresh33_2,digits)}.
 \end{align*} 
 It remains to study the case $X\in[10^8,10^{33}]$. We will study intervals $[10^u,10^v]$ with $8\leq u<v\leq 33$. Let $\eta>1$ a fixed constant such that $X\in[\eta^{k}10^u,\eta^{k+1}10^u)$, $\eta^{k}10^u\leq 10^{v}$ for $k\geq 1$ and $u<v\leq 33$. Additionally, suppose that $X/D=J+0.9999$ where $[X/D]=J\leq 80$. Then 
 \begin{equation*}\max_{10^u\leq X<10^{v}}S_0(X)\leq \min_{2\leq J\leq 80}\ \max_{0\leq k\leq[\frac{(v-u)\log 10}{\log \eta}]}U(\eta;u,v;k,J),
 \end{equation*}
  where
 \begin{align*}
 U(\eta;u,v;k,J) =& 
    \frac{\sage{Upper(2*TH1,digits)}}{J+0.9999}
    +\frac{\sage{Upper(2*sqrt(2)*H1*L1,digits)}}{\sqrt{J+0.9999}}
    +\sage{Upper(H1^2*L2,digits)}
    +\sage{Upper(APROD,digits)}\times\log 2
    + \frac{\sage{Upper(G1*(1+sqrt(2)),digits)}}{\sqrt{\eta^{k}10^u}}\\
    &+\sage{Upper(APROD,digits)}\sum_{1<j\le J}\mathrm{r}_{\Pi(j)}\log\bigg(\frac{j+1}{j}\bigg) \sum_{\substack{n|\Pi(j)\\n< \eta^{k+1}10^u/j}}
    \frac{\mu^2(n)\varphi(n)}{n^2}m_n(j)^2\\
     &+\frac{\sage{Upper(G2,digits)}}{\sqrt{\eta^k10^u}}\sum_{1<j\le J}\mathrm{t}_{\Pi(j)}(\sqrt{j+1}+\sqrt{j})\sum_{\substack{n|\Pi(j)\\n<\eta^{k+1}10^u/j}}
    \frac{\mu^2(n)\varphi(n)}{n^{3/2}}m_n(j)^2.
  \end{align*}
By optimizing, we have that
\[
\begin{alignedat}{2}
\max_{10^8\leq X\leq 10^{12}}S_0(X) &\leq U(1.2;8,12;0,22) &\ \leq\ & \sage{Upper(Bound1,digits)},\\
\max_{10^{12}\leq X\leq 10^{16}}S_0(X) &\leq U(2.2;12,16;0,61) &\ \leq\ & \sage{Upper(Bound2,digits)},\\
\max_{10^{16}\leq X\leq 10^{33}}S_0(X) &\leq U(5.2;16,33;0,80) &\ \leq\ & \sage{Upper(Bound3,digits)},
\end{alignedat}
\]
 whence the result.
\end{proofbold}

\section{Proof of Theorem \ref{Mainbis}}\label{S_Mainbis}

Recall \cite[Lemma 8.1$\mathbf{(i)}$]{Ramare-Zuniga*23-1} concerning an estimation of the Riemann zeta function $\zeta$ near $1^+$.
\begin{lem}\label{zeta} Let $\ve>0$. Then
\begin{equation*}
\frac{1}{\ve}<\zeta(1+\ve)\leq\frac{e^{\gamma\ve}}{\ve}.
\end{equation*}
\end{lem}
In this section, we consider $S_\ve(X)$ with $\ve\in(0,1/10]$, which allows us to apply Lemma \ref{m1e} in its full extension.
We find classically that
\begin{equation}
  \label{classic}S_\ve(X)=\sum_{\ell\le X}
  \frac{\mu^2(\ell)\varphi_{1+\ve}(\ell)}{\ell^{2+2\ve}}m_\ell\bigg(\frac{X}{\ell};1+\ve\bigg)^2,
\end{equation}
and this gives the lower bound $0\leq S_\ve(X)$. On the other hand, define
\begin{align}
\label{vv1}n_\ell(X;1+\ve)&=m_\ell(X;1+\ve)-\frac{\ell^{1+\ve}}{(\varphi_{1+\ve}(\ell)\zeta(1+\ve)}\\
&\label{vv2}=X^{-\ve}(m_\ell(X)+\ve\Delta_\ell(X,\ve)),
\end{align}
where the second equality comes from the definition of $\Delta_\ell(X,\ve)$ in Lemma~\ref{m1e}. Moreover, by using \eqref{vv1}, we have
\begin{align*}
  S_\ve(X)
  &=\sum_{\ell\le X}
  \frac{\mu^2(\ell)\varphi_{1+\ve}(\ell)}{\ell^{2+2\ve}}
  \biggl(n_\ell\bigg(\frac{X}{\ell};1+\ve\bigg)
  +\frac{\ell^{1+\ve}}{\varphi_{1+\ve}(\ell)\zeta(1+\ve)}\biggr)^2
  \\&=
  \sum_{\ell\le X}
  \frac{\mu^2(\ell)\varphi_{1+\ve}(\ell)}{\ell^{2+2\ve}}
  n_\ell\bigg(\frac{X}{\ell};1+\ve\bigg)^2
  \\&\qquad\qquad
  +\frac{2}{\zeta(1+\ve)}\sum_{\ell\le X}
  \frac{\mu^2(\ell)}{\ell^{1+\ve}}
  n_\ell\bigg(\frac{X}{\ell};1+\ve\bigg)
  +
  \frac{G(X;1+\ve)}{\zeta(1+\ve)^2},
\end{align*}
where
\begin{equation}
  \label{defG}
  G(X;1+\ve)=\sum_{\ell\le X}
  \frac{\mu^2(\ell)}{\varphi_{1+\ve}(\ell)}.
\end{equation}
Additionally, observe that
\begin{align*}
  \sum_{d\le x}
  \frac{\mu^2(\ell)}{\ell^{1+\ve}}
  n_\ell\bigg(\frac{X}{\ell};1+\ve\bigg)
  &=\sum_{L\le X}\frac{\mu^2(L)}{L^{1+\ve}}\sum_{\ell|L}\mu(\ell)
  -\sum_{\ell\le X}
  \frac{\mu^2(\ell)}{\varphi_{1+\ve}(\ell)\zeta(1+\ve)}
  \\&=1-\frac{G(X;1+\ve)}{\zeta(1+\ve)}.
\end{align*}
Subsequently,
\begin{equation*}
   S_\ve(X)
  =\sum_{\ell\le X}
  \frac{\mu^2(\ell)\varphi_{1+\ve}(\ell)}{\ell^{2+2\ve}}
  n_\ell\bigg(\frac{X}{\ell};1+\ve\bigg)^2
  +\frac{2}{\zeta(1+\ve)}-\frac{G(X;1+\ve)}{\zeta(1+\ve)^2}.
\end{equation*}
Furthermore, since the function $t\in\mathbb{R}\mapsto 1/t-1/t^2$ is decreasing for $t\geq 2$, for any $\ve\geq 0$, we have
\begin{equation*}
\frac{\mu^2(\ell)\varphi_{1+\ve}(\ell)}{\ell^{2+2\ve}}\leq\frac{\mu^2(\ell)\varphi(\ell)}{\ell^2},
\end{equation*}
and  
\begin{equation*}
   S_\ve(X)
  \leq\sum_{\ell\le X}
  \frac{\mu^2(\ell)\varphi(\ell)}{\ell^{2}}
  n_\ell\bigg(\frac{X}{\ell};1+\ve\bigg)^2
  +\frac{2}{\zeta(1+\ve)}-\frac{G(X;1+\ve)}{\zeta(1+\ve)^2}.
\end{equation*}
Now, by using \eqref{vv2} and recalling \eqref{classic}, we derive
\begin{align}
  S_\ve(X)\label{exxp}
  \leq&\ \frac{1}{X^{2\ve}}\sum_{\ell\le X}
  \frac{\mu^2(\ell)\varphi(\ell)}{\ell^{2}}
  m_\ell\bigg(\frac{X}{\ell}\bigg)^2
  +2\ve A(X)
  +\ve^2B(X)\nonumber\\
  &\qquad\qquad\qquad+\frac{2}{\zeta(1+\ve)}-\frac{G(X;1+\ve)}{\zeta(1+\ve)^2}\nonumber\\
 =&\ \frac{S_0(X)}{X^{2\ve}}
  +2\ve A(X)+\ve^2 B(X)+\frac{2}{\zeta(1+\ve)}-\frac{G(X;1+\ve)}{\zeta(1+\ve)^2},
\end{align}
where
\begin{align}
  \label{defA}
  A(X)& =\frac{1}{X^{2\ve}} \sum_{\ell\le X}
  \frac{\mu^2(\ell)\varphi(\ell)}{\ell^{2}}m_\ell\bigg(\frac{X}{\ell}\bigg)\Delta_\ell\bigg(\frac{X}{\ell},\ve\bigg),\\
  \label{defB}
  B(X) &= \frac{1}{X^{2\ve} } \sum_{\ell\le X}
  \frac{\mu^2(\ell)\varphi(\ell)}{\ell^{2}}\Delta_\ell^2\bigg(\frac{X}{\ell},\ve\bigg).
\end{align}

\begin{lem}\label{A_est} For any $X\geq 1$ and any $\ve\in[0,1/10]$,
    \begin{align*}
  A(X) 
   \leq
& \frac{\sage{Upper(3/100*loglemma1_1,digits)}\cdot \1_{X\ge 10^{12}}}{X^{2\ve}}\log\bigg(\frac{\log X}{12\log 10}\bigg)+\frac{\sage{Upper(3/100*loglemma1_2+ax_1,digits)}}{X^\ve}+\sage{Upper(ax_2*THx_aux11*41/10/sqrt(109/10)+THxxx_aux11*5/2/sqrt(109/10)+subs*THx_aux11*41/10/sqrt(47)+THxxx_aux11*5/2/sqrt(47),digits)}\cdot\frac{2^\ve}{X^{2\ve}}
     + \sage{Upper(ax_2*THxxx_aux11/2/sqrt(109/10)+subs*THxxx_aux11/2/sqrt(47),digits)}\cdot\frac{\ve 2^{2\ve}}{X^{2\ve}}.
  \end{align*}
\end{lem}

\begin{proof} We divide the proof into four parts.\\
\noindent\phantom{$\bullet$}$\bullet$\ Bounding $A(X)$, first reduction. By direct verification, $m_\ell(t)\ge0$
and $\Delta_\ell(t,\ve)\le 0$ when $t<5$ and $\ve\leq 1/10$. We may thus discard the terms with $X/\ell<5$ and derive
\begin{equation}
  A(X) \leq \frac{1}{X^{2\ve}}\sum_{\ell\le \frac{X}{5}}
  \frac{\mu^2(\ell)\varphi(\ell)}{\ell^{2}}m_\ell\bigg(\frac{X}{\ell}\bigg)\Delta_\ell\bigg(\frac{X}{\ell},\ve\bigg)
\end{equation}

\noindent\phantom{$\bullet$}$\bullet$\ Bounding $A(X)$, second reduction. 
 When $\ell > X/10.9$, we still have $\Delta_d(X/\ell,\ve)\le0$ but
$m_\ell(X/\ell)$ may be negative. However, a quick computation shows that
$m_\ell(t)\ge -1/30$ for $t\in[5,10.9)$. Thus
\begin{align*}
  A(X)\le &-\frac{1}{30 X^{2\ve}}\sum_{\frac{X}{10.9}< \ell\le \frac{X}{5}}
  \frac{\mu^2(\ell)\varphi(\ell)}{\ell^{2}}\Delta_\ell\bigg(\frac{X}{\ell},\ve\bigg)
  \\&\qquad +
  \frac{1}{X^{2\ve}}\sum_{\ell\le \frac{X}{10.9}}
  \frac{\mu^2(\ell)\varphi(\ell)}{\ell^{2}}m_\ell\bigg(\frac{X}{\ell}\bigg)\Delta_\ell\bigg(\frac{X}{\ell},\ve\bigg).
\end{align*}
On the first sum, by Lemma \ref{m1e}  we may use 
\begin{equation*}-\Delta_\ell\bigg(\frac{X}{\ell},\ve\bigg)\leq\frac{X^\ve}{\ell^\ve}\frac{\ell}{\varphi(\ell)}
\leq X^\ve \frac{\ell}{\varphi(\ell)}
\end{equation*} 
and
then appeal to Lemma~\ref{harmonicmu2}. Thus, provided that $X\geq 10.9$, this leads to
\begin{align}\label{second_red}
  A(X)&\le\frac{1}{30X^\ve}\sum_{\frac{X}{10.9}< \ell\le \frac{X}{5}}
  \frac{\mu^2(\ell)}{\ell} +
  \frac{1}{X^{\ve}}\sum_{\ell\le \frac{X}{10.9}}
  \frac{\mu^2(\ell)\varphi(\ell)}{\ell^{2}}m_\ell\bigg(\frac{X}{\ell}\bigg)\Delta_\ell\bigg(\frac{X}{\ell},\ve\bigg)\nonumber\\
  &\le\frac{1}{30X^\ve}\bigg(\frac{6\log(10.9/5)}{\pi^2}+\sage{Trunc(Harm2,digits)}-  \sage{Trunc(Harm1,digits)}\bigg)\nonumber\\
  &\qquad\qquad
  + \frac{1}{X^{2\ve}}\sum_{\ell\le \frac{X}{10.9}}
  \frac{\mu^2(\ell)\varphi(\ell)}{\ell^{2}}m_\ell\bigg(\frac{X}{\ell}\bigg)\Delta_\ell\bigg(\frac{X}{\ell},\ve\bigg)\nonumber\\
  &\le\frac{\sage{Upper(1/30*(6/pi^2*log(10.9/5)+Harm2-Harm1),digits)}}{X^\ve}
  + \frac{1}{X^{2\ve}} \sum_{\ell\le \frac{X}{10.9}}
  \frac{\mu^2(\ell)\varphi(\ell)}{\ell^{2}}m_\ell\bigg(\frac{X}{\ell}\bigg)\Delta_\ell\bigg(\frac{X}{\ell},\ve\bigg).
\end{align}
\noindent\phantom{$\bullet$}$\bullet$\ Bounding $A(X)$, third reduction. A simple program shows that
\begin{equation*}
\min_{\substack{10.9<t<47\\(\ell,\prod_{p<47}p)\notin\{1,11,13, 17\}}}m_\ell(t)\ge -\frac{2323}{30030},\qquad \max_{\substack{10.9<t<47\\(\ell,\prod_{p<47}p)\in\{1,11,13, 17\}}} m_\ell(t)\le \frac{57731}{570570}
\end{equation*} 
(the first inequality above could be have been 
foreseen, see Question~1 in~\cite{Ramare*12-4}). Furthermore, by Lemma \ref{m1e}, when $\ell< X/47$ and $(\ell,\prod_{p<47}p)\notin\{1,11,13, 17\}$, the
quantity $\Delta_\ell(X/\ell,\ve)$ is non-positive.
Therefore,
when $X/47<\ell\le \frac{X}{10.9}$, 
\begin{align}\nonumber
  m_\ell\bigg(\frac{X}{\ell}\bigg)\Delta_\ell\bigg(\frac{X}{\ell},\ve\bigg)&\le
  \max\Bigl(\frac{2323}{30030}\bigg|\Delta_\ell\bigg(\frac{X}{\ell},\ve\bigg)\bigg|,
  \frac{57731}{570570}\bigg|\Delta_\ell\bigg(\frac{X}{\ell},\ve\bigg)\bigg|\bigg)\\
  &\leq\sage{Upper(max(u1,u2),digits)}\bigg|\Delta_\ell\bigg(\frac{X}{\ell},\ve\bigg)\bigg|.
\end{align}
When $\ell\leq X/47$ we use the inequality given in Lemma \ref{BOUND}. By continuing from \eqref{second_red}, we have
\begin{align*}
  A(X)&\le
  \frac{\sage{Upper(ax_1,digits)}}{X^\ve}+\frac{\sage{Upper(ax_2,digits)}}{X^{2\ve}}\sum_{\frac{X}{47} < \ell\le \frac{X}{10.9}}
  \frac{\mu^2(\ell)\varphi(\ell)}{\ell^{2}}\bigg|\Delta_\ell\bigg(\frac{X}{\ell},\ve\bigg)\bigg|
  \\&\qquad\qquad+\frac{1}{X^{2\ve}}\sum_{\ell\le \frac{X}{47}}
  \frac{\mu^2(\ell)\varphi(\ell)}{\ell^{2}}\bigg|\Delta_\ell\bigg(\frac{X}{\ell},\ve\bigg)\bigg|
\end{align*}
which simplifies into
\begin{align*}
  &A(X)\le
  \frac{\sage{Upper(ax_1,digits)}}{X^\ve}+\frac{\sage{Upper(ax_2,digits)}}{X^{2\ve}}\sum_{\frac{X}{10^{12}} < \ell\le \frac{X}{10.9}}
  \frac{\mu^2(\ell)\varphi(\ell)}{\ell^{2}}\bigg|\Delta_\ell\bigg(\frac{X}{\ell},\ve\bigg)\bigg|\\
  \\&+\frac{1-\sage{Upper(ax_2,digits)}}{X^{2\ve}}\sum_{\frac{X}{10^{12}} < \ell\le \frac{X}{47}}
  \frac{\mu^2(\ell)\varphi(\ell)}{\ell^{2}}\bigg|\Delta_\ell\bigg(\frac{X}{\ell},\ve\bigg)\bigg|
  + \frac{1}{X^{2\ve}}\sum_{\ell\le \frac{X}{10^{12}}}
  \frac{\mu^2(\ell)\varphi(\ell)}{\ell^{2}}\bigg|\Delta_\ell\bigg(\frac{X}{\ell},\ve\bigg)\bigg|.
\end{align*}
By using the bound for $\Delta_\ell(X/\ell,\ve)$ presented in Lemma \ref{m1e} and upon rearranging terms, we derive 
    \begin{align}\label{crux}
  A(X)\le&\ \frac{ \sage{Upper(ax_1,digits)}}{X^\ve}+\frac{\sage{Upper(ax_2,digits)}}{X^{2\ve}}\sum_{  \ell\le \frac{X}{10.9}}
  \frac{\mu^2(\ell)\varphi(\ell)}{\ell^{2}}\bigg(\biggl(
    4.1\,g_0(\ell)+ \frac{5 + \ve 2^\ve}{2}\frac{\sqrt{\ell}}{\varphi_{\frac12}(\ell)}
    \biggr)\frac{2^\ve\sqrt{\ell}}{\sqrt{X}}\bigg)\nonumber
  \\&\ +\frac{1-\sage{Upper(ax_2,digits)}}{X^{2\ve}}\sum_{ \ell\le \frac{X}{47}}
  \frac{\mu^2(\ell)\varphi(\ell)}{\ell^{2}}\bigg(\biggl(
    4.1\,g_0(\ell)
    +
    \frac{5 + \ve 2^\ve}{2}\frac{\sqrt{\ell}}{\varphi_{\frac12}(\ell)}
    \biggr)\frac{2^\ve\sqrt{\ell}}{\sqrt{X}}\bigg)\nonumber
  \\&+ \frac{0.03}{X^{2\ve}}\sum_{\ell\le \frac{X}{10^{12}}}
  \frac{\mu^2(\ell)\varphi(\ell)}{\ell^{2}}\bigg(\,g_2(\ell)\frac{\ell^\xi}{\varphi_\xi(\ell)}\frac{\1_{X/10^{12}\ge\ell}}{\log (X/\ell)}
    \bigg)\nonumber\\
    \leq&\ 
  \frac{\sage{Upper(ax_1,digits)}}{X^\ve}+\frac{\sage{Upper(ax_2,digits)}}{X^{2\ve}}\cdot A_1(X)+\frac{\sage{Trunc(subs,digits)}}{X^{2\ve}}\cdot A_2(X)+ \frac{0.03\ \1_{X\ge 10^{12}}}{X^{2\ve}}\cdot L_1(X),
  \end{align}
where $L_1(X)$ is defined in Corollary \ref{auxnu1}.

\noindent\phantom{$\bullet$}$\bullet$\ Bounding $A(X)$, final treatment. 
Observe that
    \begin{align*}
  A_1(&X)    =\frac{4.1\cdot 2^\ve}{\sqrt{X}}\sum_{  \ell\le \frac{X}{10.9}}
  \frac{\mu^2(\ell)\varphi(\ell)g_0(\ell)}{\ell\varphi_{\frac12}(\ell)}
    + \frac{(5 + \ve 2^\ve)2^\ve}{2\sqrt{X}}\sum_{  \ell\le \frac{X}{10.9}}
  \frac{\mu^2(\ell)\varphi(\ell)}{\ell\varphi_{\frac12}(\ell)}\\
  \leq&\frac{ \sage{Upper(THx_aux11,digits)}\times 4.1\times2^\ve}{\sqrt{X}} \sqrt{\frac{X}{10.9}}
    + \frac{\sage{Upper(THxxx_aux11,digits)}\cdot (5 + \ve 2^\ve)2^\ve}{2\sqrt{X}}\sqrt{\frac{X}{10.9}}\\
    \leq&\  \bigg(\frac{\sage{Upper(THx_aux11,digits)}\times 4.1}{\sqrt{10.9}}
    + \frac{\sage{Upper(THxxx_aux11,digits)}\times 5 }{2\sqrt{10.9}}\bigg) 2^\ve
     + \frac{\sage{Upper(THxxx_aux11,digits)}}{2\sqrt{10.9}}\cdot\ve 2^{2\ve}\leq
     \sage{Upper(THx_aux11*41/10/sqrt(109/10)+THxxx_aux11*5/2/sqrt(109/10),digits)}\cdot 2^\ve
     + \sage{Upper(THxxx_aux11/2/sqrt(109/10),digits)}\cdot\ve 2^{2\ve},
     \end{align*}
     where we have used Lemma \ref{auxnu2}$\mathbf{(i)}\&\mathbf{(ii)}$ with $X$ replaced by $X/10.9$. Likewise, with $X$ replaced by $X/47$, 
    \begin{align*}
 A_2(X)&= \frac{4.1\cdot 2^\ve}{\sqrt{X}}\sum_{ \ell\le \frac{X}{47}}
  \frac{\mu^2(\ell)\varphi(\ell)g_0(\ell)}{\ell\varphi_{\frac12}(\ell)}
    +
    \frac{(5 + \ve 2^\ve)2^\ve}{2\sqrt{X}}\sum_{  \ell\le \frac{X}{47}}
  \frac{\mu^2(\ell)\varphi(\ell)}{\ell\varphi_{\frac12}(\ell)}\\
  &\leq \bigg(\frac{\sage{Upper(THx_aux11,digits)}\times 4.1}{\sqrt{47}}
    + \frac{\sage{Upper(THxxx_aux11,digits)}\times 5 }{2\sqrt{47}}\bigg) 2^\ve
     + \frac{\sage{Upper(THxxx_aux11,digits)}}{2\sqrt{47}}\cdot\ve 2^{2\ve} \leq
      \sage{Upper(THx_aux11*41/10/sqrt(47)+THxxx_aux11*5/2/sqrt(47),digits)}\cdot 2^\ve
     + \sage{Upper(THxxx_aux11/2/sqrt(47),digits)}\cdot\ve 2^{2\ve}.
      \end{align*}
  Finally, by recalling \eqref{crux}, the bound given by Corollary \ref{auxnu1}, and combining it with the bounds for $A_1(X)$ and $A_2(X)$, $A(X)$ is bounded by
    \begin{align*}
&
  \ \frac{\sage{Upper(ax_1,digits)}}{X^\ve}+\frac{\sage{Upper(ax_2,digits)}}{X^{2\ve}}(\sage{Upper(THx_aux11*41/10/sqrt(109/10)+THxxx_aux11*5/2/sqrt(109/10),digits)}\cdot 2^\ve
     + \sage{Upper(THxxx_aux11/2/sqrt(109/10),digits)}\cdot\ve 2^{2\ve})+\frac{\sage{Trunc(subs,digits)}}{X^{2\ve}}
     (\sage{Upper(THx_aux11*41/10/sqrt(47)+THxxx_aux11*5/2/sqrt(47),digits)}\cdot 2^\ve
     + \sage{Upper(THxxx_aux11/2/sqrt(47),digits)}\cdot\ve 2^{2\ve}
)\\
&\qquad\qquad+ \frac{0.03\cdot \1_{X\ge 10^{12}}}{X^{2\ve}}\bigg(\sage{Upper(loglemma1_1,digits)}\cdot\log\bigg(\frac{\log X}{12\log 10}\bigg)+\sage{Upper(loglemma1_2,digits)}\bigg)\\
\leq&\ 
 \frac{\sage{Upper(3/100*loglemma1_1,digits)}\cdot \1_{X\ge 10^{12}}}{X^{2\ve}}\log\bigg(\frac{\log X}{12\log 10}\bigg)+\frac{\sage{Upper(3/100*loglemma1_2+ax_1,digits)}}{X^\ve}\\
 &\qquad\qquad+(\sage{Upper(ax_2*THx_aux11*41/10/sqrt(109/10)+THxxx_aux11*5/2/sqrt(109/10),digits)}+\sage{Upper(subs*THx_aux11*41/10/sqrt(47)+THxxx_aux11*5/2/sqrt(47),digits)})\cdot\frac{2^\ve}{X^{2\ve}}
     + (\sage{Upper(ax_2*THxxx_aux11/2/sqrt(109/10),digits)}+\sage{Upper(subs*THxxx_aux11/2/sqrt(47),digits)})\cdot\frac{\ve 2^{2\ve}}{X^{2\ve}},
  \end{align*}
  which gives the result.
\end{proof}

\begin{lem} \label{B_est} For any $X\geq 1$ and any $\ve\in[0,1/10]$,
\begin{align*}
  B(X) \leq&   \frac{G(X;1+\ve)}{\ve^2\zeta(1+\ve)^2}-\frac{G(X/10.9;1+\ve)}{\ve^2\zeta(1+\ve)^2}
  \\&\qquad+\frac{1}{X^{2\ve}}(\sage{Upper((1+1/c)*(3/100)^2*loglemma2,digits)}+
   \sage{Upper((25/4*THx1+5*41/10*THx2+(41/10)^2*TH1)*(c+1)/109*10,digits)}\cdot 2^{2\ve}
   +\sage{Upper((10/4*THx1+41/10*THx2)*(c+1)/109*10,digits)}\cdot\ve 2^{3\ve}
  +\sage{Upper(THx1*(c+1)/4/109*10,digits)}\cdot\ve^22^{4\ve}).
\end{align*}
\end{lem}

\begin{proof} We divide the proof into three parts.\\
\noindent\phantom{$\bullet$}$\bullet$\ Bounding $B(X)$, first reduction.
When $\ell > X/10.9$, Lemma~\ref{m1e} tells us that
\begin{equation*}
-\frac{1}{\ve\zeta(1+\ve)}\frac{\ell^{1+\ve}}{\varphi_{1+\ve}(\ell)}
\le \frac{\Delta_\ell(X/\ell,\ve)}{(X/\ell)^\ve}\le 0.
\end{equation*}
Thus
\begin{align}\label{B1}
&\frac{1}{X^{2\ve} } \sum_{\frac{X}{10.9}<\ell\le X}
  \frac{\mu^2(\ell)\varphi(\ell)}{\ell^{2}}\Delta_\ell^2\bigg(\frac{X}{\ell},\ve\bigg)
  \leq
  \frac{1}{\ve^2\zeta(1+\ve)^2 } \sum_{\frac{X}{10.9}<\ell\le X}
  \frac{\mu^2(\ell)\varphi(\ell)}{\varphi_{1+\ve}(\ell)^2}\nonumber\\
  &\qquad\quad\leq
  \frac{1}{\ve^2\zeta(1+\ve)^2 } \sum_{\frac{X}{10.9}<\ell\le X}
  \frac{\mu^2(\ell)}{\varphi_{1+\ve}(\ell)}= \frac{G(X;1+\ve)}{\ve^2\zeta(1+\ve)^2}-\frac{G(X/10.9;1+\ve)}{\ve^2\zeta(1+\ve)^2}.
  \end{align}
 
\noindent\phantom{$\bullet$}$\bullet$\ Bounding $B(X)$, second treatment. It remains to bound
\begin{equation}  \label{defC}
  C(X) = \frac{1}{X^{2\ve} } \sum_{\ell\le \frac{X}{10.9}}
  \frac{\mu^2(\ell)\varphi(\ell)}{\ell^{2}}\Delta_\ell^2\bigg(\frac{X}{\ell},\ve\bigg).
\end{equation}
We use Bohr's inequality
$(a+b)^2\le (1+c^{-1})(a^2+cb^2)$ valid for $c>0$ on the bound given by Lemma~\ref{m1e} to
infer that, for any $c>0$,
\begin{align*}
  C(X)
  \le&
  \frac{(1+c^{-1})0.03^2\cdot \1_{X\ge 10^{12}}}{X^{2\ve}}L_2(X)+
  \frac{2^{2\ve}(c+1)(5+\ve 2^\ve)^2}{4X^{1+2\ve}}
  \sigma_1\bigg(\frac{X}{10.9}\bigg)
    \\&\qquad+
  \frac{2^{1+2\ve}\cdot 4.1(c+1)(5+\ve 2^\ve)}{2X^{1+2\ve}}
  \sigma_2\bigg(\frac{X}{10.9}\bigg)  +
  \frac{2^{2\ve}(c+1)\cdot 4.1^2}{X^{1+2\ve}}
  \Sigma_1\bigg(\frac{X}{10.9}\bigg),
\end{align*}
where $L_2$ is defined in Corollary~\ref{auxnu} and $\sigma_1,\sigma_2,\Sigma_1$ are defined in Lemma \ref{Aux1}$\mathbf{(i)}$,$\mathbf{(ii)}$ $\&$$\mathbf{(iii)}$, respectively. Precisely, by applying Corollary~\ref{auxnu} and Lemma \ref{Aux1}$\mathbf{(i)}$,$\mathbf{(ii)}$ $\&$$\mathbf{(iii)}$, we obtain that, for any $c>0$,
\begin{align*}
  &C(X)
  \le
 \frac{1}{X^{2\ve}}\bigg( (1+c^{-1})0.03^2\times\sage{Upper(loglemma2,digits)}+
  \frac{ \sage{Upper(THx1,digits)}\cdot 2^{2\ve}(c+1)(5+\ve 2^\ve)^2}{4\times 10.9}
    \\&\qquad\qquad+
  \frac{4.1\times \sage{Upper(THx2,digits)}\cdot 2^{2\ve} (c+1)(5+\ve 2^\ve)}{10.9}
  +
  \frac{4.1^2\times\sage{Upper(TH1,digits)}\cdot 2^{2\ve}(c+1)}{10.9}\bigg)\ \ =\\
  &\frac{1}{X^{2\ve}}\bigg( \frac{(c+1) 0.03^2\times\sage{Upper(loglemma2,digits)}}{c}\\
  &\qquad\quad+
  \frac{(25/4\times \sage{Upper(THx1,digits)}+
  5\times 4.1\times \sage{Upper(THx2,digits)}+4.1^2\times\sage{Upper(TH1,digits)})(c+1)2^{2\ve}}{10.9}\\
  &\qquad\qquad+\frac{( 10/4\times\sage{Upper(THx1,digits)}+4.1\times \sage{Upper(THx2,digits)})(c+1)\ve 2^{3\ve}}{10.9}
  +\frac{ \sage{Upper(THx1,digits)}(c+1)\ve^22^{4\ve}}{4\cdot 10.9}\bigg).
\end{align*} 

\noindent\phantom{$\bullet$}$\bullet$\ Bounding $B(X)$, final treatment. By choosing $c=\sage{c}$, we obtain
\begin{equation}\label{Cbound}
  C(X)
  \le\frac{1}{X^{2\ve}}(\sage{Upper((1+1/c)*(3/100)^2*loglemma2,digits)}+
   \sage{Upper((25/4*THx1+5*41/10*THx2+(41/10)^2*TH1)*(c+1)/109*10,digits)}\cdot 2^{2\ve}
   +\sage{Upper((10/4*THx1+41/10*THx2)*(c+1)/109*10,digits)}\cdot\ve 2^{3\ve}
  +\sage{Upper(THx1*(c+1)/4/109*10,digits)}\cdot\ve^22^{4\ve}).
\end{equation} 
The result is concluded by combining and \eqref{B1} and \eqref{Cbound}.
\end{proof}
\begin{proofbold}[Theorem~\ref{Mainbis}] By combining the estimation \eqref{exxp} with lemmas \ref{A_est}$\&$\ref{B_est}, we obtain that $S_\ve(X)$ is bounded by
  \begin{align*}
 &\frac{S_0(X)}{X^{2\ve}}+2\ve\bigg(\frac{\sage{Upper(3/100*loglemma1_1,digits)}\cdot \1_{X\ge 10^{12}}}{X^{2\ve}}\log\bigg(\frac{\log X}{12\log 10}\bigg)+\sage{Upper(ax_2*THx_aux11*41/10/sqrt(109/10)+THxxx_aux11*5/2/sqrt(109/10)+subs*THx_aux11*41/10/sqrt(47)+THxxx_aux11*5/2/sqrt(47),digits)}\frac{2^\ve}{X^{2\ve}}
     + \sage{Upper(ax_2*THxxx_aux11/2/sqrt(109/10)+subs*THxxx_aux11/2/sqrt(47),digits)}\frac{\ve 2^{2\ve}}{X^{2\ve}}\bigg)
\\&\qquad+\frac{2\ve\ \sage{Upper(3/100*loglemma1_2+ax_1,digits)}}{X^\ve}+
 \frac{\ve^2}{X^{2\ve}}(\sage{Upper((1+1/c)*(3/100)^2*loglemma2,digits)}+
   \sage{Upper((25/4*THx1+5*41/10*THx2+(41/10)^2*TH1)*(c+1)/109*10,digits)}\cdot 2^{2\ve}
   +\sage{Upper((10/4*THx1+41/10*THx2)*(c+1)/109*10,digits)}\cdot\ve 2^{3\ve}
  +\sage{Upper(THx1*(c+1)/4/109*10,digits)}\cdot\ve^22^{4\ve})\\
  &\qquad\qquad\qquad+\frac{2}{\zeta(1+\ve)}-\frac{G(X/10.9,1+\ve)}{\zeta(1+\ve)^2}.
\end{align*}
Rigorous interval arithmetic calculations show that for any $X\geq 10.9\cdot 10^8$ and $\ve\in[0,1/10]$,
\begin{equation*}
G(X/10.9,1+\ve)\geq G(10^8,1+1/10)\geq \sage{Trunc(value1,digits)},
\end{equation*}
thus, by Lemma \ref{zeta}, 
\begin{equation*}
\frac{2}{\zeta(1+\ve)}-\frac{G(X/10.9,1+\ve)}{\zeta(1+\ve)^2}\leq 2\ve-\frac{\sage{Trunc(value1,digits)}\cdot\ve^2\ \1_{X\geq 10.9\cdot 10^8}}{e^{2\gamma\ve}}.
\end{equation*}
Note that the function $\ve\in[0,1/10]\mapsto2\ve-\sage{Trunc(value1,digits)}\ve^2e^{-2\gamma\ve}\ \1_{X\geq 10.9\cdot 10^8}$ is increasing.

On the other hand, observe that
\begin{align*}
&\frac{\sage{Upper(3/100*loglemma1_1,digits)}\cdot \1_{X\ge 10^{12}}}{X^{2\ve}}\log\bigg(\frac{\log X}{12\log 10}\bigg)\\
&\qquad\leq
\frac{\sage{Upper(3/100*loglemma1_1,digits)}\cdot \1_{10^{12}\leq X<e10^{12}}}{X^{2\ve}}+\frac{\sage{Upper(3/100*loglemma1_1,digits)}\cdot \1_{X\ge e10^{12}}}{X^{2\ve}}\log\bigg(\frac{\log X}{12\log 10}\bigg)\\
&\qquad\qquad\leq
\frac{\sage{Upper(3/100*loglemma1_1,digits)}\cdot \1_{X\ge 10^{12}}}{X^{2\ve}}+\frac{\sage{Upper(3/100*loglemma1_1,digits)}\cdot \1_{X\ge e10^{12}}}{X^{2\ve}}\log\bigg(\frac{\log X}{12\log 10}\bigg).
\end{align*}
Furthermore, from the inequality $1+t\leq e^t$, $t\in\mathbb{R}$, we deduce that $t/e^t<1$ for all $t\geq 0$. Thereupon, for any $\ve\geq 0$ and $X\geq 1$,
\begin{equation*}
\frac{2\ve\log X}{X^{2\ve}}=\frac{2\ve\log X}{e^{2\ve\log X}}<1
\end{equation*}
Hence,
\begin{align*}
\frac{2\ve\ \1_{X\geq e10^{12}}}{X^{2\ve}}\log\bigg(\frac{\log X}{12\log 10}\bigg)&\leq
\frac{\1_{X\geq e10^{12}}}{\log X}\log\bigg(\frac{\log X}{12\log 10}\bigg)\\
&\qquad\leq\frac{\1_{X\geq e10^{12}}}{\log(e10^{12}) }
\leq\sage{Upper(1/(1+12*log(10)),digits)}\cdot \1_{X\geq 10^{12}}
\end{align*}
since the function $t\geq e10^{12}\mapsto\log(\log t/\log(10^{12}))/\log t$ is decreasing.

Therefore, by recalling Theorem \ref{Main}, we obtain 
 \begin{align}
  S_\ve(X)\label{TotEst}\nonumber
   \leq&  \frac{S_0(X)}{X^{2\ve}}+
   \sage{Upper(3/100*loglemma1_1,digits)}\times\sage{Upper(1/(1+12*log(10)),digits)}\cdot \1_{X\geq 10^{12}}\nonumber\\
&+2\ve\bigg(\frac{\sage{Upper(3/100*loglemma1_1,digits)}\ \1_{X\ge 10^{12}}}{X^{2\ve}}+\frac{\sage{Upper(3/100*loglemma1_2+ax_1,digits)}}{X^\ve}+\frac{\sage{Upper(ax_2*THx_aux11*41/10/sqrt(109/10)+THxxx_aux11*5/2/sqrt(109/10)+subs*THx_aux11*41/10/sqrt(47)+THxxx_aux11*5/2/sqrt(47),digits)}\cdot 2^\ve}{X^{2\ve}}
     + \frac{\sage{Upper(ax_2*THxxx_aux11/2/sqrt(109/10)+subs*THxxx_aux11/2/sqrt(47),digits)}\cdot\ve 2^{2\ve}}{X^{2\ve}}\bigg)
\nonumber\\&\ +
 \frac{\ve^2}{X^{2\ve}}(\sage{Upper((1+1/c)*(3/100)^2*loglemma2,digits)}+
   \sage{Upper((25/4*THx1+5*41/10*THx2+(41/10)^2*TH1)*(c+1)/109*10,digits)}\cdot 2^{2\ve}
   +\sage{Upper((10/4*THx1+41/10*THx2)*(c+1)/109*10,digits)}\cdot\ve 2^{3\ve}
  +\sage{Upper(THx1*(c+1)/4/109*10,digits)}\cdot\ve^22^{4\ve})\nonumber\\
  &\ \ +2\ve-\frac{\sage{Trunc(value1,digits)}\ve^2\cdot \1_{X\geq 10.9\cdot 10^8}}{e^{2\gamma\ve}}.
\end{align}

We will derive now two bounds for $S_\ve(X)$, depending on the range of the variables $\ve$ and $T$. Firstly, by the three inequalities
\begin{equation*}
\frac{\ve}{X^{\ve}}\leq\frac{1}{\log X},\quad\frac{2\ve}{X^{2\ve}}\leq\frac{1}{\log X},\quad\frac{1}{X^{2\ve}}\leq 1,
\end{equation*}
we obtain, 
 \begin{align*}
  S_\ve(X)
   \leq& S_0(X)+
   \sage{Upper(3/100*loglemma1_1,digits)}\times\sage{Upper(1/(1+12*log(10)),digits)}\cdot \1_{X\geq 10^{12}}
   +2\ve-\frac{\sage{Trunc(value1,digits)}\cdot\ve^2\ \1_{X\geq 10.9\cdot 10^8}}{e^{2\gamma\ve}}\nonumber\\
&+\min\bigg\{2\ve,\frac{1}{\log X}\bigg\}\bigg(\sage{Upper(3/100*loglemma1_1,digits)}\cdot \1_{X\ge 10^{12}}+\sage{Upper(ax_2*THx_aux11*41/10/sqrt(109/10)+THxxx_aux11*5/2/sqrt(109/10)+subs*THx_aux11*41/10/sqrt(47)+THxxx_aux11*5/2/sqrt(47),digits)}\cdot 2^\ve+
\sage{Upper(ax_2*THxxx_aux11/2/sqrt(109/10)+subs*THxxx_aux11/2/sqrt(47),digits)}\cdot\ve 2^{2\ve}\bigg)
\\&+\min\bigg\{\ve,\frac{1}{\log X}\bigg\}\cdot 2\times\sage{Upper(3/100*loglemma1_2+ax_1,digits)}\\
&+\min\bigg\{2\ve,\frac{1}{\log X}\bigg\}\frac{\ve}{2}(\sage{Upper((1+1/c)*(3/100)^2*loglemma2,digits)}+
   \sage{Upper((25/4*THx1+5*41/10*THx2+(41/10)^2*TH1)*(c+1)/109*10,digits)}\cdot 2^{2\ve}
   +\sage{Upper((10/4*THx1+41/10*THx2)*(c+1)/109*10,digits)}\cdot\ve 2^{3\ve}
  +\sage{Upper(THx1*(c+1)/4/109*10,digits)}\cdot\ve^22^{4\ve})
\end{align*}
By Theorem \ref{Main}, we have that
\begin{equation}\label{Theta}
S_0(X)\leq\Theta(T,T_0):=v\text{ if }T\leq X<T_0,
\end{equation}
for 
\begin{align*}(T,T_0;v)\in&\{(422,10^8;\sage{Trunc(Bound0,digits)}),(10^8,10^{12};\sage{Upper(Bound1,digits)}),\\
&\quad\quad(10^{12},10^{16};\sage{Upper(Bound2,digits)}),(10^{16},10^{33};\sage{Upper(Bound3,digits)}),\\
&\qquad\qquad(10^{33},\infty;\sage{Upper(Bound4,digits)})\}. 
\end{align*}

Let $\ve\leq1/R<25/4$. If we assume that $T\leq X<T_0$ for the indicated values above, then
 \begin{equation}\label{1stEst}
  \max_{\substack{\ve\in[0,1/R]\\T\leq X< T_0}}S_\ve(X)
   \leq\Omega_1[R;T,T_0]
\end{equation}
where
 \begin{align*}
  \Omega_1[R;T,T_0]=&\ \Theta(T,T_0)+
   \sage{Upper(3/100*loglemma1_1,digits)}\times\sage{Upper(1/(1+12*log(10)),digits)}\cdot \1_{T_0> 10^{12}}
   +\frac{2}{R}-\frac{\sage{Trunc(value1,digits)}\cdot \1_{T\geq 10.9\cdot 10^8}}{R^2e^{2\gamma/R}}\nonumber\\
&+\min\bigg\{\frac{2}{R},\frac{1}{\log T}\bigg\}\bigg(\sage{Upper(3/100*loglemma1_1,digits)}\cdot \1_{T_0> 10^{12}}+\sage{Upper(ax_2*THx_aux11*41/10/sqrt(109/10)+THxxx_aux11*5/2/sqrt(109/10)+subs*THx_aux11*41/10/sqrt(47)+THxxx_aux11*5/2/sqrt(47),digits)}\cdot 2^{1/R}+
\sage{Upper(ax_2*THxxx_aux11/2/sqrt(109/10)+subs*THxxx_aux11/2/sqrt(47),digits)}\cdot\frac{2^{2/R}}{R}\bigg)\\
&
+\min\bigg\{\frac{1}{R},\frac{1}{\log T}\bigg\}\cdot 2\times\sage{Upper(3/100*loglemma1_2+ax_1,digits)}\\
+\min\bigg\{&\frac{2}{R},\frac{1}{\log T}\bigg\}\frac{1}{2R}\bigg(\sage{Upper((1+1/c)*(3/100)^2*loglemma2,digits)}+
   \sage{Upper((25/4*THx1+5*41/10*THx2+(41/10)^2*TH1)*(c+1)/109*10,digits)}\cdot 2^{2/R}
   +\sage{Upper((10/4*THx1+41/10*THx2)*(c+1)/109*10,digits)}\cdot\frac{2^{3/R}}{R}
  +\sage{Upper(THx1*(c+1)/4/109*10,digits)}\cdot\frac{2^{4/R}}{R^2}\bigg),
\end{align*}
and $\Theta(T,T_0)$ is defined in \eqref{Theta}. 

Alternatively, when $\ve\in(1/R,1/S]$, $S\geq 25/4$ and $T\leq X< T_0$, we have the following two inequalities
 \begin{equation*}
 \frac{1}{X^\ve}\leq\frac{1}{e^{(\log T)/R}},\quad \frac{1}{X^{2\ve}}\leq\frac{1}{e^{2(\log T)/R}}
  \end{equation*}
and the bound
\begin{equation}\label{2ndEst}
  \max_{\substack{\ve\in(1/R,1/S]\\T\leq X< T_0}}S_\ve(X)
   \leq\Omega_2[R,S;T,T_0],
\end{equation}
where
 \begin{align*} 
&\Omega_2[R,S;T,T_0]=\frac{\Theta(T,T_0)}{e^{2(\log T)/R}}+   \sage{Upper(3/100*loglemma1_1,digits)}\times\sage{Upper(1/(1+12*log(10)),digits)}\cdot \1_{T_0> 10^{12}}+\frac{1}{e^{(\log T)/R}}\frac{2\times\sage{Upper(3/100*loglemma1_2+ax_1,digits)}}{S}
   \\&\quad+\frac{2}{Se^{2(\log T)/R}}\bigg(\sage{Upper(3/100*loglemma1_1,digits)}\cdot \1_{T_0> 10^{12}}+\sage{Upper(ax_2*THx_aux11*41/10/sqrt(109/10)+THxxx_aux11*5/2/sqrt(109/10)+subs*THx_aux11*41/10/sqrt(47)+THxxx_aux11*5/2/sqrt(47),digits)}\cdot 2^{1/S}
     + \frac{\sage{Upper(ax_2*THxxx_aux11/2/sqrt(109/10)+subs*THxxx_aux11/2/sqrt(47),digits)}\cdot 2^{2/S}}{S}\bigg)
\nonumber\\&\qquad+
\frac{1}{S^2e^{2(\log T)/R}}\bigg(\sage{Upper((1+1/c)*(3/100)^2*loglemma2,digits)}+
   \sage{Upper((25/4*THx1+5*41/10*THx2+(41/10)^2*TH1)*(c+1)/109*10,digits)}\cdot 2^{2/S}
   +\frac{\sage{Upper((10/4*THx1+41/10*THx2)*(c+1)/109*10,digits)}\cdot 2^{3/S}}{S}
  +\frac{\sage{Upper(THx1*(c+1)/4/109*10,digits)}\cdot 2^{4/S}}{S^2}\bigg)\nonumber\\
  &\qquad\quad+\frac{2}{S}-\frac{\sage{Trunc(value1,digits)}\cdot \1_{T\geq 10.9\cdot 10^8}}{S^2e^{2\gamma/S}}.
\end{align*}
On combining the bounds \eqref{1stEst} and \eqref{2ndEst}, we have that for any $S\geq 25/4$ and $R>S$,
\begin{equation}\label{FinalEst}
  \max_{\substack{\ve\in[0,1/S]\\T\leq X< T_0}}S_\ve(X)
   \leq\max\{\Omega_1[R;T,T_0],\Omega_2[R,S;T,T_0]\}
\end{equation}
With estimation \eqref{FinalEst} at hand, we can select $S$ and a range $[T,T_0)$ for the variable $X$ ($T_0=\infty$ is permitted) and then run a computation to optimise in $R$. Here we present some of our outcomes, which in turn give the result.

\begin{align*}
  \max_{\substack{\ve\in[0,1/\sage{S01}]\\422\leq X< 10^{8}}}S_\ve(X)
   &\leq\max\{\Omega_1[\sage{R01};422,10^8],\Omega_2[\sage{R01},\sage{S01};422,10^8]\}\\
   &\quad\leq\max\{\sage{Upper(Omega1(R01,422,10^8),digits)},\sage{Upper(Omega2(R01,S01,422,10^8),digits)}\}
   =\sage{max(Upper(Omega1(R01,422,10^8),digits),Upper(Omega2(R01,S01,422,10^8),digits))},\\
  \max_{\substack{\ve\in[0,1/\sage{S12}]\\422\leq X< 10^{8}}}S_\ve(X)
   &\leq\max\{\Omega_1[\sage{R02};422,10^8],\Omega_2[\sage{R02},\sage{S02};422,10^8]\}\\
   &\quad\leq\max\{\sage{Upper(Omega1(R02,422,10^8),digits)}\sage{Upper(Omega2(R02,S02,422,10^8),digits)}\},
   =\sage{max(Upper(Omega1(R02,422,10^8),digits),Upper(Omega2(R02,S02,422,10^8),digits))},\\
  \max_{\substack{\ve\in[0,1/\sage{S03}]\\422\leq X< 10^{8}}}S_\ve(X)
   &\leq\max\{\Omega_1[\sage{R03};422,10^8],\Omega_2[\sage{R03},\sage{S03};422,10^8]\}\\
   &\quad\leq\max\{\sage{Upper(Omega1(R03,422,10^8),digits)}\sage{Upper(Omega2(R03,S03,422,10^8),digits)}\}
   =\sage{max(Upper(Omega1(R03,422,10^8),digits),Upper(Omega2(R03,S03,422,10^8),digits))}.
\end{align*}

\begin{align*}
  \max_{\substack{\ve\in[0,1/\sage{S11}]\\10^8\leq X< 10^{12}}}S_\ve(X)
   &\leq\max\{\Omega_1[\sage{R11};10^8,10^{12}],\Omega_2[\sage{R11},\sage{S11};10^8,10^{12}]\}\\
   &\quad\leq\max\{\sage{Upper(Omega1(R11,10^8,10^(12)),digits)},\sage{Upper(Omega2(R11,S11,10^8,10^(12)),digits)}\}
   =\sage{max(Upper(Omega1(R11,10^8,10^(12)),digits),Upper(Omega2(R11,S11,10^8,10^(12)),digits))},\\
  \max_{\substack{\ve\in[0,1/\sage{S12}]\\10^8\leq X< 10^{12}}}S_\ve(X)
   &\leq\max\{\Omega_1[\sage{R12};10^8,10^{12}],\Omega_2[\sage{R12},\sage{S12};10^8,10^{12}]\}\\
   &\quad\leq\max\{\sage{Upper(Omega1(R12,10^8,10^(12)),digits)}\sage{Upper(Omega2(R12,S12,10^8,10^(12)),digits)}\},
   =\sage{max(Upper(Omega1(R12,10^8,10^(12)),digits),Upper(Omega2(R12,S12,10^8,10^(12)),digits))},\\
  \max_{\substack{\ve\in[0,1/\sage{S13}]\\10^8\leq X< 10^{12}}}S_\ve(X)
   &\leq\max\{\Omega_1[\sage{R13};10^8,10^{12}],\Omega_2[\sage{R13},\sage{S13};10^8,10^{12}]\}\\
   &\quad\leq\max\{\sage{Upper(Omega1(R13,10^8,10^(12)),digits)}\sage{Upper(Omega2(R13,S13,10^8,10^(12)),digits)}\}
   =\sage{max(Upper(Omega1(R13,10^8,10^(12)),digits),Upper(Omega2(R13,S13,10^8,10^(12)),digits))}.
\end{align*}

\begin{align*}
  \max_{\substack{\ve\in[0,1/\sage{S21}]\\10^{12}\leq X<10^{16}}}S_\ve(X)
   &\leq\max\{\Omega_1[\sage{R21};10^{12},10^{16}],\Omega_2[\sage{R21},\sage{S21};10^{12},10^{16}]\}\\
   &\quad\leq\max\{\sage{Upper(Omega1(R21,10^(12),10^(16)),digits)},\sage{Upper(Omega2(R21,S21,10^(12),10^(16)),digits)}\}
   =\sage{max(Upper(Omega1(R21,10^(12),10^(16)),digits),Upper(Omega2(R21,S21,10^(12),10^(16)),digits))},\\
  \max_{\substack{\ve\in[0,1/\sage{S22}]\\10^{12}\leq X<10^{16}}}S_\ve(X)
   &\leq\max\{\Omega_1[\sage{R22};10^{12},10^{16}],\Omega_2[\sage{R22},\sage{S22};10^{12},10^{16}]\}\\
   &\quad\leq\max\{\sage{Upper(Omega1(R22,10^(12),10^(16)),digits)},\sage{Upper(Omega2(R22,S22,10^(12),10^(16)),digits)}\}
   =\sage{max(Upper(Omega1(R22,10^(12),10^(16)),digits),Upper(Omega2(R22,S22,10^(12),10^(16)),digits))},\\
  \max_{\substack{\ve\in[0,1/\sage{S23}]\\10^{12}\leq X<10^{16}}}S_\ve(X)
   &\leq\max\{\Omega_1[\sage{R23};10^{12},10^{16}],\Omega_2[\sage{R23},\sage{S23};10^{12},10^{16}]\}\\
   &\quad\leq\max\{\sage{Upper(Omega1(R23,10^(12),10^(16)),digits)},\sage{Upper(Omega2(R23,S23,10^(12),10^(16)),digits)}\}
   =\sage{max(Upper(Omega1(R23,10^(12),10^(16)),digits),Upper(Omega2(R23,S23,10^(12),10^(16)),digits))},
   \end{align*}
   
   \begin{align*}
  \max_{\substack{\ve\in[0,1/\sage{S31}]\\10^{16}\leq X<10^{33}}}S_\ve(X)
   &\leq\max\{\Omega_1[\sage{R31};10^{16},10^{33}],\Omega_2[\sage{R31},\sage{S31};10^{16},10^{33}]\}\\
   &\quad\leq\max\{\sage{Upper(Omega1(R31,10^(16),10^(33)),digits)},\sage{Upper(Omega2(R31,S31,10^(16),10^(33)),digits)}\}
   =\sage{max(Upper(Omega1(R31,10^(16),10^(33)),digits),Upper(Omega2(R31,S31,10^(16),10^(33)),digits))},\\
  \max_{\substack{\ve\in[0,1/\sage{S32}]\\10^{16}\leq X< 10^{33}}}S_\ve(X)
   &\leq\max\{\Omega_1[\sage{R32};10^{16},10^{33}],\Omega_2[\sage{R32},\sage{S32};10^{16},10^{33}]\}\\
   &\quad\leq\max\{\sage{Upper(Omega1(R32,10^(16),10^(33)),digits)},\sage{Upper(Omega2(R32,S32,10^(16),10^(33)),digits)}\}
   =\sage{max(Upper(Omega1(R32,10^(16),10^(33)),digits),Upper(Omega2(R32,S32,10^(16),10^(33)),digits))},\\
  \max_{\substack{\ve\in[0,1/\sage{S33}]\\10^{16}\leq X< 10^{33}}}S_\ve(X)
   &\leq\max\{\Omega_1[\sage{R33};10^{16},10^{33}],\Omega_2[\sage{R33},\sage{S33};10^{16},10^{33}]\}\\
   &\quad\leq\max\{\sage{Upper(Omega1(R33,10^(16),10^(33)),digits)},\sage{Upper(Omega2(R33,S33,10^(16),10^(33)),digits)}\}
   =\sage{max(Upper(Omega1(R33,10^(16),10^(33)),digits),Upper(Omega2(R33,S33,10^(16),10^(33)),digits))}.
\end{align*}

\begin{align*}
  \max_{\substack{\ve\in[0,1/\sage{S41}]\\X\geq 10^{33}}}S_\ve(X)
   &\leq\max\{\Omega_1[\sage{R41};10^{33},\infty],\Omega_2[\sage{R41},\sage{S41};10^{33},\infty]\}\\
   &\quad\leq\max\{\sage{Upper(Omega1(R41,10^(33),10^(50)),digits)},\sage{Upper(Omega2(R41,S41,10^(33),10^(50)),digits)}\}
   =\sage{max(Upper(Omega1(R41,10^(33),10^(50)),digits),Upper(Omega2(R41,S41,10^(33),10^(50)),digits))},\\
  \max_{\substack{\ve\in[0,1/\sage{S42}]\\X\geq 10^{33}}}S_\ve(X)
   &\leq\max\{\Omega_1[\sage{R42};10^{33},\infty],\Omega_2[\sage{R42},\sage{S42};10^{33},\infty]\}\\
   &\quad\leq\max\{\sage{Upper(Omega1(R42,10^(33),10^(50)),digits)},\sage{Upper(Omega2(R42,S42,10^(33),10^(50)),digits)}\}
   =\sage{max(Upper(Omega1(R42,10^(33),10^(50)),digits),Upper(Omega2(R42,S42,10^(33),10^(50)),digits))},\\
  \max_{\substack{\ve\in[0,1/\sage{S43}]\\X\geq 10^{33}}}S_\ve(X)
   &\leq\max\{\Omega_1[\sage{R43};10^{33},\infty],\Omega_2[\sage{R43},\sage{S43};10^{33},\infty]\}\\
   &\quad\leq\max\{\sage{Upper(Omega1(R43,10^(33),10^(50)),digits)},\sage{Upper(Omega2(R43,S43,10^(33),10^(50)),digits)}\}
   =\sage{max(Upper(Omega1(R43,10^(33),10^(50)),digits),Upper(Omega2(R43,S43,10^(33),10^(50)),digits))}.
\end{align*}

  \end{proofbold}


\section{Proofs of the auxiliary results}\label{auxiliaries}

\begin{proofbold}[Lemma \ref{Aux1}] Let $h\in\{\mathds{1},g_0,g_0^2\}$. Observe that, for every sufficiently large prime $p$,
\begin{equation*}
f_h(p)=\frac{\varphi(p)h(p)}{\varphi_{\frac{1}{2}}(p)^2}=\frac{p-1}{(\sqrt{p}-1)^2}=1+\frac{2}{\sqrt{p}-1}=1+O\bigg(\frac{1}{\sqrt{p}}\bigg)
\end{equation*}
Note that $h$ is meaningful only at the prime $p=2$. We use Theorem \ref{general} with $q=1$, $(\alpha,\beta)=(0,1/2)$ and the shortcuts $H_{f'}^1(0)=H_{f'}(0)$,  $\Delta_{1}^{\delta}=\Delta^{\delta}$ and $\overline{H}_{f'}^{\phantom{.}1}(-\delta)=\overline{H}_{f'}(-\delta)$. We select $\delta=5/12\in(0,1/2)$. Hence
 \begin{equation}\label{expression}
 \Sigma_1(X)=H_{f'}(0)\ X+O^*\bigg(\frac{(2-\delta)\Delta_{1}^{\delta}\overline{H}_{f'}(-\delta)}{1-\delta}\ X^{1-\delta}\bigg),
 \end{equation} 
 where 
\begin{align*}
 &H_{f'}(0)=\bigg(1-\frac{1-f_h(2)}{2}-\frac{f_h(2)}{2^{2}}  \bigg)\prod_{p\geq 3}\bigg(1+\frac{2p-\sqrt{p}-1}{p^2(\sqrt{p}-1)} \bigg),\\
 &\qquad 1-\frac{1-f_{\mathds{1}}(2)}{2}-\frac{f_{\mathds{1}}(2)}{2^{2}} \leq\sage{Upper(Fx1_2,digits)},\qquad
 1-\frac{1-f_{g_0}(2)}{2}-\frac{f_{g_0}(2)}{2^{2}} \leq\sage{Upper(Fx2_2,digits)},\\
 &\qquad\qquad\qquad\quad 1-\frac{1-f_{g_0^2}(2)}{2}-\frac{f_{g_0^2}(2)}{2^{2}} \leq\sage{Upper(F1_2,digits)},
 \end{align*}
 \begin{align*}
 &\overline{H}_{f'}(-5/12)=\bigg(1+\frac{|1-f_{h}(2)|}{2^{7/12}}+\frac{|f_{h}(2)|}{2^{7/6}}\bigg)\cdot\prod_{p\geq 3}\bigg(1+\frac{2p^{7/12}+\sqrt{p}+1}{p^{7/6}(\sqrt{p}-1)}\bigg),\\
  &\qquad 1+\frac{|1-f_{\mathds{1}}(2)|}{2^{7/12}}+\frac{|f_{\mathds{1}}(2)|}{2^{7/6}} \leq\sage{Upper(Fx1_2_delta2,digits)},\qquad
   1+\frac{|1-f_{g_0}(2)|}{2^{7/12}}+\frac{|f_{g_0}(2)|}{2^{7/6}} \leq\sage{Upper(Fx2_2_delta2,digits)},\\
 &\qquad\qquad\qquad 1+\frac{|1-f_{g_0^2}(2)|}{2^{7/12}}+\frac{|f_{g_0^2}(2)|}{2^{7/6}} \leq\sage{Upper(F1_2_delta2,digits)}.
 \end{align*}
  Recall that $M=10^8$. By definition \eqref{Bk_def} and estimation \eqref{multup}, we obtain
 \begin{align*}
 \prod_{3\leq p\leq M}\bigg(1+\frac{2p-\sqrt{p}-1}{p^2(\sqrt{p}-1)} \bigg)&\in(\sage{Lower(P1_M,digits)},\sage{Upper(P1_M,digits)})\\
\prod_{3\leq p\leq M}\bigg(1+\frac{2p^{7/12}+\sqrt{p}+1}{p^{7/6}(\sqrt{p}-1)}\bigg)
&\in(\sage{Lower(P1_delta2_M,digits)},\sage{Upper(P1_delta2_M,digits)})
\end{align*} 
 Whence
 \begin{align*}
&\prod_{p\geq 3}\bigg(1+\frac{2p-\sqrt{p}-1}{p^2(\sqrt{p}-1)} \bigg)\\
&\qquad\in(\sage{Lower(P1_M,digits)},\sage{Upper(P1_M,digits)})\times\bigg[1, \exp\bigg(\frac{2\cdot B_{3/2}(M)}{1-M^{-1/2}}\bigg)\bigg]\subset(\sage{Lower(P1_M,digits)},\sage{Upper(P1_M*exp(2/(1-M^(-1/2))*B(3/2,M)),digits)}),\nonumber\\
&\prod_{p\geq 3}\bigg(1+\frac{2p^{7/12}+\sqrt{p}+1}{p^{7/6}(\sqrt{p}-1)}\bigg)\\
&\qquad\in(\sage{Lower(P1_delta2_M,digits)},\sage{Upper(P1_delta2_M,digits)})\times\bigg[1, \exp\bigg(\frac{(2+M^{-1/12}+M^{-7/12})\cdot B_{13/12}(M)}{1-M^{-1/2}}\bigg)\bigg]\\
&\qquad\qquad\subset
(\sage{Lower(P1_delta2_M,digits)},\sage{Upper(P1_delta2_M*exp(Cm(delt2)*B(3/2-delt2,M)),digits)}). 
\end{align*} 
Moreover $
\Delta^{5/12}\in(\sage{Lower(A(delt2),digits)},\sage{Upper(A(delt2),digits)})$.
 Thus, from \eqref{expression}, we derive that for any $X>0$,
 \begin{align*}
 \sigma_1(X)&\leq \sage{Upper(Fx1_2,digits)}\times \sage{Upper(P1_M*exp(2/(1-M^(-1/2))*B(3/2,M)),digits)}\ X+19/7\times\sage{Upper(A(delt2),digits)}\times\sage{Upper(Fx1_2_delta2,digits)}\times\sage{Upper(P1_delta2_M*exp(Cm(delt2)*B(3/2-delt2,M)),digits)}\cdot X^{7/12}\\
 &\leq \sage{Upper(Fx1_2*P1_M*exp(2/(1-M^(-1/2))*B(3/2,M)),digits)}\ X+\sage{Upper((2-delt2)/(1-delt2)*A(delt2)*Fx1_2_delta2*P1_delta2_M*exp(Cm(delt2)*B(3/2-delt2,M)),digits)}\cdot X^{7/12},\\
 \sigma_2(X)&\leq \sage{Upper(Fx2_2,digits)}\times \sage{Upper(P1_M*exp(2/(1-M^(-1/2))*B(3/2,M)),digits)}\ X+19/7\times\sage{Upper(A(delt2),digits)}\times\sage{Upper(Fx2_2_delta2,digits)}\times\sage{Upper(P1_delta2_M*exp(Cm(delt2)*B(3/2-delt2,M)),digits)}\cdot X^{7/12}\\
 &\leq \sage{Upper(Fx2_2*P1_M*exp(2/(1-M^(-1/2))*B(3/2,M)),digits)}\ X+\sage{Upper((2-delt2)/(1-delt2)*A(delt2)*Fx2_2_delta2*P1_delta2_M*exp(Cm(delt2)*B(3/2-delt2,M)),digits)}\cdot X^{7/12},\\
 \Sigma_1(X)&\leq \sage{Upper(F1_2,digits)}\times \sage{Upper(P1_M*exp(2/(1-M^(-1/2))*B(3/2,M)),digits)}\ X+19/7\times\sage{Upper(A(delt2),digits)}\times\sage{Upper(F1_2_delta2,digits)}\times\sage{Upper(P1_delta2_M*exp(Cm(delt2)*B(3/2-delt2,M)),digits)}\cdot X^{7/12}\\
 &\leq \sage{Upper(F1_2*P1_M*exp(2/(1-M^(-1/2))*B(3/2,M)),digits)}\ X+\sage{Upper((2-delt2)/(1-delt2)*A(delt2)*F1_2_delta2*P1_delta2_M*exp(Cm(delt2)*B(3/2-delt2,M)),digits)}\cdot X^{7/12}.
 \end{align*} 
 Thereupon, if $X\geq 4\cdot 10^9$, then
 \begin{align}\label{AuxThres1}
 \sigma_1(X)&\leq X(\sage{Upper(Fx1_2*P1_M*exp(2/(1-M^(-1/2))*B(3/2,M)),digits)}+\sage{Upper((2-delt2)/(1-delt2)*A(delt2)*Fx1_2_delta2*P1_delta2_M*exp(Cm(delt2)*B(3/2-delt2,M)),digits)}\times(4\cdot 10^9)^{-5/12})\leq
 \sage{Upper(Fx1_2*P1_M*exp(2/(1-M^(-1/2))*B(3/2,M))+(2-delt2)/(1-delt2)*A(delt2)*Fx1_2_delta2*P1_delta2_M*exp(Cm(delt2)*B(3/2-delt2,M))*T^(-delt2),digits)}\cdot X,\\
 \label{Aux2Thres1}
 \sigma_2(X)&\leq X(\sage{Upper(Fx2_2*P1_M*exp(2/(1-M^(-1/2))*B(3/2,M)),digits)}+\sage{Upper((2-delt2)/(1-delt2)*A(delt2)*Fx2_2_delta2*P1_delta2_M*exp(Cm(delt2)*B(3/2-delt2,M)),digits)}\times(4\cdot 10^9)^{-5/12})\leq
 \sage{Upper(Fx2_2*P1_M*exp(2/(1-M^(-1/2))*B(3/2,M))+(2-delt2)/(1-delt2)*A(delt2)*Fx2_2_delta2*P1_delta2_M*exp(Cm(delt2)*B(3/2-delt2,M))*T^(-delt2),digits)}\cdot X,\\
 \label{Thres1}
 \Sigma_1(X)&\leq X(\sage{Upper(F1_2*P1_M*exp(2/(1-M^(-1/2))*B(3/2,M)),digits)}+\sage{Upper((2-delt2)/(1-delt2)*A(delt2)*F1_2_delta2*P1_delta2_M*exp(Cm(delt2)*B(3/2-delt2,M)),digits)}\times(4\cdot 10^9)^{-5/12})\leq
 \sage{Upper(F1_2*P1_M*exp(2/(1-M^(-1/2))*B(3/2,M))+(2-delt2)/(1-delt2)*A(delt2)*F1_2_delta2*P1_delta2_M*exp(Cm(delt2)*B(3/2-delt2,M))*T^(-delt2),digits)}\cdot X
 \end{align} 
  Finally, observe that, for any positive arithmetic function $F$ and any real number $X>1$,
 \begin{equation}\label{maxx}
 M_F(X):=\frac{1}{X}\sum_{\ell\leq X}F(\ell)= \frac{1}{X}\sum_{\ell\leq [X]}F(\ell)\leq \frac{1}{[X]}\sum_{\ell\leq [X]}F(\ell),
 \end{equation}
 so the maximum of the function $M_F$ is always achieved at integer values of $X$. A numerical inspection for $X\in[1,4\cdot 10^9]$ tells us that $M_{f_h}$ has a maximum at $\ell=42$ with
  \begin{align}\label{AuxMax1}
 \max_{X\in[1,4\cdot 10^9]\cap\mathbb{N}}M_{f_{\mathds{1}}}(X)&=M_{f_{\mathds{1}}}(42)\leq\sage{Upper(Fmx1,digits)}.\\
 \label{Aux2Max1}
 \max_{X\in[1,4\cdot 10^9]\cap\mathbb{N}}M_{f_{g_0}}(X)&=M_{f_{g_0}}(42)\leq\sage{Upper(Fmx2,digits)}.\\
 \label{Max1}
 \max_{X\in[1,4\cdot 10^9]\cap\mathbb{N}}M_{f_{g_0^2}}(X)&=M_{f_{g_0^2}}(42)\leq\sage{Upper(mx1,digits)}.
 \end{align}
 By combining \eqref{AuxThres1} and \eqref{AuxMax1}, \eqref{Aux2Thres1} and \eqref{Aux2Max1}, \eqref{Thres1} and \eqref{Max1}, respectively, we conclude the result.
 \end{proofbold} 
\newline

 
\begin{proofbold}[Lemma \ref{Aux2}] Observe that, for every sufficiently large prime $p$,
\begin{equation*}
f(p)=\frac{\varphi(p)}{p}\bigg[g_0(p)\frac{\sqrt{p}}{\varphi_{\frac{1}{2}}(p)}\bigg]\bigg[g_2(p)\frac{p^\xi}{\varphi_{\xi}(p)}\bigg]
 =1+\frac{ p^{\xi-1/2}+1}{p^\xi-1}=1+O\bigg(\frac{1}{\sqrt{p}}\bigg)
\end{equation*}
Note that $g_0$ and $g_2$ are only meaningful at the prime $p=2$. Thus we may use Theorem \ref{general} with $q=1$, $(\alpha,\beta)=(0,1/2)$ and $\delta=1/3\in(0,1/2)$. Considering \eqref{expression}, we obtain in this case
 \begin{align*}
 H_{f'}(0)&\leq\sage{Upper(F2_2,digits)}\prod_{p\geq 3}\bigg(1+\frac{p^{\xi+1/2}+p-p^\xi-p^{\xi-1/2}}{p^2(p^\xi-1)}\bigg),\\
 \overline{H}_{f'}(-1/3)&\leq\sage{Upper(F2_2_delta,digits)}\prod_{p\geq 3}\bigg(1+\frac{p^{1/6+\xi}+p^\xi+p^{2/3}+p^{\xi-1/2}}{p^{4/3}(p^\xi-1)}\bigg).
 \end{align*} 
where
\begin{align*}
 &\prod_{p\geq 3}\bigg(1+\frac{p^{\xi+1/2}+p-p^\xi-p^{\xi-1/2}}{p^2(p^\xi-1)}\bigg)\\
 &\qquad\in (\sage{Lower(P2,digits)},\sage{Upper(P2,digits)})\times\bigg[1,\exp\bigg(\frac{(1+M^{1/2-\xi})\cdot B_{3/2}(M)}{1-M^{-\xi}}\bigg)\bigg]\subset(\sage{Lower(P2,digits)},\sage{Upper(P2*exp(cm0_2*B(3/2,M)),digits)}),\\
 &\prod_{p\geq 3}\bigg(1+\frac{p^{1/6+\xi}+p^\xi+p^{2/3}+p^{\xi-1/2}}{p^{4/3}(p^\xi-1)}\bigg)\\
 &\qquad\in(\sage{Lower(P2_delta,digits)},\sage{Upper(P2_delta,digits)})
\times\bigg[1,\exp\bigg(\frac{(1+M^{-1/6}+M^{1/2-\xi}+M^{-2/3})\cdot B_{7/6}(M)}{1-M^{-\xi}}\bigg)\bigg]\\
&\qquad \ \subset(\sage{Lower(P2_delta,digits)},\sage{Upper(P2_delta*exp(cmd_2*B(7/6,M)),digits)}).
  \end{align*} 
  Moreover
  \begin{equation}\label{Delta3}
\Delta^{1/3}\in(\sage{Lower(A(delt),digits)},\sage{Upper(A(delt),digits)})
 \end{equation} 
 Thus, by recalling \eqref{expression}, we obtain
 \begin{align*}
 Z_1(X)&\leq\sage{Upper(F2_2,digits)}\times\sage{Upper(P2*exp(cm0_2*B(3/2,M)),digits)}\cdot X+5/2\times\sage{Upper(A(delt),digits)}\times\sage{Upper(F2_2_delta,digits)}\times\sage{Upper(P2_delta*exp(cmd_2*B(7/6,M)),digits)}\cdot X^{2/3}\\
 &\leq X(\sage{Upper(F2_2*P2*exp(cm0_2*B(3/2,M)),digits)}+\sage{Upper(5/2*A(delt)*F2_2_delta*P2_delta*exp(cmd_2*B(7/6,M)),digits)}\cdot X^{-1/3}).
 \end{align*}
Thereupon, for any $X\geq 10^7$,
\begin{align}\label{Thres2}
 Z_1(X)&\leq \sage{Upper(F2_2*P2*exp(cm0_2*B(3/2,M))+5/2*A(delt)*F2_2_delta*P2_delta*exp(cmd_2*B(7/6,M))*T2^-delt,digits)}\cdot X.
 \end{align}
 By recalling \eqref{maxx}, a numerical inspection for $X\in[1,10^7]$ tells us that 
 \begin{equation}\label{Max2}
 \max_{X\in[1,10^7]\cap\mathbb{N}}\frac{Z_1(X)}{X}=\frac{Z_1(6)}{6}\leq\sage{Upper(mx2,digits)}.
 \end{equation}
  By combining \eqref{Thres2} and \eqref{Max2}, we conclude the result.
\end{proofbold}
\newline


\begin{proofbold}[Corollary \ref{Cor1}]
    Let $Y=\min\{D, X/10^{12}\}$. Then
  \begin{equation*}
    \frac{\text{d}}{\text{d}t}\bigg(\frac{1}{\sqrt{t}\log(X/t)}\bigg)=
    -\frac{1}{2t^{3/2}\log (X/t)}+\frac{1}{t^{3/2}\log^2 (X/t)},
  \end{equation*}
  which is negative when $X/t\ge 10^{12}$.
  Hence, by Lemma~\ref{Aux2} and summation by parts, 
  \begin{align*}
   \Sigma_2(X)
    &\le \int_1^Y\frac{\sage{Upper(TH2,digits)} (\log (X/t)-2)}{2\sqrt{t}\log^2(X/t)}dt
    + \frac{\sage{Upper(TH2,digits)} \sqrt{D}}{\log(10^{12})}
    \\&\ \leq\sage{Upper(TH2,digits)}\cdot \sqrt{X}\int_{X/Y}^X  \frac{\log u-2}{2u^{3/2}\log^2u}du
    + \sage{Upper(TH2/log(10)/12,digits)}\cdot \sqrt{D}
    \\&	\ \ \leq\sage{Upper(TH2,digits)}\cdot \sqrt{X}\bigg[\frac{-1}{\sqrt{t}\log t}\bigg|_{X/Y}^X
    + \sage{Upper(TH2/log(10)/12,digits)}\cdot \sqrt{D}\\
    &\ \ \leq
    \sage{Upper(TH2/12/log(10),digits)}\cdot \sqrt{Y} + \sage{Upper(TH2/log(10)/12,digits)}\cdot \sqrt{D}
    \ \ \leq\sage{Upper(TH2/12/log(10)+TH2/log(10)/12,digits)}\cdot \sqrt{D},
  \end{align*}
where we have used that $X/Y\geq 10^{12}$. The result follows readily.
\end{proofbold}
\newline


\begin{proofbold}[Lemma \ref{Aux3}] Observe that, for every sufficiently large prime $p$,
\begin{equation*}
\frac{\varphi(p)}{p}\bigg[g_2(p)\frac{p^\xi}{\varphi_{\xi}(p)}\bigg]^2=\frac{(p-1)p^{2\xi-1}}{(p^\xi-1)^2}=1+\frac{2p^\xi-p^{2\xi-1}-1}{p^{2\xi}-2p^{\xi}+1}=1+O\bigg(\frac{1}{p^{\xi}}\bigg).
\end{equation*}
Note that $g_2$ is only meaningful at the prime $p=2$. Thus $(\alpha,\beta)=(0,\xi)$ and, since $\beta-\alpha>1/2$, we can use the more efficient Theorem \ref{general++}$\mathbf{[B]}$ with $q=1$. Hence
 \begin{equation}\label{expression2}
Z_2(X)=H_{f'}(0)\ X+O^*(3\mathrm{w'}\mathrm{P}\ \sqrt{X}),
 \end{equation} 
where 
\begin{align*}
H_{f'}(0)&=H_{f'}^1(0)=\sage{Upper(F3_2,digits)}\prod_{p\geq 3}\bigg(1+\frac{2p^{1+\xi}-2p^{2\xi}-p+p^{2\xi-1}}{p^2(p^\xi-1)^2}\bigg),\\
\mathrm{w'}&=\mathrm{w'}_{0}^1= 
\bigg(\frac{\sqrt{2}-1}{\sqrt{2}-1+| f(2)-1|}\bigg)\bigg(\mathrm{E}_1^{(1)}+\frac{| f(2)-1|\ \mathrm{E}_1^{(2)}}{\sqrt{2}-1}\bigg)\in(\sage{Lower(W,digits)},\sage{Upper(W,digits)}),\\
\mathrm{P}&=\mathrm{P}_0\leq\sage{Upper(F3_2_spec,digits)}\prod_{p\geq 3}\bigg(1+\frac{|f(p)-1|}{\sqrt{p}-1}\bigg)=\sage{Upper(F3_2_spec,digits)}\prod_{p\geq 3}\bigg(1+\frac{2p^\xi-p^{2\xi-1}-1}{(p^\xi-1)^2(\sqrt{p}-1)}\bigg).
\end{align*} 
Observe that
\begin{align*}
&\prod_{p\geq 3}\bigg(1+\frac{2p^{1+\xi}-2p^{2\xi}-p+p^{2\xi-1}}{p^2(p^\xi-1)^2}\bigg)\\
&\quad\in(\sage{Lower(P3,digits)},\sage{Upper(P3,digits)})\times\bigg[1,\exp\bigg(\frac{(2+M^{\xi-2})\cdot B_{1+\xi}(M)}{(1-M^{-\xi})^2}\bigg)\bigg]\subset(\sage{Lower(P3,digits)},\sage{Upper(P3*exp(cm0_3*B(1+xi,M)),digits)}),\\
&\prod_{p\geq 3}\bigg(1+\frac{2p^\xi-p^{2\xi-1}-1}{(p^\xi-1)^2(\sqrt{p}-1)}\bigg)\\
&\quad\in(\sage{Lower(P3_spec,digits)},\sage{Upper(P3_spec,digits)})\times\bigg[1,\exp\bigg(\frac{2\cdot B_{1/2+\xi}(M)}{(1-M^{-\xi})^2(1-M^{-1/2})}\bigg)\bigg]
\subset(\sage{Lower(P3_spec,digits)},\sage{Upper(P3_spec*exp(cmd_3*B(1/2+xi,M)),digits)}).
\end{align*} 
Thus, from \eqref{expression2}, we obtain
 \begin{align*}
Z_2(X)&\leq\sage{Upper(F3_2,digits)}\times\sage{Upper(P3*exp(cm0_3*B(1+xi,M)),digits)}\cdot X+3\times\sage{Upper(W,digits)}\times\sage{Upper(P3_spec*exp(cmd_3*B(1/2+xi,M)),digits)}\cdot \sqrt{X}\\
&\leq X(\sage{Upper(F3_2*P3*exp(cm0_3*B(1+xi,M)),digits)}+ \sage{Upper(3*W*P3_spec*exp(cmd_3*B(1/2+xi,M)),digits)}\cdot X^{-1/2}),
\end{align*} 
so that, for any $X\geq 10^6$,
\begin{align}\label{Thres3}
Z_2(X)\leq \sage{Upper(F3_2*P3*exp(cm0_3*B(1+xi,M))+3*W*P3_spec*exp(cmd_3*B(1/2+xi,M))*T3^(-1/2),digits)}\cdot X ,
\end{align} 
 By recalling \eqref{maxx}, a quick numerical inspection for $X\in[1,10^6]$ tells us that 
 \begin{equation}\label{Max3}
 \max_{X\in[1,10^6]\cap\mathbb{N}}\frac{Z_2(X)}{X}=\frac{Z_2(2)}{2}\leq\sage{Upper(mx3,digits)}.
 \end{equation}
  By combining \eqref{Thres3} and \eqref{Max3}, we conclude the result.
\end{proofbold}
\newline


\begin{proofbold}[Corollary \ref{Cor2}]
  Let $Y=\min\{D, X/10^{12}\}$. Then
  \begin{equation*}
    \frac{\text{d}}{\text{d}t}\bigg(\frac{1}{t\log(X/t)^2}\bigg)=
    -\frac{1}{t^2\log^2(X/t)}+\frac{2}{t^2\log^3 (X/t)}
  \end{equation*}
  which is negative when $X/t\ge 10^{12}$.
  Hence, by Lemma~\ref{Aux3} and summation by parts,  
  \begin{align*}
    \Sigma_3(X)
    &\le\sage{Upper(TH3,digits)}\ \int_1^Y  \frac{\log (X/t)-2}{t\log^3(X/t)}dt
    + \frac{\sage{Upper(TH3,digits)}}{\log(X/Y)^2}
    \\&\ \le \sage{Upper(TH3,digits)}\ \int_{X/Y}^X  \frac{\log u-1}{u\log^3u}du
    + \frac{\sage{Upper(TH3,digits)}}{\log^2(10^{12})}\\
    &\ \ \le \sage{Upper(TH3,digits)}\bigg[-\frac{1}{\log t}+\frac{1}{\log^2t}\bigg|_{X/Y}^X
    + \sage{Upper(TH3/12^2/log(10)^2,digits)}\\
&\ \ \ \le \frac{\sage{Upper(TH3,digits)}}{\log(X/Y)} -\frac{\sage{Upper(TH3,digits)}}{\log^2(X/Y)}
    + \sage{Upper(TH3/12^2/log(10)^2,digits)}
    \le \sage{Upper(TH3/12/log(10)-TH3/12^2/log(10)^2+TH3/12^2/log(10)^2,digits)}
      \end{align*}
  where we have used that $X/Y\geq 10^{12}$ and that the function $t\mapsto 1/\log t-1/\log^2 t$ is decreasing for $t\geq e^2$. The result follows readily.
\end{proofbold}
\newline

\begin{proofbold}[Lemma \ref{auxnu2}]
   Let $h\in\{\mathds{1},g_0\}$. Observe that, for every sufficiently large prime $p$,
\begin{equation*}
 f_h(p)=\frac{\varphi(p)}{p^{3/2}}\bigg[h(p)\frac{\sqrt{p}}{\varphi_{\frac{1}{2}}(p)}\bigg]= \frac{\sqrt{p}+1}{p}=\frac{1}{\sqrt{p}}+O\bigg(\frac{1}{p}\bigg)
 \end{equation*}
Note that $h$ is only meaningful at the prime $p=2$. Thus we may apply Theorem \ref{general} with $(\alpha,\beta)=(1/2,1)$, $q=1$ and $\delta=1/3\in(0,1/2)$. We obtain 
\begin{equation*}
 \Xi_h(X)=2H_{f'}(0)\sqrt{X}+O^*\bigg(4\Delta_{1}^{1/3}\overline{H}_{f'}(-1/3)\ X^{1/6}\bigg),
 \end{equation*} 
Considering \eqref{expression}, we obtain in this case
 \begin{align*}
 H_{f'}(0)&=\bigg(1-\frac{1-f_h(2)\sqrt{2}}{2}-\frac{f_h(2)\sqrt{2}}{2^2}\bigg)\prod_{p\geq 3}\bigg(1+\frac{p-\sqrt{p}-1}{p^{5/2}}\bigg),\\
 \overline{H}_{f'}(-1/3)&=\bigg(1+\frac{|1-f_h(2)\sqrt{2}|}{2^{2/3}}+\frac{|f_h(2)\sqrt{2}|}{2^{4/3}}\bigg)\prod_{p\geq 3}\bigg(1+\frac{p^{2/3}+\sqrt{p}+1}{p^{11/6}}\bigg).
 \end{align*} 
 where
 \begin{align*}
 1-\frac{1-f_h(2)\sqrt{2}}{2}-\frac{f_h(2)\sqrt{2}}{2^2}&\leq\begin{cases}\sage{Upper(F1xxx_2,digits)}\text{ if }h=\mathds{1},\\
 \sage{Upper(F1x_2,digits)}\text{ if }h=g_0,
 \end{cases}\\
 1+\frac{|1-f_h(2)\sqrt{2}|}{2^{2/3}}+\frac{|f_h(2)\sqrt{2}|}{2^{4/3}}&\leq\begin{cases}\sage{Upper(F1xxx_2_delta,digits)}\text{ if }h=\mathds{1},\\
 \sage{Upper(F1x_2_delta,digits)}\text{ if }h=g_0.
 \end{cases}
 \end{align*}
 Moreover
    \begin{align*}
  \prod_{p\geq 3}\bigg(1+\frac{p-\sqrt{p}-1}{p^{5/2}}\bigg)&\in(\sage{Lower(P1x,digits)},\sage{Upper(P1x,digits)})\times[1,\exp(B_{3/2}[
  M]))
  \subset(\sage{Lower(P1x,digits)},\sage{Upper(P1x*exp(B(3/2,M)),digits)}),\\
 \prod_{p\geq 3}\bigg(1+\frac{p^{2/3}+\sqrt{p}+1}{p^{11/6}}\bigg)&\in(\sage{Lower(P1x_delta,digits)},\sage{Upper(P1x_delta,digits)})\times[1,\exp((1+M^{-1/6}+M^{-2/3})\cdot B_{7/6}(M)))\\
 &\qquad\subset(\sage{Lower(P1x_delta,digits)},\sage{Upper(P1x_delta*exp((1+M^(-1/6)+M^(-2/3))*B(7/6,M)),digits)}).
 \end{align*} 
 All in all, for any $X>0$, by recalling \eqref{Delta3}, we obtain
 \begin{align*}
 \Xi_1(X)&\leq \sqrt{X}\ (2\times\sage{Upper(F1xxx_2,digits)}\times\sage{Upper(P1x*exp(B(3/2,M)),digits)}+ 4\times\sage{Upper(A(delt),digits)}\times\sage{Upper(F1xxx_2_delta,digits)}\times\sage{Upper(P1x_delta*exp((1+M^(-1/6)+M^(-2/3))*B(7/6,M)),digits)}\cdot X^{-1/3})\\
 &\leq\sqrt{X}\ ( \sage{Upper(2*F1xxx_2*P1x*exp(B(3/2,M)),digits)}+\sage{Upper(4*A(delt)*F1xxx_2_delta*P1x_delta*exp((1+M^(-1/6)+M^(-2/3))*B(7/6,M)),digits)}\cdot X^{-1/3}),\\
 \Xi_2(X)&\leq \sqrt{X}\ (2\times\sage{Upper(F1x_2,digits)}\times\sage{Upper(P1x*exp(B(3/2,M)),digits)}+ 4\times\sage{Upper(A(delt),digits)}\times\sage{Upper(F1x_2_delta,digits)}\times\sage{Upper(P1x_delta*exp((1+M^(-1/6)+M^(-2/3))*B(7/6,M)),digits)}\cdot X^{-1/3})\\
 &\leq\sqrt{X}\ ( \sage{Upper(2*F1x_2*P1x*exp(B(3/2,M)),digits)}+\sage{Upper(4*A(delt)*F1x_2_delta*P1x_delta*exp((1+M^(-1/6)+M^(-2/3))*B(7/6,M)),digits)}\cdot X^{-1/3}).
 \end{align*} 
 When $X\geq  10^6$, we thus have
 \begin{align}
 \label{Xi1_est}\Xi_1(X)& \leq\sqrt{X}\ ( \sage{Upper(2*F1xxx_2*P1x*exp(B(3/2,M)),digits)}+\sage{Upper(4*A(delt)*F1xxx_2_delta*P1x_delta*exp((1+M^(-1/6)+M^(-2/3))*B(7/6,M)),digits)}\times (10^6)^{-1/3})\nonumber\\
 &\leq\sage{Upper(2*F1xxx_2*P1x*exp(B(3/2,M))+4*A(delt)*F1xxx_2_delta*P1x_delta*exp((1+M^(-1/6)+M^(-2/3))*B(7/6,M))/(10^6)^delt,digits)}\cdot\sqrt{X}, \\
 \label{Xi2_est}\Xi_2(X)& \leq\sqrt{X}\ ( \sage{Upper(2*F1x_2*P1x*exp(B(3/2,M)),digits)}+\sage{Upper(4*A(delt)*F1x_2_delta*P1x_delta*exp((1+M^(-1/6)+M^(-2/3))*B(7/6,M)),digits)}\times (10^6)^{-1/3})\nonumber\\
 &\leq\sage{Upper(2*F1x_2*P1x*exp(B(3/2,M))+4*A(delt)*F1x_2_delta*P1x_delta*exp((1+M^(-1/6)+M^(-2/3))*B(7/6,M))/(10^6)^delt,digits)}\cdot\sqrt{X}.
 \end{align} 
When $X\in(1,5\cdot 10^6]\cup\mathbb{Z}$, a quick calculation tells us that
\begin{align}
\label{Xi1_max}\max_{X\in(0,5\cdot 10^6]}\frac{\Xi_1(X)}{\sqrt{X}}&=\frac{\Xi_1(978\, 118)}{\sqrt{978\,118}}\leq\sage{Upper(mxxx_aux11,digits)},\\
\label{Xi2_max}\max_{X\in(0,5\cdot 10^6]}\frac{\Xi_2(X)}{\sqrt{X}}&=\frac{\Xi_2(478\,671)}{\sqrt{478\,671}}\leq\sage{Upper(mx_aux11,digits)}.
\end{align}
By combining \eqref{Xi1_est} with \eqref{Xi1_max} and then \eqref{Xi2_est} with \eqref{Xi2_max}, we conclude the result.
 \end{proofbold}
\newline


\begin{proofbold}[Lemma \ref{auxf2}] 
Observe that, for any sufficiently large prime $p$,
\begin{equation*}
f(p)=\frac{\varphi(p)}{p^2}\bigg[g_2(p)\frac{p^\xi}{\varphi_{\xi}(p)}\bigg]=\frac{1}{p}+\frac{p-p^\xi}{p^2(p^\xi-1)}=\frac{1}{p}+O\bigg(\frac{1}{p^{1+\xi}}\bigg).
\end{equation*}
Thus in the language of Theorem \ref{general}, we have $(\alpha,\beta)=(1,1+\xi)$ with $\xi>1/2$. Thus, we can use Theorem \ref{general++}$\mathbf{[A]}$ with $q=1$  and obtain
 \begin{align*}
K_1(X)= F_1(X)+O^*\bigg( \frac{\mathrm{w}_1\ \mathrm{P}_1}{\sqrt{X}}\bigg),
 \end{align*} 
 where
  \begin{equation*}
 F_1(X)=H_{f}(0)\bigg(\log X+\sum_{p}\frac{(1-(p-2)f(p))\log p}{(f(p)+1)(p-1)}+\gamma\bigg),
 \end{equation*}
 \begin{align*}
 H_{f}(0)&=\left(1-\frac{1-2f(2)}{2}-\frac{f(2)}{2}  \right)\prod_{p\geq 3}\bigg(1+\frac{p^{2-\xi}-2p-1}{p^{3-\xi}(p^\xi-1)}\bigg),\\
\mathrm{w}_1&=\sage{Upper(W2_lx,digits)},\\
 \mathrm{P}_1&=\bigg(1+\frac{|2f(2)-1|}{\sqrt{2}-1}\bigg)\prod_{p\geq 3}\bigg(1+\frac{p^{1-\xi}-1}{p^{1-\xi}(p^\xi-1)(\sqrt{p}-1)}\bigg),
 \end{align*}    
Note that
 \begin{align*}
 &1-\frac{1-2f(2)}{2}-\frac{f(2)}{2} \in(\sage{Lower(F2_LX,digits)},\sage{Upper(F2_LX,digits)}),\quad 1+\frac{|2f(2)-1|}{\sqrt{2}-1}\in(\sage{Lower(F2P1_LX,digits)},\sage{Upper(F2P1_LX,digits)}),\\
&\qquad\prod_{3\leq p\leq M}\bigg(1+\frac{p^{2-\xi}-2p-1}{p^{3-\xi}(p^\xi-1)}\bigg)\in(\sage{Lower(P0_LX,digits)},\sage{Upper(P0_LX,digits)}),\\
&\qquad\prod_{3\leq p\leq M}\bigg(1+\frac{p^{1-\xi}-1}{p^{1-\xi}(p^\xi-1)(\sqrt{p}-1)}\bigg)\in(\sage{Lower(P1_LX,digits)},\sage{Upper(P1_LX,digits)}),
 \end{align*}    
 so that
   \begin{align*}
 \prod_{p}\bigg(1+\frac{p^{2-\xi}-2p-1}{p^{3-\xi}(p^\xi-1)}\bigg)\in(\sage{Lower(F2_LX,digits)},\sage{Upper(F2_LX,digits)})\times(\sage{Lower(P0_LX,digits)},\sage{Upper(P0_LX,digits)})\times\bigg[1,\exp\bigg(\frac{B_{1+\xi}(M)}{1-M^{-\xi}}\bigg)\bigg]&\\
 \qquad\subset(\sage{Lower(F2_LX*P0_LX,digits)},\sage{Upper(F2_LX*P0_LX*exp((1-M^(-xi))^(-1)*B(1+xi,M)),digits)}),\phantom{xxxxxxxxxxxxxxxxxxxxxxxxxx}&\\
\prod_{p}\bigg(1+\frac{p^{1-\xi}-1}{p^{1-\xi}(p^\xi-1)(\sqrt{p}-1)}\bigg)\phantom{xxxxxxxxxxxxxxxxxxxxxxxxxxxxxxxxxxxxxx}&\\
\in(\sage{Lower(F2P1_LX,digits)},\sage{Upper(F2P1_LX,digits)})\times(\sage{Lower(P1_LX,digits)},\sage{Upper(P1_LX,digits)})\times\bigg[1,\exp\bigg(\frac{B_{1/2+\xi}(M)}{(1-M^{-\xi})(1-M^{-1/2})}\bigg)\bigg]&\\
\subset(\sage{Lower(F2P1_LX*P1_LX,digits)},\sage{Upper(F2P1_LX*P1_LX*exp((1-M^(-xi))^(-1)*(1-M^(-1/2))^(-1)*B(1/2+xi,M)),digits)}).\phantom{xxxxxxxxxxxxxxxxxxxxxxxxxx}&
 \end{align*}   
 We also have
 \begin{align*}
  \sum_{p\leq M}\frac{(1-(p-2)f(p))\log p}{(f(p)+1)(p-1)}&=\frac{\log 2}{f(2)+1}-\sum_{3<p\leq M}\frac{(p^2-3p^{1+\xi}+2p^\xi)\log p}{(p^\xi(p-1)+p^2(p^\xi-1))(p-1)}\\
  &\qquad\in(\sage{Lower(Sumx-euler_gamma,digits)},\sage{Upper(Sumx-euler_gamma,digits)})
 \end{align*}
 Moreover, by \eqref{Ck_def} and \eqref{sumdown}, we obtain
  \begin{align*}
 \sum_{p>M}\frac{(p^2-3p^{1+\xi}+2p^\xi)\log p}{(p^{2+\xi}-p^2+p^{1+\xi}-p^\xi)(p-1)}&\leq\frac{(1+2M^{\xi-2})\cdot C_{1+\xi}(M)}{(1-M^{-\xi}-M^{-2})(1-M^{-1})},
 \end{align*}
 so that
  \begin{align*}
 &\gamma+ \sum_{p\leq M}\frac{(1-(p-2)f(p))\log p}{(f(p)+1)(p-1)}\\
 &\qquad\in\gamma+(\sage{Lower(Sumx-euler_gamma,digits)},\sage{Upper(Sumx-euler_gamma,digits)})+\bigg[-\frac{(1+2M^{\xi-2})\cdot C_{1+\xi}(M)}{(1-M^{-\xi}-M^{-2})(1-M^{-1})},0\bigg]\\
 &\qquad\qquad\subset(\sage{Lower(Sumx-(1+2*M^(xi-2))/(1-M^(-xi)-M^(-2))*C(1+xi,M),digits)},\sage{Upper(Sumx,digits)}).
 \end{align*}
 Therefore
\begin{align*}
&K_1(X)\in(\sage{Lower(F2_LX*P0_LX,digits)},\sage{Upper(F2_LX*P0_LX*exp((1-M^(-xi))^(-1)*B(1+xi,M)),digits)})\cdot(\log X+(\sage{Lower(Sumx,digits)},\sage{Upper(Sumx,digits)}))+O^*\bigg( \frac{\sage{Upper(W2_lx,digits)}\times\sage{Upper(F2P1_LX*P1_LX*exp((1-M^(-xi))^(-1)*(1-M^(-1/2))^(-1)*B(1/2+xi,M)),digits)}}{\sqrt{X}}\bigg)\phantom{xxxxxxxxxxxxx}\\
\qquad&\subset(\sage{Lower(F2_LX*P0_LX,digits)},\sage{Upper(F2_LX*P0_LX*exp((1-M^(-xi))^(-1)*B(1+xi,M)),digits)})\cdot\log X\\
&\quad+\bigg(\sage{Lower(F2_LX*P0_LX,digits)}\times\sage{Lower(Sumx,digits)}-\frac{\sage{Upper(W2_lx,digits)}\times\sage{Upper(F2P1_LX*P1_LX*exp((1-M^(-xi))^(-1)*(1-M^(-1/2))^(-1)*B(1/2+xi,M)),digits)}}{\sqrt{X}},\sage{Upper(F2_LX*P0_LX*exp((1-M^(-xi))^(-1)*B(1+xi,M)),digits)}\times\sage{Upper(Sumx,digits)}+\frac{\sage{Upper(W2_lx,digits)}\times\sage{Upper(F2P1_LX*P1_LX*exp((1-M^(-xi))^(-1)*(1-M^(-1/2))^(-1)*B(1/2+xi,M)),digits)}}{\sqrt{X}}\bigg).
 \end{align*} 
 If we assume that $X\geq 2\cdot 10^8$, then
   \begin{align*}
K_1(X)&\in(\sage{Lower(F2_LX*P0_LX,digits)},\sage{Upper(F2_LX*P0_LX*exp((1-M^(-xi))^(-1)*B(1+xi,M)),digits)})\cdot\log X+\bigg(\sage{Lower(F2_LX*P0_LX*Sumx,digits)}-\frac{\sage{Upper(W2_lx*F2P1_LX*P1_LX*exp((1-M^(-xi))^(-1)*(1-M^(-1/2))^(-1)*B(1/2+xi,M)),digits)}}{\sqrt{2\cdot 10^8}},\sage{Upper(F2_LX*P0_LX*exp((1-M^(-xi))^(-1)*B(1+xi,M))*Sumx,digits)}+\frac{\sage{Upper(W2_lx*F2P1_LX*P1_LX*exp((1-M^(-xi))^(-1)*(1-M^(-1/2))^(-1)*B(1/2+xi,M)),digits)}}{\sqrt{2\cdot 10^8}}\bigg)\\
&\ \subset(\sage{Lower(F2_LX*P0_LX,digits)},\sage{Upper(F2_LX*P0_LX*exp((1-M^(-xi))^(-1)*B(1+xi,M)),digits)})\cdot\log X+(\sage{Lower(F2_LX*P0_LX*Sumx-W2_lx*F2P1_LX*P1_LX*exp((1-M^(-xi))^(-1)*(1-M^(-1/2))^(-1)*B(1/2+xi,M))/sqrt(N),digits)},\sage{Upper(F2_LX*P0_LX*exp((1-M^(-xi))^(-1)*B(1+xi,M))*Sumx+W2_lx*F2P1_LX*P1_LX*exp((1-M^(-xi))^(-1)*(1-M^(-1/2))^(-1)*B(1/2+xi,M))/sqrt(N),digits)}).
 \end{align*} 
 We complete the proof by direct inspection. Indeed, for any $X\geq 1$, 
  \begin{align}\label{analysis}
 K_1(X)-c_1\log X&=K_1([X])-c_1\log([X])-c_1\log\bigg(\frac{X}{[X]}\bigg)\nonumber\\
 &\ \in K_1([X])-c_1\log([X])+\bigg[-c_1\log\bigg(\frac{[X]+1}{[X]}\bigg),0\bigg]\nonumber\\
 &\ \ \subset K_1([X])-c_1\log([X])+\bigg[-\sage{Upper(uc1,digits)}\log\bigg(\frac{[X]+1}{[X]}\bigg),0\bigg].
 \end{align}
Further,
 \begin{equation*}
 K_1([X])-\sage{Upper(F2_LX*P0_LX*exp((1-M^(-xi))^(-1)*B(1+xi,M)),digits)}\cdot\log [X]\leq K_1([X])-c_1\log [X]\leq K_1([X])-\sage{Lower(F2_LX*P0_LX,digits)}\cdot\log [X].
 \end{equation*}	
 Thus for any $X\geq 1$, we calculate and observe that
\begin{align*} \max_{X\in[1,2\cdot 10^8]} \{K_1([X])-\sage{Lower(F2_LX*P0_LX,digits)}\cdot\log [X]\}&=K_1(3)-\sage{Lower(F2_LX*P0_LX,digits)}\times\log 3\phantom{xxxxxxxxx}\\
&\leq\sage{Upper(logmx,digits)},\\
\min_{X\in[1,2\cdot 10^8]} \bigg\{K_1([X])-\sage{Upper(uc1,digits)}\bigg(\log [X]-\log\bigg(\frac{[X]+1}{[X]}\bigg)\bigg)\bigg\}\\
= K_1(1)-\sage{Upper(uc1,digits)}\times(\log 1-\log 2)&\geq\sage{Lower(logmm,digits)}.
\end{align*}
\end{proofbold}
\newline


\begin{proofbold}[Corollary \ref{auxnu1}]
   Set $Y=X/10^{12}$.
  We readily find that
  \begin{equation*}
    \frac{1}{\log(X/\ell)}
    =
    \frac{1}{12\log10}
    -\int_{\ell}^Y\frac{dt}{t\log^2(X/t)}
  \end{equation*}
  so that 
  \begin{align*}
     L_1(X)       &=
          - \sum_{\ell\le Y}
     \frac{\mu^2(\ell)\varphi(\ell)}{\ell^2}\bigg[g_2(\ell)\frac{\ell^\xi}{\varphi_{\xi}(\ell)}\bigg]
\int_\ell^Y \frac{dt}{t\log^2(X/t)}dt
    +
    \frac{K_1(Y)}{12\log10}\\
    &=
           -\int_1^Y \frac{K_1(t)}{t\log^2(X/t)}dt
    +
    \frac{K_1(Y)}{12\log10},
  \end{align*}
  where the last inequality was obtained by summation by parts and $K_1(X)$ is defined in Lemma \ref{auxf2}. Thereupon,
  Lemma~\ref{auxf2} applies to give the bound 
  \begin{align*}
    L_1(X)
    \le -\int_1^Y \frac{c_1\log t +\sage{Lower(logmm,digits)}}{t\log^2(X/t)}dt
    +
    \frac{c_1\log Y + \sage{Upper(logmx,digits)}}{12\log10}\phantom{xxxxxxxxxxxxxxxxxxxxxx}&
    \\\le - \int_0^{\log Y} \frac{c_1 u +\sage{Lower(logmm,digits)}}{(12\log 10+ \log Y -u)^2}du
    +
    \frac{c_1\log Y + \sage{Upper(logmx,digits)}}{12\log10}\phantom{xxxxxxxxxxxxxx}&
    \\\le - \int_{12\log 10}^{12\log 10+\log Y} \frac{c_1
    (12\log10+\log Y-v) + \sage{Lower(logmm,digits)}}{v^2}dv
    +
     \frac{c_1\log Y + \sage{Upper(logmx,digits)}}{12\log10}\phantom{xxxxx}&
     \\= c_1\log\biggl(1+\frac{\log Y}{12\log 10}\biggr)    +\frac{c_1\log X + \sage{Lower(logmm,digits)}}{\log X}-\frac{c_1\log X + \sage{Lower(logmm,digits)}}{12\log 10}+
     \frac{c_1\log Y + \sage{Upper(logmx,digits)}}{12\log10}&\\
     \\\leq c_1\log\biggl(1+\frac{\log Y}{12\log 10}\biggr)    +c_1+\frac{\sage{Lower(logmm,digits)}}{12\log 10}-\frac{c_1\log X + \sage{Lower(logmm,digits)}}{12\log 10}+
     \frac{c_1\log Y + \sage{Upper(logmx,digits)}}{12\log10}&
     \\= c_1\log\biggl(1+\frac{\log Y}{12\log 10}\biggr) +
     \frac{ \sage{Upper(logmx,digits)}}{12\log10}.\phantom{xxxxxxxxxxxxxxxxxxxxxxxxxxxxxxx}&
      \end{align*}
The result is obtained by recalling that $c_1\leq\sage{Upper(uc1,digits)}$.
\end{proofbold}
\newline


\begin{proofbold}[Lemma \ref{auxf22}]
 Observe that, for any sufficiently large prime $p$,
\begin{equation*}
f(p)=\frac{\varphi(p)}{p^2}\bigg[g_2(p)\frac{p^\xi}{\varphi_{\xi}(p)}\bigg]^2=\frac{1}{p}+\frac{2p^{1+\xi}-p^{2\xi}-p}{p^2(p^\xi-1)^2}=\frac{1}{p}+O\bigg(\frac{1}{p^{1+\xi}}\bigg).
\end{equation*}
Hence, we may use Theorem \ref{general++}$\mathbf{[A]}$ with $q=1$, $(\alpha,\beta)=(1,1+\xi)$ with $\xi>1/2$  and obtain
 \begin{align*}
K_2(X)= F_1(X)+O^*\bigg( \frac{\mathrm{w}_1\ \mathrm{P}_1}{\sqrt{X}}\bigg),
 \end{align*} 
 where $\mathrm{w}_1$ is defined in the proof of Lemma \ref{auxf2} and
  \begin{equation*}
 F_1(X)=H_{f}(0)\bigg(\log X+\sum_{p}\frac{(1-(p-2)f(p))\log p}{(f(p)+1)(p-1)}+\gamma\bigg),
 \end{equation*}
 \begin{align*}
 H_{f}(0)&=\left(1-\frac{1-2f(2)}{2}-\frac{f(2)}{2}  \right)\prod_{p\geq 3}\bigg(1+\frac{2p^{2-\xi}-2p-p^{2-2\xi}+1}{p^{3-2\xi}(p^\xi-1)^2}\bigg)\\ 
 \mathrm{P}_1&=\bigg(1+\frac{|2f(2)-1|}{\sqrt{2}-1}\bigg)\prod_{p\geq 3}\bigg(1+\frac{2p^{\xi}-p^{2\xi-1}-1}{(p^\xi-1)^2(\sqrt{p}-1)}\bigg),
 \end{align*}    
Note that
 \begin{align*}
 1-\frac{1-2f(2)}{2}-\frac{f(2)}{2} \in(\sage{Lower(AF2_LX,digits)},\sage{Upper(AF2_LX,digits)}),&\quad 1+\frac{|2f(2)-1|}{\sqrt{2}-1}\in(\sage{Lower(AF2P1_LX,digits)},\sage{Upper(AF2P1_LX,digits)}),\\
 \prod_{p\geq 3}\bigg(1+\frac{2p^{2-\xi}-2p-p^{2-2\xi}+1}{p^{3-2\xi}(p^\xi-1)^2}\bigg)&\in(\sage{Lower(AP0_LX,digits)},\sage{Upper(AP0_LX,digits)}),\\
\prod_{p\geq 3}\bigg(1+\frac{2p^{\xi}-p^{2\xi-1}-1}{(p^\xi-1)^2(\sqrt{p}-1)}\bigg)&\in(\sage{Lower(AP1_LX,digits)},\sage{Upper(AP1_LX,digits)}).
 \end{align*}    
 Therefore
   \begin{align*}
& \prod_{p}\bigg(1+\frac{2p^{2-\xi}-2p-p^{2-2\xi}+1}{p^{3-2\xi}(p^\xi-1)^2}\bigg)\\&
\qquad\in(\sage{Lower(AF2_LX,digits)},\sage{Upper(AF2_LX,digits)})\times(\sage{Lower(AP0_LX,digits)},\sage{Upper(AP0_LX,digits)})\times\bigg[1,\exp\bigg(\frac{(2+M^{-(2-\xi)})\cdot B_{1+\xi}(M)}{(1-M^{-\xi})^2}\bigg)\bigg]\\
&\qquad\qquad\subset(\sage{Lower(AF2_LX*AP0_LX,digits)},\sage{Upper(AF2_LX*AP0_LX*exp((2+M^(xi-2))*(1-M^(-xi))^(-2)*B(1+xi,M)),digits)}),\\
&\prod_{p}\bigg(1+\frac{2p^{\xi}-p^{2\xi-1}-1}{(p^\xi-1)^2(\sqrt{p}-1)}\bigg)\\
&\qquad\in(\sage{Lower(AF2P1_LX,digits)},\sage{Upper(AF2P1_LX,digits)})\times(\sage{Lower(AP1_LX,digits)},\sage{Upper(AP1_LX,digits)})\times\bigg[1,\exp\bigg(\frac{2\cdot B_{1/2+\xi}(M)}{(1-M^{-\xi})^2(1-M^{-1/2})}\bigg)\bigg]\\
&\qquad\qquad\subset(\sage{Lower(AF2P1_LX*AP1_LX,digits)},\sage{Upper(AF2P1_LX*AP1_LX*exp(2*(1-M^(-xi))^(-2)*(1-M^(-1/2))^(-1)*B(1/2+xi,M)),digits)}).
 \end{align*}   
 We also have
 \begin{align*}
 \gamma+\sum_{p}\frac{(1-(p-2)f(p))\log p}{(f(p)+1)(p-1)}\phantom{xxxxxxxxxxxxxxxxxxxxxxxxxxxxxxxxx}&\\
 \qquad\qquad\in\gamma+\frac{\log 2}{f(2)+1}-\sum_{3\leq p\leq M}\frac{(2p^{2+\xi}-3p^{1+2\xi}-p^2+2p^{2\xi})\log p}{(p^{2+2\xi}-2p^{2+\xi}+p^{1+2\xi}+p^2-p^{2\xi})(p-1)}&\\
 +\bigg[-\frac{(2+2M^{\xi-2})\ C_{1+\xi}(M)}{(1-2M^{-\xi}-M^{-2})(1-M^{-1})},0\bigg]&\\
 \subset(\sage{Lower(ASumx-2*(1+M^(xi-2))/(1-2*M^(-xi)-M^(-2))/(1-M^(-1))*C(1+xi,M),digits)},\sage{Upper(
 ASumx,digits)}).\phantom{xxxxxxxxxxxxxxxxxxxxxxxxxxxxxxxxx}&
 \end{align*} 
 All in all, we conclude that 
  \begin{align*}
K_2(X)\in(&\sage{Lower(AF2_LX*AP0_LX,digits)},\sage{Upper(AF2_LX*AP0_LX*exp((2+M^(xi-2))*(1-M^(-xi))^(-2)*B(1+xi,M)),digits)})\cdot(\log X+(\sage{Lower(ASumx,digits)},\sage{Upper(ASumx,digits)}))+O^*\bigg( \frac{\sage{Upper(W2_lx,digits)}\times\sage{Upper(AF2P1_LX*AP1_LX*exp(2*(1-M^(-xi))^(-2)*(1-M^(-1/2))^(-1)*B(1/2+xi,M)),digits)}}{\sqrt{X}}\bigg)\phantom{xxxxx}\\
&\ \subset(\sage{Lower(AF2_LX*AP0_LX,digits)},\sage{Upper(AF2_LX*AP0_LX*exp((2+M^(xi-2))*(1-M^(-xi))^(-2)*B(1+xi,M)),digits)})\cdot\log X\\
+\bigg(\sage{Lower(AF2_LX*AP0_LX,digits)}&\times\sage{Lower(ASumx,digits)}-\frac{\sage{Upper(W2_lx,digits)}\times\sage{Upper(AF2P1_LX*AP1_LX*exp(2*(1-M^(-xi))^(-2)*(1-M^(-1/2))^(-1)*B(1/2+xi,M)),digits)}}{\sqrt{X}},\sage{Upper(AF2_LX*AP0_LX*exp((2+M^(xi-2))*(1-M^(-xi))^(-2)*B(1+xi,M)),digits)}\times\sage{Upper(ASumx,digits)}+\frac{\sage{Upper(W2_lx,digits)}\times\sage{Upper(AF2P1_LX*AP1_LX*exp(2*(1-M^(-xi))^(-2)*(1-M^(-1/2))^(-1)*B(1/2+xi,M)),digits)}}{\sqrt{X}}\bigg).
 \end{align*} 
 If we assume that $X\geq 2\cdot 10^8$, then
 \begin{align*}
K_2(X)&\in(\sage{Lower(AF2_LX*AP0_LX,digits)},\sage{Upper(AF2_LX*AP0_LX*exp((2+M^(xi-2))*(1-M^(-xi))^(-2)*B(1+xi,M)),digits)})\cdot\log X+\bigg(\sage{Lower(AF2_LX*AP0_LX*ASumx,digits)}-\frac{\sage{Upper(W2_lx*AF2P1_LX*AP1_LX*exp(2*(1-M^(-xi))^(-2)*(1-M^(-1/2))^(-1)*B(1/2+xi,M)),digits)}}{\sqrt{2\cdot 10^8}},\sage{Upper(AF2_LX*AP0_LX*exp((2+M^(xi-2))*(1-M^(-xi))^(-2)*B(1+xi,M))*ASumx,digits)}+\frac{\sage{Upper(W2_lx*AF2P1_LX*AP1_LX*exp(2*(1-M^(-xi))^(-2)*(1-M^(-1/2))^(-1)*B(1/2+xi,M)),digits)}}{\sqrt{2\cdot 10^8}}\bigg)\\
&\subset(\sage{Lower(AF2_LX*AP0_LX,digits)},\sage{Upper(AF2_LX*AP0_LX*exp((2+M^(xi-2))*(1-M^(-xi))^(-2)*B(1+xi,M)),digits)})\cdot\log X+(\sage{Lower(AF2_LX*AP0_LX*ASumx-W2_lx*AF2P1_LX*AP1_LX*exp(2*(1-M^(-xi))^(-2)*(1-M^(-1/2))^(-1)*B(1/2+xi,M))/sqrt(N),digits)},\sage{Upper(AF2_LX*AP0_LX*exp((2+M^(xi-2))*(1-M^(-xi))^(-2)*B(1+xi,M))*ASumx+W2_lx*AF2P1_LX*AP1_LX*exp(2*(1-M^(-xi))^(-2)*(1-M^(-1/2))^(-1)*B(1/2+xi,M))/sqrt(N),digits)}),
 \end{align*}  
 and we complete the proof by recalling \eqref{analysis}. Indeed, for any $X\geq 1$, 
 \begin{equation*}
 K_2([X])-\sage{Upper(AF2_LX*AP0_LX*exp((2+M^(xi-2))*(1-M^(-xi))^(-2)*B(1+xi,M)),digits)}\cdot\log [X]\leq K_2([X])-c_2\log [X]\leq K_2([X])-\sage{Lower(AF2_LX*AP0_LX,digits)}\cdot\log [X];
 \end{equation*}	
 thus for any $X\geq 1$, we calculate and observe
\begin{align*} \max_{X\in[1,2\cdot 10^8]} \{K_2([X])-\sage{Lower(lc2,digits)}\cdot\log [X]\}&=K_2(3)-\sage{Lower(lc2,digits)}\log 3\\
&\leq\sage{Upper(log2mx,digits)},\\
\min_{X\in[1,2\cdot 10^8]} \bigg\{K_2([X])-\sage{Upper(uc2,digits)}\bigg(\log [X]-\log\bigg(\frac{[X]+1}{[X]}\bigg)\bigg)\bigg\}\\
= K_2(1)-\sage{Upper(uc2,digits)}\times(\log 1-\log 2)&\geq\sage{Lower(log2mm,digits)}.
\end{align*}
 \end{proofbold}
 \newline
 

\begin{proofbold}[Corollary \ref{auxnu}]  
  Set $Y=X/10^{12}$.
  We readily find that
  \begin{equation*}
    \frac{1}{\log^2(X/\ell)}
    =
    \frac{1}{(12\log10)^2}
    -\int_{\ell}^Y\frac{2\ dt}{t\log^3(X/t)}
  \end{equation*}
  so that, similarly to the observation of Corollary \ref{auxnu1},
  \begin{equation*}
   L_2(X)
       =
    -\int_1^Y \frac{2K_2(t)}{t\log^3(X/t)}dt
    +
    \frac{K_2(Y)}{(12\log10)^2}
  \end{equation*}
  where $K_2(X)$ is defined in Lemma \ref{auxf22}. Thus, Lemma~\ref{auxf22} applies to give the bound
  \begin{align*}
   L_2(X)
    \le -2\int_1^Y \frac{c_2\log t + \sage{Lower(log2mm,digits)}}{t\log^3(X/t)}dt
    +
    \frac{c_2\log Y + \sage{Upper(log2mx,digits)}}{(12\log10)^2}\phantom{xxxxxxxxxxxxxxxxxxxxx}&
    \\\le -2 \int_{12\log 10}^{12\log 10+\log Y} \frac{c_2
    (12\log10+\log Y -v) + \sage{Lower(log2mm,digits)}}{v^3}dv
    +
    \frac{c_2\log Y + \sage{Upper(log2mx,digits)}}{(12\log10)^2}&\\
    \\= \frac{c_2\log X + \sage{Lower(log2mm,digits)}}{\log^2X}-\frac{c_2\log X + \sage{Lower(log2mm,digits)}}{(12\log 10)^2}-\frac{2c_2}{\log X}    +\frac{2c_2}{12\log 10}+
    \frac{c_2\log Y + \sage{Upper(log2mx,digits)}}{(12\log10)^2}&
    \\\leq\frac{c_2}{\log X}-\frac{c_2\log X}{(12\log 10)^2}-\frac{2c_2}{\log X}    +\frac{2c_2}{12\log 10}+
    \frac{c_2\log Y + \sage{Upper(log2mx,digits)}}{(12\log10)^2}\phantom{xxxxxxxxxx}&
  \\=-\frac{c_2}{\log X}   +\frac{c_2}{12\log 10}+
    \frac{ \sage{Upper(log2mx,digits)}}{(12\log10)^2}\leq
    \frac{c_2}{12\log 10}+
    \frac{ \sage{Upper(log2mx,digits)}}{(12\log10)^2}.\phantom{xxxxxxxxxx}&
  \end{align*}
The result is concluded by recalling that $c_2\leq\sage{Upper(uc2,digits)}$.
\end{proofbold}

\bibliographystyle{plain}
\bibliography{Local}

@article {Breteche-Dress-Tenenbaum*20,
    AUTHOR = {de la Bret\`eche, R. and Dress, F. and
              Tenenbaum, G.},
     TITLE = {Remarques sur une somme li\'{e}e \`a la fonction de
              {M}\"{o}bius},
   JOURNAL = {Mathematika},
  FJOURNAL = {Mathematika. A Journal of Pure and Applied Mathematics},
    VOLUME = {66},
      YEAR = {2020},
    NUMBER = {2},
     PAGES = {416--421},
      ISSN = {0025-5793,2041-7942},
   MRCLASS = {11N37 (11A25 11N56 11N64)},
  MRNUMBER = {4130331},
MRREVIEWER = {Y.-F.\ S.\ P\'{e}termann},
       DOI = {10.1112/mtk.12021},
}

@Article{Dress-Iwaniec-Tenenbaum*83,
  Title                    = {Sur une somme li\'ee \`a la fonction de {M}\"obius},
  Author                   = {Dress, F. and Iwaniec, H. and Tenenbaum, G.},
  Journal                  = {J. Reine Angew. Math.},
  Year                     = {1983},
  Pages                    = {53-58},
  Volume                   = {340},
  Coden                    = {JRMAA8},
  Fjournal                 = {Journal f\"ur die Reine und Angewandte Mathematik},
  ISSN                     = {0075-4102},
  Mrclass                  = {10H25 (10A20)},
  Mrreviewer               = {Christian Radoux}
}

@article {Dusart*18,
    AUTHOR = {Dusart, P.},
     TITLE = {Explicit estimates of some functions over primes},
   JOURNAL = {Ramanujan J.},
    VOLUME = {45},
      YEAR = {2018},
     PAGES = {227--251},
}

@Article{Granville-Ramare*96,
  Title                    = {Explicit bounds on exponential sums and the scarcity of squarefree binomial coefficients},
  Author                   = {Granville, A. and Ramar\'e, O.},
  Journal                  = {Mathematika},
  Year                     = {1996},
  Number                   = {1},
  Pages                    = {73-107},
  Volume                   = {43}
}

@Book{Helfgott*30,
  author = 	 {Helfgott, H.A.},
  ALTeditor = 	 {},
  title = 	 {The ternary Goldbach conjecture},
  publisher = 	 {Book accepted for publication in Ann. of Math. Studies},
  year = 	 {2019},
  OPTkey = 	 {},
  OPTvolume = 	 {},
  OPTnumber = 	 {},
  OPTseries = 	 {},
  OPTaddress = 	 {},
  OPTedition = 	 {},
  OPTmonth = 	 {},
  OPTnote = 	 {},
  URL={https://webusers.imj-prg.fr/~harald.helfgott/anglais/book.html (version 09/2019)},
  OPTannote = 	 {}
}

@article {Johnston*22,
    AUTHOR = {Johnston, D.R.},
     TITLE = {Improving bounds on prime counting functions by partial
              verification of the {R}iemann hypothesis},
   JOURNAL = {Ramanujan J.},
  FJOURNAL = {Ramanujan Journal. An International Journal Devoted to the
              Areas of Mathematics Influenced by Ramanujan},
    VOLUME = {59},
      YEAR = {2022},
    NUMBER = {4},
     PAGES = {1307--1321},
      ISSN = {1382-4090},
   MRCLASS = {11Y70 (11M26 11N05)},
  MRNUMBER = {4507211},
       DOI = {10.1007/s11139-022-00616-x},

}

@Article{Ramare*95,
  Title                    = {On {S}nirel'man's constant},
  Author                   = {Ramar\'e, O.},
  Journal                  = {Ann. Scu. Norm. Pisa},
  Year                     = {1995},
  Pages                    = {645-706},
  Volume                   = {21}
}

@Article{Ramare*12-4,
  Title                    = {Some elementary explicit bounds for two mollifications of the {M}oebius function},
  Author                   = {Ramar\'e, O.},
  Journal                  = {Functiones et Approximatio},
  Year                     = {2013},
  Number                   = {2},
  Pages                    = {229--240},
  Volume                   = {49}
}

@Article{Ramare*13d,
  Title                    = {An explicit density estimate for {D}irichlet ${L}$-series},
  Author                   = {Ramar\'e, O.},
  Journal                  = {Math. Comp.},
  Year                     = {2016},
  Volume                   = {85},
  Number                   = {297},
  Pages                    = {335-356}
}

@Article{Ramare-Zuniga*23-1,
  author = 	 {Ramar\'e, O. and Zuniga-Alterman, S.},
  title = 	 {M\"obius function and primes: an identity factory with applications},
  journal = 	 {Math. Comp. (to appear)},
  year = 	 {2026},
  OPTkey = 	 {},
  OPTvolume = 	 {},
  OPTnumber = 	 {},
  OPTmonth = 	 {},
  OPTnote = 	 {},
  OPTannote = 	 {}
}

@Article{Ramare-Zuniga*25-1,
  author = 	 {Ramar\'e, O. and Zuniga-Alterman, S.},
  title = 	 {From explicit estimates for the primes to explicit estimates for the Möbius function II},
  journal = 	 {Funct. Approx. Comment. Math. Advance Publication},
  year = 	 {2025},
  OPTkey = 	 {},
  OPTvolume = 	 {1-19},
  OPTnumber = 	 {},
  OPTmonth = 	 {},
  OPTnote = 	 {},
  OPTannote = 	 {}
}

@article {Alterman*22,
    AUTHOR = {Zuniga-Alterman, S.},
     TITLE = {Explicit averages of square-free supported functions: to the
              edge of the convolution method},
   JOURNAL = {Colloq. Math.},
  FJOURNAL = {Colloquium Mathematicum},
    VOLUME = {168},
      YEAR = {2022},
    NUMBER = {1},
     PAGES = {1--23},
      ISSN = {0010-1354,1730-6302},
   MRCLASS = {11N37},
  MRNUMBER = {4378559},
MRREVIEWER = {Olivier\ Bordell\`es},
       DOI = {10.4064/cm8337-11-2020},
}

\Addresses

\end{document}